\documentclass[a4paper, reqno]{amsart}
\usepackage[margin=2.5cm]{geometry}                
\usepackage{amssymb,amsmath, mathtools, mathabx}  
\usepackage[colorlinks,
             linkcolor=blue,
             citecolor=green!70!black,
             pdfproducer={pdfLaTeX},
             pdfpagemode=None,
             bookmarksopen=true
             bookmarksnumbered=true]{hyperref}
\usepackage{cleveref}
\usepackage{hyperref}
\usepackage[mathscr]{euscript}
\usepackage[all]{xy}

\DeclareMathAlphabet{\mathpzc}{OT1}{pzc}{m}{it}

\usepackage[utf8]{inputenc}
\usepackage{tikz}
\usetikzlibrary{matrix,arrows,decorations.pathmorphing,decorations.pathreplacing,decorations.markings}
\tikzset{
    dot/.style={circle,draw,fill,inner sep=1pt},
    arrow/.style={->,thick,shorten <=2pt,shorten >=2pt},
    every label/.append style = {font = \small}
    }
\usepackage{tikz-cd}
\newtheorem{theorem}{Theorem}[section]
\newtheorem*{theorem*}{Theorem}
\newtheorem{prop}[theorem]{Proposition}
\newtheorem*{prop*}{Proposition}
\newtheorem{cor}[theorem]{Corollary}
\newtheorem{lemma}[theorem]{Lemma}

\theoremstyle{definition}
\newtheorem{defn}[theorem]{Definition}

\newtheorem*{defn*}{Definition}
\newtheorem*{const*}{Construction}
\newtheorem*{warning*}{Warning}
\newtheorem{ex}[theorem]{Example}

\newtheorem{remark}[theorem]{Remark}

\newtheorem{notation}[theorem]{Notation}
\newtheorem{construction}[theorem]{Construction}

\usepackage[utf8]{inputenc}
\DeclareFontFamily{U}{min}{}
\DeclareFontShape{U}{min}{m}{n}{<-> udmj30}{}

\usepackage{bbm}

\newcommand\cA{\mathscr A} 
\newcommand\cB{\mathscr B} 
\newcommand\cC{\mathscr C} 
\newcommand\cD{\mathscr D}
\newcommand\cE{\mathscr E} 
\newcommand\cF{\mathscr F}
\newcommand\cG{\mathscr G}
\newcommand\cI{\mathscr I}
\newcommand\cK{\mathscr K}
\newcommand\cL{\mathscr L}

\newcommand\cO{\mathscr O}
\newcommand\cP{\mathscr P}
\newcommand\cS{\mathscr S}

\newcommand\cT{\mathscr T}
\newcommand\cU{\mathscr U}
\newcommand\cV{\mathscr V}
\newcommand\cW{\mathscr W}

\newcommand\rB{\mathrm B}

\newcommand\NN{\mathbb N} \newcommand\bN\NN

\newcommand\mc{\mathrm{mc}}

\newcommand\edm{\mathrm{End}}
\newcommand\good{\mathrm{good}}
\newcommand\art{\mathrm{art}}

\newcommand\id{\mathrm{id}}

\newcommand\mor{\mathrm{Hom}}
\newcommand\nat{\mathrm{Nat}}
\newcommand\ima{\mathrm{Im}}
\newcommand\inj{\mathrm{inj}}
\newcommand\op{\mathrm{op}}
\newcommand\cat{\mathrm{Cat}}
\newcommand\comp{\mathrm{comp}}
\newcommand\lcomp{\mathrm{lcomp}}
\newcommand\ccat{\mathpzc{Cat}}
\newcommand\arr{\mathrm{Ar}}

\newcommand\thr{\mathrm{Th}}
\newcommand\spanc{\mathrm{Span}}
\newcommand\colim{\mathrm{colim}}
\newcommand\bydef{\overset{\mathrm{def}}{=}}

\newcommand\corr{\mathrm{Corr}}
\newcommand\cell{\mathrm{Cell}}

\newcommand\modl{\mathrm{Mod}}
\newcommand\fib{\mathrm{Fib}}
\newcommand\bo{\mathrm{i.o.}}
\newcommand\li{\mathrm{l.i.}}
\newcommand\cart{\mathrm{coCart}}
\newcommand\ccart{\mathpzc{coCart}}
\newcommand\lax{\mathrm{lax}}
\newcommand\con{\mathrm{Cone}}
\newcommand\maxm{\mathrm{max}}
\newcommand\un{\mathrm{unit}}
\newcommand\nun{\mathrm{nun}}
\newcommand\bimod{\mathrm{Bimod}}
\newcommand\bn{\langle n\rangle}
\newcommand\bm{\langle m\rangle}
\newcommand\bl{\langle l\rangle}
\newcommand\wrr{{\overset{\sim}{\rightarrow}}}
\newcommand\twar{\mathrm{TwAr}}
\newcommand\ladj{\mathrm{ladj}}
\newcommand\radj{\mathrm{radj}}

\newcommand\eso{\mathrm{eso}}

\newcommand\act{\mathrm{act}}
\newcommand\sur{\mathrm{surj}}
\newcommand\Act{\mathrm{Act}}
\newcommand\inrt{\mathrm{int}}
\newcommand\el{\mathrm{el}}
\newcommand\ob{\mathrm{Ob}}
\newcommand\seg{\mathrm{Seg}}

\newcommand\coc{\mathrm{coCone}}

\newcommand{\mElo}[2]{\mathpzc{#1}^{\mathrm{el}}_{#2/}}

\title{completeness for monads and theories}
\author{Roman Kositsyn}
\date{April 2021}

\begin{document}

\maketitle
\begin{abstract}
    We generalize the correspondence between theories and monads with arities of \cite{berger2012monads} to $\infty$-categories. Additionally, we introduce the notion of complete theories that is unique to the $\infty$-categorical case and provide a completion construction for a certain class of theories. Along the way we also develop the necessary technical material related to the flagged bicategory of correspondences and lax functor in the $\infty$-categorical context.
\end{abstract}

\tableofcontents

\section{Introduction}\label{sect:one}
The principal goal of this paper is to generalize the results of \cite{berger2012monads} to the $\infty$-categorical case, so we start this introduction with a brief review of the objects of study and results of \cite{berger2012monads}. That paper is concerned with providing an isomorphism between two apparently quite distinct categories - the category of \textit{nervous theories} and the category of \textit{monads with arities}. A \textit{category with arities} is a presentable category $\cE$ together with a full subcategory $i_\cA:\cA\hookrightarrow\cE$ such that the corresponding nerve functor $\nu_\cA:\cE\rightarrow\cP(\cA)$ is fully faithful, where $\nu_\cA(X)(A)\bydef\mor_\cE(i_\cA(A),X)$. A functor $F:\cE\rightarrow\cF$ betwen categories with arities $(\cE,\cA)$ and $(\cF,\cB)$ is called \textit{arity-respecting} if it takes $\cA$-colimits to colimits in $\cF$. A monad with arities is then simply a monad in the bicategory of categories and functors with arities (and arbitrary natural transformations). A nervous theory is a bijective-on-objects functor $j:\cA\rightarrow\Theta$ such that $j^*j_!$ restricts to an endofunctor of $\cE\subset\cP(\cA)$. The main result of \cite{berger2012monads} provides an equivalence between the category of nervous theories and the category of monads with arities. A common source of these objects are cartesian monads on presheaf categories (studied in \cite{weber2004generic} for ordinary categories and in \cite{chu2019homotopy} in the $\infty$-categorical contexts) which have canonical arities, however the monad/theory correspondence extends to more general cases that include Lawvere theories \cite{lawvere1963functorial} and the free groupoid monad described in the last section of \cite{berger2012monads}.\par
We generalize \cite{berger2012monads} in three principal directions. Firstly, we extend it to the general $\infty$-categorical context. This is not as straightforward as one might think - already the notion of a nervous theory presents some issues. More specifically, the problem is with the notion of a bijective-on-objects functor. The reason they appear in \cite{berger2012monads} is essentially because there is an equivalence between monads in the category of correspondences with the underlying object given by a category $C$ and bijective-on-objects functors $F:C\rightarrow D$. Our model of choice for $\infty$-categories are complete Segal spaces. With some care one can prove this equivalence for an arbitrary $\infty$-category $\cC$ (which is one of the results of \cref{sect:five}), however the corresponding Segal space $\cD$ will generally not be complete. Addressing this issue behooves us to define the category of correspondences in the setting of Segal spaces and prove some basic results about it, which we do in \cref{sect:two} and \cref{sect:three}.\par
Another way in which we generalize \cite{berger2012monads} is our use of lax functors. Lax functors were introduced in \cite{street1972two} and they roughly correspond to functors that respect compositions and identities only up to a non-invertible 2-morphism. In particular, a lax functor from a point to a bicategory $B$ is the same thing as a monad in $B$. Since monads in the bicategory of categories with arities correspond to nervous theories, one may inquire about the class of objects corresponding to lax functors from a more general category $\cC$, and we give an answer to this question in the form of the notion of a $\cC$-\textit{theory} introduced in \cref{sect:six}.\par
Our final contribution is the notion of \textit{completeness} of a $\cC$-theory. This notion is unique to the $\infty$-categorical setting and it generalizes the notion of completeness of ordinary Segal spaces. We adapted it from \cite{chu2019homotopy} in which it was introduced in a rather general context of cartesian monads and algebraic patterns. Namely, the main result of \cite{chu2019homotopy} provides an equivalence between saturated algebraic patterns and complete cartesian monads on presheaf categories and also posits that the inclusion of the category of complete cartesian monads into the category of all cartesian monads admits a left adjoint called \textit{completion} and that the categories of algebras for a monad and for its completion are isomorphic. Although The notion of completeness makes sense for a general $\cC$-theory $\cT$, we are only able to provide the completion construction for a rather restrictive class of theories which we call \textit{good} (those include the theories associated to cartesian monads on presheaf categories). We now describe the contents of each section in more detail.\par
In \cref{sect:two} we construct the bicategory of correspondences $\corr$ (which are also sometimes called profunctors, distributors and bimodules). This bicategory was originally introduced in \cite{benabou1973distributeurs} and numerous accounts of it have been written since. The $\infty$-categorical version appears in \cite[Section 2.3.1.]{lurie2009higher}, \cite{ayala2017fibrations}, \cite{haugseng2019segal} and \cite{haugseng2015bimodules}, our only novelty here is a consistent use of the formalism of algebraic patterns. Near the end of the section we also introduce an explicit description of the full subcategory $\spanc$ of $\corr$ on constant Segal spaces.\par
In \cref{sect:three} we prove certain basic properties of the bicategory of correspondences. All of those properties can be found for example in \cite{street1980fibrations} for ordinary categories, however we could not find appropriate references for all of them in the $\infty$-categorical context. Our main result here is the combination of \cref{prop:fun_comp} and \cref{prop:corr_bicat} that state that the bicategory of Segal spaces, functors and natural transformations $\cat$ as well as its opposite $\cat^\op$ can be seen as a faithful subcategory of $\corr$ through a functor that sends every category to itself and every functor $F$ to the correspondence $F_!$ (resp. the correspondence $F^*$). Another important result in this section is \cref{cor:corr_comp} that provides an isomorphism between the category of correspondences between $\cC$ and $\cD$ and the category of correspondences between their completions. In particular, this allows us to connect our notion of correspondences to the more classical one used in \cite{lurie2009higher} and \cite{ayala2017fibrations}.\par
In \cref{sect:four} we introduce the notion of a lax functor from a Segal space $\cC$ to a twofold Segal space $\cB$. In defining the lax functors we follow the approach of \cite{gaitsgory2017study}, namely we view $\cC$ and $\cB$ as coCartesian fibrations over $\Delta^\op$ and define a (unital) lax functor to be a morphism $\cC\rightarrow\cB$ of categories over $\Delta^\op$ that takes coCartesian morphisms lying over inert (resp. over inert and surjective) morphisms to coCartesian morphisms. To make working with this notion more convenient, we introduce for an $\infty$-category $\cC$ with a subcategory $\cD$ such that $\cC_0\cong\cD_0$ an algebraic pattern $\ccart^\cD(\cC)$ whose category of Segal spaces has objects $(\cE\rightarrow\cC)\in\cat_{/\cC}$ that admit coCartesian liftings for all morphisms $f\in\cD$ and morphisms given by functors over $\cC$ that preserve those liftings. We then use this to prove that the categories of lax and unital lax functors from a given $\cC$ are corepresentable by a twofold Segal spaces $L^\lax\cC$ and $L^{\lax,\un}\cC$ respectively. Near the end of the section we introduce the notion of lax limits and lax colimits. The classical reference for this notion is \cite{street1976limits}, the notion was also introduced in \cite{gepner2015lax} in the special case of the bicategory $\cat^\comp$ and more recently in \cite{gagna2020fibrations} for a general bicategory $\cB$, however we make no effort to connect those definitions to ours. \par
\Cref{sect:five} is probably the most technical part of the paper. In it we first introduce for a given Segal space $\cC$ the category $\cat^\wrr_{/\cC}$ whose objects are morphisms $(\cD\xrightarrow{f}\cE)$ in $\cat_{/\cC}$ such that for every $c\in\cC$ we $f_c$ induces an isomorphism $(\cD_c)_0\xrightarrow{\sim}(\cE_c)_0$ and such that for every morphism $f:c\rightarrow c'$ in $\cC$ we also have $\cD_f\xrightarrow{\sim}\cE_f$. Observe that in the special case $\cC\cong*$ this category is simply the category of isomorphism-on-objects functors $F:\cD\rightarrow\cE$. The first main result of this section identifies $\cat^\wrr_{/\cC}$ with a subcategory $\mor_\cat^{\lax,!}(\cC,\corr)$ of the category of lax functors from $\cC$ to $\corr$ having the same objects but only those natural transformations $\alpha$ for which each component $\alpha_c$ is a correspondence of the form $f_{c,!}$. In particular, it follows that to every lax functor $F:\cC\rightsquigarrow\corr$ we can attach the Segal space $\cE$. Our second main result describes $\cE$ as a lax colimit of the functor $F$. At the end of the section we discuss the notion of completeness. We define an object $(\cD\xrightarrow{f}\cE)$ of $\cat^\wrr_{/\cC}$ \textit{left-complete} if $\cD$ is a complete Segal space and \textit{complete} if both $\cD$ and $\cE$ are complete. The fact that the inclusion $\cat^{\lcomp,\wrr}_{/\cC}\hookrightarrow\cat^\wrr_{/\cC}$ of the full subcategory on left-complete objects admits a left adjoint is relatively elementary, however the fact that $\cat^{\comp,\wrr}_{/\cC}\hookrightarrow\cat^{\lcomp,\wrr}_{/\cC}$ also admits a left adjoint $L^\bo$ is much more difficult and is a subject of \cref{cor:adj_comp}.\par
In the final \cref{sect:six} we introduce the main objects of interest of this paper - the category of $\cC$-theories for a complete Segal space $\cC$ and the category of correspondences with arities $\corr^\art$. We must note here that our setup is slightly different from \cite{berger2012monads}. First of all, we do not require the category $\cE$ to be presentable and instead consider an arbitrary full subcategory $\cE_\cC\hookrightarrow\mor_\cat(\cC,\cS)$. Moreover, the category of morphisms from $(\cE_\cC,\cC)$ to $(\cE_\cD,\cD)$ is given by a full subcategory of $\mor_\corr(\cC,\cD)$ that take $\cE_\cC$ to $\cE_\cD$. This might look different from \cite{berger2012monads}, however it is in fact not: observe that if we are given a categories with arities $(\cE,\cA)$ and $(\cF,\cB)$ in the sense of \cite{berger2012monads} then any arity-respecting functor $F:\cE\rightarrow\cF$ defines a colimit preserving functor $\cA\overset{i_\cA}{\hookrightarrow}\cE\xrightarrow{F}\cF\xrightarrow{\nu_\cB}\cP(\cB)$ which can be seen a a correspondence from $\cA$ to $\cB$, and moreover the functor $F$ is uniquely defined by this correspondence. We also define a $\cC$-theory to be an object $(\cT_0\xrightarrow{t}\cT_1)$ of $\cat^{\lcomp,\wrr}_{/\cC}$ for which each $\cT_{0,c}$ for all $c\in\cC$ is a category with arities, each correspondence $\cT_f$ for a morphism $f$ in $\cC$ respects arities and also each monad $t^*_c t_{c,!}$ respects arities. We also define the category of \textit{models} by means of the following pullback square
\[
\begin{tikzcd}[row sep=huge, column sep=huge]
\modl_\cT(\cS)\arrow[r]\arrow[d]&\mor_\cat(\cT_1,\cS)\arrow[d]\\
\coprod_{c\in\cC}\cE_c\arrow[r, "\coprod_{c\in\cC} i_c"]&\coprod_{c\in\cC}\mor_\cat(\cT_{0,c},\cS)
\end{tikzcd}.
\]
The main result of this section is \cref{thm:art} that describes an equivalence between the category of $\cC$-theories and the category of lax functors $F:\cC\rightsquigarrow\corr^\art$ and natural transformations of the form $(f_c)_!$. It moreover defines an isomorphism between $\modl_\cT(\cS)$ and $\coc^\lax_{\corr^\art}(*,F_\cT)$. In the special case $\cC\cong*$ this reduces to the isomorphism between the algebras over a monad with arities and the category of models for the corresponding theory. Finally, we consider the notion of completeness for $\cC$-theories. A $\cC$-theory is called complete if the underlying object of $\ccat^\wrr_{/\cC}$ is complete. Unlike in the case of algebraic patterns of \cite{chu2019homotopy}, the inclusion $\thr^\comp(\cC)\hookrightarrow\thr(\cC)$ does not admit a left adjoint. However \cref{thm:comp} states that the inclusion $\thr^\comp(\cC)\hookrightarrow\thr^\good(\cC)$ does, where $\thr^\good(\cC)$ denotes the subcategory of good $\cC$-theories described in \cref{def:good}. The $*$-theories attached to algebraic patterns are good, and it follows from the universal properties that our completion construction corresponds to the completion construction of \cite{chu2019homotopy}.
\paragraph{\textbf{Notations and conventions}:} Our standard reference for $\infty$-categories is \cite{lurie2009higher}, however we will also use the equivalence between quasicategories and complete Segal spaces (see e.g. \cite{joyal2007quasi}) without further mention. We will typically call $\infty$-categories simply categories except in cases where it may lead to confusion. We denote by $\cS$ the category of spaces and by $\cP(\cC)$ the category of presheaves on a category $\cC$. We assume the reader is familiar with the terminology and results of \cite{chu2019homotopy}. For an algebraic pattern $\cO$ we will typically denote $\mathrm{O}$ the category $\seg_\cO(\cS)$ of Segal $\cO$-spaces. In particular, we will denote by $\ccat$ the algebraic pattern of Segal spaces (i.e. $\Delta^\op$ with the usual active/inert factorization and with elementary objects given by $[0]$ and $[1]$) and by $\cat$ the category of Segal spaces. The category of complete Segal spaces will instead be denoted $\cat^\comp$. We will use the terms "twofold Segal" space and "flagged bicategory" interchangeably (the latter being the terminology introduced in \cite{ayala2018flagged}).
\paragraph{\textbf{Acknowledgements}} I am grateful to Artem Prikhodko and Grigory Kondyrev for reading the manuscript and for their helpful comments.
\section{Definition of Corr}\label{sect:two}
In this section we present the construction of the flagged bicategory $\corr$ whose objects are Segal spaces and whose morphisms are given by correspondences. Our main result is thus \cref{prop:corr} which explicitly constructs (the underlying Segal space of) $\corr$. We then continue to give an explicit description of the bicategory $\spanc$ which we prove in \cref{prop:span_subcat} is a full subcategory of $\corr$ on constant Segal spaces. 
\begin{prop}\label{prop:slice}
Consider an algebraic pattern $\cO$ together with a Segal $\cO$-space $X$ which we view as a presheaf on $\cO^\op$. Then there is an algebraic pattern $\cO_{/X}$ which is extendable (resp. saturated) if $\cO$ is, such that Segal spaces for $\cO_{/X}$ correspond to Segal spaces for $\cO$ endowed with a morphism to X. Any morphism $f:X\rightarrow Y$ of Segal spaces induces an extendable Segal morphism $f:\cO_{X/}\rightarrow\cO_{Y/}$.
\end{prop}
\begin{proof}
Define $\cO_{/X}$ to be a Segal fibration associated to the functor
\[(X:\cO\rightarrow\cS):i\mapsto X(I)\cong\underset{E\in\cO^\el_{I/}}{\lim}X(E),\]
where the last isomorphism follows since $X$ is assumed to be a Segal $\cO$-space. It then follows from \cite[Corollary 9.17]{chu2019homotopy} that it admits a structure of an extendable algebraic pattern if $\cO$ is extendable.\par
To prove the statement about Segal sheaves on $\cO_{/X}$, recall from \cite[Lemma 4.1]{ayala2017fibrations} that for a category $\cC$ and $p\in\cP(\cC)$ we have
\[\cP(\cC_{/P})\cong\cP(\cC)_{/P}.\]
Applying it to our case we see that every Segal $\cO_{/X}$-space $F$ automatically admits a morphism to $X$. \par
Now we need to prove that a morphism $(\cF\xrightarrow{f} X)$ viewed as a presheaf on $(\cO_{/X})^\op$ is a Segal space if and only if $\cF$ is a Segal space for $\cO$. First, assume that $\cF$ is a Segal $\cO$-space and fix an object $c\in\cO$ and $x\in X(c)$. Then since by definition $(\cO_{/X})^\el_{x/}\cong\cO^\el_{c/}$ we need to prove that the fiber $\cF_x$ of $f$ over $x$ is also a Segal $\cO$-space. To do this consider the following pullback square
\[
\begin{tikzcd}[row sep=huge, column sep=huge]
\cF_x\arrow[d]\arrow[r]&\cF(c)\cong\underset{(c\rightarrowtail e)\in\cO^\el_{c/}}{\lim}\cF(e)\arrow[d,"f"]\\
\underset{(c\rightarrowtail e)\in\cO^\el_{c/}}{\lim}*\cong*\arrow[r,"x"]&X(c)\cong\underset{(c\rightarrowtail e)\in\cO^\el_{c/}}{\lim}X(e)
\end{tikzcd}.
\]
The claim now follows from the commutativity of limits.\par
Assume conversely that $f$ is a Segal $\cO_{/X}$-space. To prove that $\cF$ is then a Segal $\cO$-space it would suffice to prove the commutativity of the following diagram
\[
\begin{tikzcd}[row sep=huge, column sep=huge]
\cP((\cO_{/X})^{\el,\op})\arrow[d,"p^\el_!"]\arrow[r,"i_{X,*}"]&\cP((\cO_{/X})^{\inrt,\op})\arrow[d,"p^\inrt_!"]\\
\cP(\cO^{\el,\op})\arrow[r,"i_*"]&\cP(\cO^{\inrt,\op})
\end{tikzcd}.
\]
Since the left vertical maps is a right fibrations, we have by \cite[Corollary 7.17]{chu2019homotopy} the following commutative diagram
\[
\begin{tikzcd}[row sep=huge]
{}&\cP(i^*i_*(\cO_{/X})^{\el,\op})\cong\cP((\cO_{/X})^{\el,\op})\arrow[ld, "\sim" swap]\arrow[r,"i_{X,*}"]&\cP(i_*(\cO_{/X})^{\el,\op})\cong\cP((\cO_{/X})^{\inrt,\op})\arrow[dd,"p^\inrt_!"]\\
\cP((\cO_{/X})^{\el,\op})\arrow[d,"p^\el_!"]\\
\cP(\cO^{\el,\op})\arrow[rr,"i_*"]&{}&\cP(\cO^{\inrt,\op})
\end{tikzcd}
\]
where the isomorphism $i^*i_*(\cO_{/X})^{\el}\cong(\cO_{/X})^{\el}$ follows since $i$ is a fully faithful inclusion and the isomorphism $i_*(\cO_{/X})^{\el,\op}\cong(\cO_{/X})^{\inrt}$ follows since $X$ is a Segal space.\par
Finally, given a morphism $f:X\rightarrow Y$ we have a natural morphism $f_!:\cO_{/X}\rightarrow\cO_{/Y}$ given by sending $(h_I\rightarrow X)\in\mor_\cO(h_i,X)\cong X(I)$ to $h_I\rightarrow X\xrightarrow{f}Y$. The fact that $f_!$ is extendable once again follows from \cite[Proposition 9.5]{chu2019homotopy}. To prove that it is Segal simply observe that the natural morphism $\mElo{(\cO_{/X})}{(h_I\rightarrow X)}\rightarrow\mElo{(\cO_{/Y})}{f_!(h_I\rightarrow X)}$ is an isomorphism.
\end{proof}
\begin{cor}\label{cor:slice}
For a saturated pattern $\cO$ and a Segal sheaf $X$ there is a natural equivalence 
\[\cO_{/X}\cong\underset{(o\rightarrow X)\in\cO_{/X}}{\lim}\cO_{/o}.\]
\end{cor}
\begin{proof}
By Yoneda lemma $X\cong \underset{(o\rightarrow X)\in\cO^\op_{/X}}{\colim}\:h_{o}$, so since $\cP(\cO^\op)$ is an $\infty$-topos we see that 
\[\cP(\cO^\op_{/X})\cong \cP(\cO^\op)_{/X}\cong \underset{(o\rightarrow X)\in\cO^\op_{/X}}{\lim}\cP(\cO^\op_{/o}).\]
Observe that since $\cO$ is saturated, every $o\in\cO$ is itself a Segal $\cO$-space, so $\cO_{/o}$ is also an algebraic pattern. Now we need to prove that the equivalence $\cP(\cO^\op_{/X})\cong\underset{(o\rightarrow X)\in\cO^\op_{/X}}{\lim}\cP(\cO^\op_{/o})$ respects the subcategory of Segal spaces. First, observe that for a Segal $\cO_{/X}$-space $(\cF\xrightarrow{f}X)$ and  $(o\xrightarrow{\gamma}X)$ it follows from \cref{prop:slice} that the restriction $\gamma^*\cF$ is a Segal $\cO_{/o}$-space. Conversely, assume that  $(\cF\xrightarrow{f}X)$ is such that all $\gamma^*\cF$ are Segal spaces. Then it follows that for every $o\in\cO$ and $x\in X(o)\cong \mor_{\cP(\cO^\op)}(h_o,X)$ we have that the fiber $\cF_x$ satisfies the Segal condition, which is exactly what we need for $\cF$ to be a Segal $\cO_{/X}$-space.
\end{proof}
\begin{defn}\label{def:cell}
Call a morphism $f:[n]\rightarrow [m]$ in $\Delta$ \textit{cellular} if $f(i+1)\leq f(i)+1$. Denote by $\mathpzc{Cat}^{\mathrm{Cell}}_{/[m]}$ the full subcategory of $(\Delta_{/[m]})^\op$ on cellular morphisms.
\end{defn}
\begin{prop}\label{prop:cell_def}
$\mathpzc{Cat}^{\mathrm{Cell}}_{/[m]}$ can be given a structure of an extendable algebraic pattern such that the inclusion $i_m^{\mathrm{Cell}}:\mathpzc{Cat}^{\mathrm{Cell}}_{/[m]}\hookrightarrow \mathpzc{Cat}_{/[m]}$ is an extendable morphism of algebraic patterns. Moreover, the morphism $i^\cell_{m,!}$ restricts to a morphism
\[i^\cell_{m,!}:\cat^\cell_{/[m]}\rightarrow\cat_{/[m]}\]
between the corresponding categories of Segal objects.
\end{prop}
\begin{proof}
To define the required structure we declare the subcategory of elementary objects and the subcategories of active and inert morphisms of $\mathpzc{Cat}^{\mathrm{Cell}}_{/[m]}$ to be the preimages of the respective subcategories of $\mathpzc{Cat}_{/[m]}$ under $i^\mathrm{Cell}$. To prove that this gives $\mathpzc{Cat}^{\mathrm{Cell}}_{/[m]}$ the structure of an algebraic pattern it suffices to prove that it is closed under the active/inert decomposition. In other words, we need to prove that if in the commutative diagram
\[
\begin{tikzcd}[row sep=huge, column sep=huge]
{[n]}\arrow[r, two heads, "a"]\arrow[rd, "f"]&{[l]}\arrow[r, tail, "i"]\arrow[d, "g"]&{[k]}\arrow[ld, "h" swap]\\
{}&{[m]}
\end{tikzcd}
\]
both $f$ and $h$ are cellular, then $g$ is also cellular. This follows since $i$ identifies $[l]$ with a subinterval of $[k]$.\par
To prove the extendability of $\ccat^\cell_{/[m]}$ and $i_m^\cell$, we first observe that if $f:[n]\rightarrow [m]$ is a cellular morphism, then so are its restrictions to all elementary subintervals $i:[1]\hookrightarrow [n]$. Together with the fact that $\mathpzc{Cat}^{\mathrm{Cell}}_{/[m]}$ is a full subcategory of $\mathpzc{Cat}_{/[m]}$ this implies that $\mElo{(\mathpzc{Cat}^{\mathrm{Cell}}_{/[m]})}{f}\cong\mElo{(\mathpzc{Cat}_{/[m]})}{i_m^{\mathrm{Cell}}(f)}$. Since the active/inert factorization for $\ccat^\cell_{/[m]}$ is induced from that of $\ccat$, we also see that for any active morphism $\phi:i^\cell_m(f)\twoheadrightarrow g$ we have $\mathpzc{Cat}^{\mathrm{Cell},\el}_{/[m]}(\phi)\cong\mathpzc{Cat}^\el_{/[m]}(\phi)$. This proves the relative Segal condition for both $\ccat^\cell_{/[m]}$ and $i_m^\cell$.\par 
We will now prove that for every $f\in\mathpzc{Cat}_{/[m]}$ we have that $(\mathpzc{Cat}^{\mathrm{Cell}}_{/[m]})^\act_{/f}$ is a Segal category. This will in particular imply that $\ccat^\cell_{/[m]}$ is an extendable algebraic pattern. Since we know that $(\ccat_{/[m]})^\act_{/f}$ is a Segal category it suffices to prove the following statement: given a commutative diagram 
\[
\begin{tikzcd}[row sep=huge, column sep=huge]
{[n]}\arrow[rr,two heads, "a"]\arrow[rd, "f"]&{}&{[l]}\arrow[ld, "g" swap]\\
{}&{[m]}
\end{tikzcd}
\]
in which for every interval $[i,i+1]\in[n]$ we have that $g|_{\{a(i),a(i)+1,...,a(i+1)\}}$ is cellular, it follows that $g$ itself is cellular. This is obvious since $a$ by definition preserves the endpoints.\par
It now remains to prove that $(\mathpzc{Cat}^{\mathrm{Cell}}_{/[m]})^\act_{/f}$-colimits distribute over $(\ccat_{/[m]})^\el_{f/}$-limits. To prove this we first make the condition more explicit. First factor $f$ as $[n]\xrightarrow{f_s}[\overline{n}]\xrightarrow{f_i}[m]$ where $f_i$ is injective and $f_s$ is surjective and let $s_j$ be the cardinality of the fiber of $f_s$ over $j$ for $0\leq j\leq s$. Without loss of generality, we can assume $f(0)=0$ and $f(n)=m$. First, denote by $(\ccat^\cell_{/[m]})^{\act,\un}_{/f}$ the subcategory of $(\ccat^\cell_{/[m]})^\act_{/f}$ on those active morphisms $f\overset{a}{\twoheadrightarrow}g$ whose restriction to any elementary subinterval $[s,s+1]\subset [n]$ that is sent to a point by $f$ is identity, we claim that it is cofinal. Indeed, using the Segal condition for $(\ccat^\cell_{/[m]})^\act_{/f}$ proved above we see that
\[(\ccat^\cell_{/[m]})^\act_{/f}\cong\bigtimes_{(f\rightarrowtail e)\in\overline{(\ccat_{/[m]})^\el}_{f/}}(\ccat^\cell_{/[m]})^\act_{/e},\]
where $\overline{(\ccat_{/[m]})^\el}_{f/}$ denotes the full subcategory of $(\ccat_{/[m]})^\el_{f/}$ on those morphisms $(f\rightarrowtail e)$ for which $e$ is an elementary morphism of the form $[1]\overset{i}{\rightarrowtail} [n]$. Observe that for every elementary subinterval $e$ that is sent to the identity in $[m]$ we have
\[(\ccat^\cell_{/[m]})^\act_{/e}\cong \Delta^{\act,\op}_{/[1]}.\]
It follows that
\[(\ccat^\cell_{/[m]})^\act_{/f}\cong(\Delta^{\act,\op}_{/[1]})^{\sum_{i=0}^{\overline{n}} s_i}\times (\ccat^\cell_{/[m]})^{\act,\un}_{/f}.\]
Observe that $\Delta^{\act,\op}_{/[1]}$ has a final object, and so the subcategory $(\ccat^\cell_{/[m]})^{\act,\un}_{/f}$ is indeed cofinal.\par
It follows that we can assume that $f$ is injective. Now let $[j,j+1]$ be a subinterval of $[n]$, observe that 
\[(\ccat^\cell_{/[m]})^{\act,\un}_{/[j,j+1]}\cong\Delta^\op_{-\infty}\times(\Delta^\op)^{f(j+1)-f(j)-1}\times\Delta^\op_\infty,\]
where $\Delta^\op_{-\infty}$ (resp. $\Delta^\op_{\infty}$) denotes the subcategory of $\Delta^\op$ on morphisms that preserve the minimal (resp. maximal) element. Indeed, an object of $(\ccat^\cell_{/[m]})^{\act,\un}_{/[j,j+1]}$ can be identified with a surjective cellular morphism to $[f(j),f(j+1)]\subset[m]$. An object like this is uniquely determined by its fibers over each $i\in[f(j),f(j+1)]$, and those fibers are all nonempty, so it can be identified with an object of $(\Delta^\op)^{f(j+1)-f(j)+1}$. An arbitrary morphism between those objects can also be identified with a morphism in $(\Delta^\op)^{f(j+1)-f(j)+1}$, observe that this morphism is active if and only if its restriction to the first component belongs to $\Delta^\op_{-\infty}$ and its restriction to the last component to $\Delta^\op_{\infty}$, which concludes the proof of the claim. Since both $\Delta^\op_{-\infty}$ and $\Delta^\op_{\infty}$ have a final object we see that $(\ccat^\cell_{/[m]})^{\act,\un}_{/f}$ has a cofinal subcategory of the form $(\Delta^\op)^{m-n}$.\par
Now we will need to introduce some notation. Let $\cF$ be a Segal $\ccat^\cell_{/[m]}$-space, denote by $\cF_i$ the value of $\cF$ on a morphism $[0]\rightarrow[m]$ sending $0$ to $i$ and by $\cF_{s,t}$ the value on a morphism $[1]\rightarrow[m]$ sending $0$ to $s$ and $1$ to $t$. Further, denote 
\[\mathop \times_{\cF_j}^{n_j}\cF_{j,j}\bydef \overbrace{(\cF_{j,j}\times_{\cF_j}\cF_{j,j}\times_{\cF_j}...\times_{\cF_j}\cF_{j,j})}^\text{$n_j$ times}\]
and for $\overline{n^i}=\{n^i_1,...,n^i_{f(i)-f(i-1)-1}\}\in\Delta^{f(i)-f(i-1)-1}$
\[\overline{\cF}(\overline{n^i})\bydef (\mathop \times_{\cF_{f(i-1)+1}}^{n^i_1}\cF_{f(i-1)+1,f(i-1)+1})\times_{\cF_{f(i-1)+1,f(i-1)+2}}...\times_{\cF_{f(i)-1,f(i)}}(\mathop \times_{\cF_{f(i)-1}}^{n^i_{f(i)-f(i-1)-1}}\cF_{f(i)-1,f(i)-1}).\]
With these notations, we need to prove that
\begin{align*}
    i^\cell_{m,!}\cF\cong &\underset{\{\overline{n}^i\}_{i\in\{1,2,...,n\}}\in\prod_{i=1}^n (\Delta^\op)^{f(i)-f(i-1)-1}}{\colim}\overline{\cF}(\overline{n^1})\times_{\cF_{f(1)}}...\times_{\cF_{f(n-1)}}\overline{\cF}(\overline{n^n})\\
    \cong& (\underset{\overline{n}^1\in(\Delta^\op)^{f(1)-f(0)-1}}{\colim}\overline{\cF}(\overline{n^1}))\times_{\cF_{f(1)}}...\times_{\cF_{f(n-1)}}(\underset{\overline{n}^n\in(\Delta^\op)^{f(n)-f(n-1)-1}}{\colim}\overline{\cF}(\overline{n^n})).
\end{align*}
Using that $\cS$ is an $\infty$-topos, we see that this colimit diagram can be viewed as a colimit diagram in $\cS_{/\prod_{i=1}^{m-1}\cF_i}$ and also that the limits appearing in $i^\cell_{m,!}\cF(f)$ are given by finite products in $\cS_{/\prod_{i=1}^{m-1}\cF_i}$. It will now suffice to prove that $(\Delta^\op)^{m-n}$-colimits distribute over finite products in $\cS_{/\prod_{i=1}^{m-1}\cF_i}$, which follows from \cite[Lemma 7.15]{chu2019homotopy} since $\cS_{/\prod_{i=1}^{m-1}\cF_i}$ is $\times$-admissible.
\end{proof}
\begin{notation}
We will denote by  $\mathrm{Cat}^{\mathrm{Cell}}_{/[m]}$ the category $\mathrm{Seg}_{\mathpzc{Cat}^\mathrm{Cell}_{/[m]}}(\cS)$.
\end{notation}
\begin{construction}\label{constr:corr_aux}
Define $(\ccat^\cell_{/[m]})_a$ to be a category with objects given by cellular morphisms \[s:\sqcup_{j=1}^k[n_j]\rightarrow[m]\]
and define morphisms between two objects to be morphisms of categories over $[m]$. We will call a morphism \textit{inert} if its restriction to all disjoint components is inert in the usual sense and denote $(\ccat^\cell_{/[m]})_a^\inrt$ the full subcategory on inert morphisms. We call an object $e:[l]\rightarrow[m]$ \textit{elementary} if $l=0$ or $l=1$ and denote $j:(\ccat^\cell_{/[m]})_a^\el\hookrightarrow(\ccat^\cell_{/[m]})_a^\inrt$ the inclusion of a full subcategory on elementary objects.
\end{construction}
\begin{lemma}
$\cat^\cell_{/[m]}$ is isomorphic to the category of presheaves on $(\ccat^\cell_{/[m]})_a$ whose restriction to $(\ccat^\cell_{/[m]})_a^\inrt$ is the right Kan extension of its restriction to $(\ccat^\cell_{/[m]})_a^\el$.
\end{lemma}
\begin{proof}
Observe that there is a natural inclusion of a full subcategory $i:\ccat^\cell_{/[m]}\hookrightarrow(\ccat^\cell_{/[m]})_a$. We will first prove that $i_*$ takes Segal $\ccat^\cell_{/[m]}$-spaces to presheaves satisfying the conditions of the lemma and that $i^*$ is its inverse. By definition
\[i_*\cF(s)\cong\underset{(s\rightarrow ix)\in s/i}{\lim}\cF(x).\]
Denote by $(s/i)^\mathrm{incl}\hookrightarrow s/i$ the full subcategory on inclusions of the form $[n_i]\hookrightarrow\sqcup_{j=1}^k[n_j]$, it is obviously coinitial in $s/i$ and consists of $k$ disjoint objects. It follows that 
\[i_*\cF(s)\cong \mathop{\times}\limits_{j=1}^k\cF(s_j).\]
It is also easy to see that 
\[(\ccat^\cell_{/[m]})^\el_{a,s/}\cong \mathop{\times}\limits_{j=1}^k(\ccat^\cell_{/[m]})^\el_{a,s_j/},\]
so the claim follows. Now observe that the isomorphism $i^*i_*\cong\id$ follows since $i$ is a full subcategory inclusion and $i_*i^*\cF\cong\cF$ for presheaves satisfying the conclusion of the lemma follows from the above considerations.
\end{proof}
\begin{cor}\label{cor:iso}
There is a natural isomorphism
\[\mathrm{Cat}^{\mathrm{Cell}}_{/[m]}\cong\underset{e\in\mElo{Cat}{[m]}}{\lim}\mathrm{Cat}^{\mathrm{Cell}}_{/e}.\]
\end{cor}
\begin{proof}
It follows from the previous lemma in view of the obvious isomorphism
\begin{equation}\label{eq:one}
    (\mathpzc{Cat}^{\mathrm{Cell}}_{/[m]})_a\cong\underset{e\in\mElo{Cat}{[m]}}{\lim}(\mathpzc{Cat}^{\mathrm{Cell}}_{/e})_a
\end{equation}
which moreover respects elementary objects and inert morphisms.
\end{proof}
\begin{prop}\label{prop:corr}
There is a $\mathpzc{Cat}$-Segal space $\mathrm{Corr}$ whose space of objects is isomorphic to $\mathrm{Cat}^\sim\cong(\mathrm{Cat}^{\mathrm{Cell}}_{/[0]})^\sim$ and whose space of morphisms is isomorphic to $(\mathrm{Cat}^{\mathrm{Cell}}_{/[1]})^\sim$.
\end{prop}
\begin{proof}
We first define $\mathrm{Corr}([m])\bydef\mathrm{Cat}^{\mathrm{Cell}}_{/[m]}$. The isomorphism of \cref{cor:iso} then implies that this indeed defines a $\mathpzc{Cat}^\inrt$-Segal space. We will now prove that for every $f:[n]\rightarrow[m]$ the natural morphism $f^*:\seg_{\ccat_{/[m]}}(\cS)\rightarrow\seg_{\ccat_{/[n]}}(\cS)$ restricts to $f^*:\mathrm{Seg}_{\mathpzc{Cat}^\mathrm{Cell}_{/[m]}}(\cS)\rightarrow\mathrm{Seg}_{\mathpzc{Cat}^\mathrm{Cell}_{/[n]}}(\cS)$.\par
We first write the required condition more explicitly. Given $\cF\in\mathrm{Seg}_{\mathpzc{Cat}^\mathrm{Cell}_{/[m]}}(\cS)$ and a morphism $g:[k]\rightarrow[n]$ we have
\[f^*\cF(g)\cong\cF(f\circ g)\cong\underset{\widetilde{c}\in i_m^\mathrm{Cell}/f\circ g}{\colim}\cF(\widetilde{c}),\]
where the last isomorphism follows since $\cF$ is a $\mathpzc{Cat}^\mathrm{Cell}_{/[m]}$-Segal space. For it to be a $\mathpzc{Cat}^\mathrm{Cell}_{/[n]}$-Segal space we need
\[f^*\cF(g)\cong\underset{c\in i_n^\mathrm{Cell}/g}{\colim}\cF(f\circ c)\cong\underset{c\in i_n^\mathrm{Cell}/g}{\colim}\underset{c_0\in i_m^\mathrm{Cell}/f\circ c}{\colim}\cF(c_0)\bydef\underset{c_0\in i_m^\mathrm{Cell}/f/g}{\colim}\cF(c_0),\]
where $i_m^\mathrm{Cell}/f/g$ is by definition a Cartesian fibration over $i_n^\mathrm{Cell}/g$ associated to a functor sending $c\rightarrow g$ to $i_m^\mathrm{Cell}/f\circ c$. The objects of $i_m^\mathrm{Cell}/f/g$ can be identified with pairs $(c\xrightarrow{\alpha} g,c_0\xrightarrow{\beta} f\circ c)$ where both $c$ and $c_0$ are cellular. There is a natural functor from $m:i_m^\mathrm{Cell}/f/g\rightarrow i_m^\mathrm{Cell}/f\circ g$ given by sending $(c\xrightarrow{\alpha} g,c_0\xrightarrow{\beta} f\circ c)$ to the composition $c_0\xrightarrow{\beta}f\circ c\xrightarrow{f*\alpha}f\circ g$ and we will prove that $m$ is cofinal.\par
To accomplish this, first recall the classical fact that every morphism $f:[n]\rightarrow [m]$ in $\Delta$ is expressible as a finite composition of face and degeneracy maps $\sigma_i$ and $\delta_j$. It will be enough to prove our statement for any of those maps individually. We start with the face map $\delta_j:[n]\rightarrow[n+1]$. The object of $i_m^\mathrm{Cell}/\delta_j/g$ can be identified with a pair of commutative diagrams
\[
\begin{tikzcd}[row sep=huge, column sep=huge]
{[l]}\arrow[dr, "g"]\arrow[r, "\alpha"]&{[k_1]\arrow[d, "c_1"]}\\
{}&{[n]}
\end{tikzcd},
\begin{tikzcd}[row sep=huge, column sep=huge]
{[k_1]}\arrow[r, "\beta"]\arrow[d, "c_1"]&{[k_2]}\arrow[d, "c_2"]\\
{[n]}\arrow[r, "\delta_j"]&{[n+1]}
\end{tikzcd}.
\]
For a given object $[l]\xrightarrow{\gamma}[k]$ of $i_m^\mathrm{Cell}/\delta_j\circ g$ an object of $m/\gamma$ can be identified with a commutative diagram
\[
\begin{tikzcd}[row sep=huge, column sep=huge]
{[l]}\arrow[dr, "g"]\arrow[r, "\alpha"]&{[k_1]}\arrow[d, "c_1"]\arrow[r, "\beta"]&{[k_2]}\arrow[d, "c_2"]\arrow[r, "\eta"]&{[k]}\arrow[dl, "c" swap]\\
{}&{[n]}\arrow[r, "\delta_j"]&{[n+1]}
\end{tikzcd}
\]
in which $\eta\circ \beta\circ \alpha\cong\gamma$ and a morphism is given by a commutative diagram 
\[
\begin{tikzcd}[row sep=huge, column sep=huge]
{}&{[k_1]}\arrow[r, "\beta"]\arrow[dd, "a_1"]&{[k_2]}\arrow[dd, "a_2"]\arrow[dr, "\eta"]\\
{[l]}\arrow[ur, "\alpha"]\arrow[dr, "\alpha'"]&{}&{}&{[k]}\\
{}&{[k_1']}\arrow[r, "\beta'"]&{[k_2']}\arrow[ur, "\eta'"]
\end{tikzcd}
\]
over $[n+1]$. Observe that we can factor $[k_1]\xrightarrow{\gamma}[k]$ canonically as the top row in the following diagram 
\[
\begin{tikzcd}[row sep=huge, column sep=huge]
{[l]}\arrow[r]\arrow[dr, "g"]&\delta_j^*{[k]}\arrow[d, "\delta_j^* c"]\arrow[r]&{[k]}\arrow[d,"c"]\arrow[r, equal]&{[k]}\arrow[dl, "c" swap]\\
{}&{[n]}\arrow[r, "\delta_j"]&{[n+1]}
\end{tikzcd},
\]
where the square is a pullback square. Observe that $\ima(c)\cap\ima(\delta_j)\neq\varnothing$ since otherwise it would not admit a morphism from $\delta_j\circ g$. Since $\delta_j$ is injective it follows from \cref{lem:pullback} that $\delta_j^*[k]$ is indeed an object of $\Delta$, moreover it is easy to see that the pullback $\delta_j^* c$ is also cellular. It follows that we can identify the diagram with an object of $m/\gamma$. Moreover, by construction this object is final, so in particular it follows that $m/\gamma$ is contractible and $m:i_{n+1}^\mathrm{Cell}/\delta_j/g\rightarrow i_{n+1}^\mathrm{Cell}/\delta_j\circ g$ is cofinal.\par
The case of a degeneracy map $\sigma_i:[n]\rightarrow[n-1]$ is somewhat more difficult. First, observe that the composition of cellular maps is cellular and that $\sigma_i$ is a cellular map. It follows that in this case the functor $m$ restricts to $\widetilde{m}:i_{n}^\mathrm{Cell}/g\rightarrow i_{n-1}^\mathrm{Cell}/\delta_j\circ g$, and it suffices to prove that this functor is cofinal. Observe that we can decompose every $\alpha:c\rightarrow g$ as $c\overset{i_\alpha}{\rightarrowtail}c_\alpha\overset{a_\alpha}{\twoheadrightarrow}g$. It follows that the inclusions $i^\act:(i_{n}^\mathrm{Cell}/g)^\act\hookrightarrow i_{n}^\mathrm{Cell}/g$ is cofinal. Moreover, $\widetilde{m}$ restricts to $\widehat{m}:(i_{n}^\mathrm{Cell}/g)^\act\rightarrow(i_{n-1}^\mathrm{Cell}/\delta_j\circ g)^\act$ and it is enough to prove that $\widehat{m}$ is cofinal. Next, observe that, since $\sigma_i^*$ is a Segal morphism, we have
\[\sigma_i^*\cF(g)\cong\underset{(g\rightarrowtail e)\in(\mathpzc{Cat}^\mathrm{Cell}_{/[n]})^\el_{/g}}{\lim}\sigma_i^*\cF(e),\]
so we can assume that $g$ is a morphism from $e$ to $[n]$ where $e$ is an elementary object of $\mathpzc{Cat}$. The case of $e=[0]$ is trivial, so we can assume that $g$ is a morphism $[1]\rightarrow [n]$. Denote the minimal subinterval of $[n]$ containing the image of $g$ by $[n_g]$, then since there is at most one active morphism to $[1]$ from any object of $\mathpzc{Cat}$, it follows that $(i_{n}^\mathrm{Cell}/g)^\act\cong\mathpzc{Cat}^{\mathrm{Cell},\mathrm{Surj}}_{/[n_g]}$, where $\ccat^{\cell,\sur}_{/[n_g]}$ denotes the subcategory of $\mathpzc{Cat}^{\mathrm{Cell}}_{/[n_g]}$ on surjective cellular morphisms $[m]\rightarrow[n_g]$ and active morphisms between them, and that a similar isomorphism holds for $\sigma_i\circ g$. Observe that the isomorphism of \cref{cor:iso} restricts to the isomorphism
\[\mathpzc{Cat}^{\mathrm{Cell},\mathrm{Surj}}_{/[n_g]}\cong\underset{e\in\mElo{Cat}{[m]}}{\lim}\mathpzc{Cat}^{\mathrm{Cell},\mathrm{Surj}}_{/e}.\]
Observe that the restriction of $\widehat{m}$ to $\mathpzc{Cat}^{\mathrm{Cell},\mathrm{Surj}}_{/e}$ is isomorphic to identity if $e$ does not intersect with $\{i,i+1\}$ inside $n_g$. It follows that it suffices to prove that $\widehat{m}$ is cofinal when restricted to only those components that intersect $\{i,i+1\}$.\par
Now we have three cases to consider: $i=0$, $i=n-1$ and $0<i<n-1$. The first and second case are entirely similar, so we only consider one of them. Assume $i=0$, then we are reduced to the following situation: $n=2$ and $g$ is the unique active morphism from $[1]$ to $[2]$. Observe that in this case $\mathrm{Cat}^{\mathrm{Cell},\mathrm{Surj}}_{/[2]}\cong \Delta^\op_{-\infty}\times\Delta^\op\times \Delta^\op_\infty$ and $\mathrm{Cat}^{\mathrm{Cell},\mathrm{Surj}}_{/[1]}\cong \Delta^\op_{-\infty}\times \Delta^\op_\infty$. The functor $\widehat{m}$ then sends $(a,n,b)$ to $(a+n+1,b)$. Observe that we have the following diagram
\[
\begin{tikzcd}[row sep=huge, column sep=huge]
\Delta^\op_{-\infty}\times\Delta^\op\times \Delta^\op_\infty\arrow[rr,"{(a,n,b)}\rightarrow{(a+n+1,b)}"]&{}&\Delta^\op_{-\infty}\times \Delta^\op_\infty\\
\Delta^\op\arrow[u, "n\rightarrow{(0,n,0)}"]\arrow[rr,"{n\rightarrow n+1}"]&{}&\Delta^\op_{-\infty}\arrow[u, "a\rightarrow{(a,0)}" swap]
\end{tikzcd}.
\]
Here both vertical morphisms are cofinal (since both $\Delta^\op_\infty$ and $\Delta^\op_{-\infty}$ have a final object), so in order to prove that the top horizontal morphism is cofinal it suffices to prove that the bottom horizontal morphism is cofinal. This follows from \cite[Lemma 6.1.3.16.]{lurie2009higher}.\par
The proof in the third case is similar. Assume $0<i<n$, then we are reduced to the following situation: $n=3$ and $g$ is the unique active morphism from $[1]$ to $[3]$. Observe that in this case $\mathrm{Cat}^{\mathrm{Cell},\mathrm{Surj}}_{/[3]}\cong \Delta^\op_{-\infty}\times\Delta^\op\times\Delta^\op\times \Delta^\op_\infty$ and $\mathrm{Cat}^{\mathrm{Cell},\mathrm{Surj}}_{/[2]}\cong \Delta^\op_{-\infty}\times\Delta^\op\times \Delta^\op_\infty$. The functor $\widehat{m}$ then sends $(a,n,m,b)$ to $(a,n+m+1,b)$. Observe that we have the following diagram
\[
\begin{tikzcd}[row sep=huge, column sep=huge]
\Delta^\op_{-\infty}\times\Delta^\op\times\Delta^\op\times \Delta^\op_\infty\arrow[rr,"{(a,n,m,b)}\rightarrow{(a,m+n+1,b)}"]&{}&\Delta^\op_{-\infty}\times\Delta^\op\times \Delta^\op_\infty\\
\Delta^\op\times\Delta^\op\arrow[u, "{(n,m)}\rightarrow{(0,n,m,0)}"]\arrow[rr,"{{(n,m)}\rightarrow n+m+1}"]&{}&\Delta^\op\arrow[u, "n\rightarrow{(0,n,0)}" swap]
\end{tikzcd}.
\]
Here both vertical morphisms are again cofinal, so in order to prove that the top horizontal morphism is cofinal it suffices to prove that the bottom horizontal morphism is cofinal. Denote by $F:\Delta\times\Delta\rightarrow\Delta$ the morphism sending $([k],[l])$ to $[k+l+1]$, we need to prove it is coinitial, i.e. we need to prove that for any $[n]\in\Delta$ the overcategory $\Delta_{F/[n]}$ is contractible. An object of $\Delta_{F/[n]}$ is given by the data of natural numbers $k$ and $l$, a morphism $\alpha:[1]\rightarrow[n]$ and a pair of morphisms $f:[k]\rightarrow [0,\alpha(0)]$ and $g:[l]\rightarrow [\alpha(1),n]$. Denote by $i^\id:\Delta_{F/[n]}^\id\hookrightarrow\Delta_{F/[n]}$ the inclusion of the full subcategory on those objects for which both $f$ and $g$ are identities. It admits an obvious left adjoint $L:\Delta_{F/[n]}\rightarrow\Delta_{F/[n]}^\id$ sending $(k,l,\alpha,f,g)$ to $(\alpha(0),(n-\alpha(1)),\alpha,\id_{[\alpha(0)]},\id_{[n-\alpha(1)]})$. In particular, there is a natural isomorphism \[N(\Delta_{F/[n]})\xrightarrow{\sim}N(\Delta_{F/[n]}^\id)\]
on geometric realizations, so it is enough to prove that $\Delta_{F/[n]}^\id$ is contractible. The objects of $\Delta_{F/[n]}^\id$ are easily seen to be uniquely determined by $\alpha:[1]\rightarrow[n]$, and there is a unique morphism from $\alpha$ to $\beta$ if and only if $\alpha(0)\leq\beta(0)\leq\beta(1)\leq\alpha(1)$. In particular, it admits an initial object given by the unique active morphism $a:[1]\twoheadrightarrow[n]$, so it follows that $\Delta_{F/[n]}^\id$ is contractible.
\end{proof}
\begin{notation}
We will denote a morphism from $\cC$ to $\cD$ in $\mathrm{Corr}$ by $\cC\nrightarrow \cD$ and call it a \textit{correspondence} from $\cC$ to $\cD$. We will also frequently denote the space $M(\id_{[1]})$ by the same letter.
\end{notation}
\begin{remark}\label{rem:bicat}
Observe that both the space of objects and the space of morphisms of $\mathrm{Corr}$ are themselves categories, and moreover the unit and multiplication maps for $\mathrm{Corr}$ are functors between those categories. This endows $\mathrm{Corr}$ with a structure of a double Segal space. If we only consider those morphisms whose restriction to the category of objects are identity we get a structure of a twofold Segal space, and we will typically view $\mathrm{Corr}$ as endowed with this structure.
\end{remark}
\begin{remark}\label{rem:corr_corr}
Given a morphism $h:[k]\rightarrow[s]$, we can consider the induced functor from $\ccat^\cell_{/[s]}$ to $\ccat^\cell_{/[k]}$ as given by correspondence
\[
h^*_\cell\bydef\begin{tikzcd}[row sep=huge, column sep=huge]
\ccat^\cell_{/[s]}\arrow[dr, "i^\cell_{s,!}"]&{}&{}&\ccat^\cell_{/[k]}\\
{}&\ccat_{/[s]}\arrow[r, "h^*"]&\ccat_{/[k]}\arrow[ur, "i_k^{\cell,*}"]
\end{tikzcd}.
\]
Observe that when viewed as a functor $(\ccat^\cell_{/[k]}\times(\ccat^\cell_{/[s]})^\op)\rightarrow\cS$ it sends a pair of cellular morphisms $[m]\xrightarrow{c_k}[k]$ and $[n]\xrightarrow{c_s}[s]$ to the space of vertical morphisms $f$ making the following diagram commute
\[
\begin{tikzcd}[row sep=huge, column sep=huge]
{[m]}\arrow[d, dotted, "f"]\arrow[r, "c_k"]&{[k]}\arrow[d, "h"]\\
{[n]}\arrow[r, "c_s"]&{[s]}
\end{tikzcd}.
\]
\end{remark}
\begin{notation}\label{not:inrt}
Given $[m]\in\Delta$, denote by $\ccat^\inrt_{/[m]}$ the algebraic pattern whose underlying category is $(\Delta^\inrt_{/[m]})^\op$, whose inert subcategory contains all morphism and whose active subcategory is trivial. Declare an object $[k]\rightarrowtail[m]$ to be elementary if $k\in\{0,1\}$. Denote by $i_m^\inrt:\ccat^\inrt_{/[m]}\hookrightarrow\ccat^\cell_{/[m]}$ the natural inclusion.
\end{notation}
\begin{prop}
$i_m^\inrt$ is an extendable morphism of algebraic patterns.
\end{prop}
\begin{proof}
Observe that $i_m^\inrt$ has unique lifting of inert morphisms. The relative Segal condition follows from the relative Segal condition in $\ccat$. Observe that for any $f\in\ccat^\cell_{/[m]}$ the category $(\ccat^\inrt_{/[m]})^\act_{/f}$ is a singleton $\{f_i\}$ containing the inert part of the active/inert decomposition of $f$. In particular, it is a Segal space.
\end{proof}
\begin{prop}\label{prop:span}
For every $f:[n]\rightarrow[m]$ the induced morphism $f^*:\seg_{\ccat_{/[m]}}(\cS)\rightarrow\seg_{\ccat_{/[n]}}(\cS)$ preserves the image of $i_m^\cell\circ i_m^\inrt$.
\end{prop}
\begin{proof}
Observe that $\cF\in\ccat_{/[m]}$ is in the essential image of $i_m^\cell\circ i_m^\inrt$ if and only if $\cF(g)\cong\cF(g_i)$ where $[l]\overset{g_a}{\twoheadrightarrow}[\overline{l}]\overset{g_i}{\rightarrowtail}[m]$ is the active/inert factorization of $g$. By definition $f^*\cF(h)\cong \cF(f\circ h)$, so let $f\cong f_i\circ f_a$ and $h\cong h_i\circ h_a$ be the active/inert factorizations in $\Delta$ and consider the following commutative diagram
\[
\begin{tikzcd}[row sep=huge, column sep=huge]
{[k]}\arrow[d, two heads, "h_a"]\\
{[\overline{k}]}\arrow[d, tail, "h_i"]\arrow[r, two heads, "\widetilde{f}"]&{[\widetilde{n}]}\arrow[d, tail, "\widetilde{h}"]\\
{[n]}\arrow[r, two heads, "f_a"]&{[\overline{n}]}\arrow[r, tail, "f_i"]&{[m]}
\end{tikzcd}.
\]
Then we see that $f^*\cF(h)\cong \cF(f_i\circ \widetilde{h}\circ \widetilde{f}\circ h_a)\cong \cF(f_i\circ \widetilde{h})$ since $\cF$ is in the image of of $i_m^\cell\circ i_m^\inrt$. Similarly, $f^*\cF(h_i)\cong \cF(f_i\circ \widetilde{h}\circ \widetilde{f})\cong \cF(f_i\circ \widetilde{h})$, so in particular $f^*\cF(h)\cong f^*\cF(h_i)$, so $f^*\cF$ also lies in the image of $i_n^\cell\circ i_n^\inrt$.
\end{proof}
\begin{cor}\label{cor:span}
There is a flagged bicategory $\spanc$ whose space of objects is given by $(\cS)^\sim$ and whose category of morphisms is given by $\cP((\ccat^\inrt_{[1]})^\el)\cong \seg_{\ccat^\inrt_{/[1]}}(\cS)$.
\end{cor}
\begin{proof}
Using \cref{prop:span} we construct $\spanc$ just as $\corr$ in \cref{prop:corr} and endow it with the structure of a flagged bicategory as in \cref{rem:bicat}.
\end{proof}
\begin{prop}\label{prop:span_subcat}
$\spanc$ can be identified with a full subcategory of $\corr$ on $\infty$-groupoids.
\end{prop}
\begin{proof}
Indeed, by definition every object of $\spanc$ is a groupoid. Conversely, if $\cF\in\seg_{\ccat^\cell_{/[1]}}\cong \corr_1$ is such that $\cF(f)\cong\cF(i)$, where $i\in\{0,1\}$ and $f:[n]\rightarrow[1]$ is a morphism sending every object of $[n]$ to $i$, then we see that for arbitrary $g:[l]\rightarrow[m]$ we either have $\cF(g)\cong \cF(i)$ or $\cF(g)\cong\cF(g_0)\times_{\cF(0)}\cF(\id_{[1]})\times_{\cF(1)}\cF(g_1)\cong \cF(\id_{[1]})$ where $g_i$ is a fiber of $g$ over $i$ (we merely applied the Segal condition). In particular, we have that $\cF(g)\cong \cF(g_i)$, so $\cF$ belongs to $\spanc_1$.
\end{proof}
\begin{remark}
It is possible to show that our flagged bicategory $\spanc$ is isomorphic to the faithful subcategory of $\spanc_2(\cS)$ in the notation of \cite{haugseng2018iterated} having the same objects and morphisms, but only the 2-morphisms of the form 
\[
\begin{tikzcd}[row sep=huge, column sep=huge]
{}&S\arrow[dl, "f" swap]\arrow[dr, "g"]\arrow[d, equal]\\
X&S\arrow[d, "s"]&Y\\
{}&T\arrow[ul, "f'" swap]\arrow[ur, "g'"]
\end{tikzcd}
\]
\end{remark}
\section{Basic properties of Corr}\label{sect:three}
Our goal in this section is to prove certain basic properties of $\corr$, most of which would be used in later chapters. We begin by defining correspondences $f^*$ and $f_!$ for every morphism $f$ in $\cat$ in \cref{not_cat} and proving that this defines functors $\cat\rightarrow\corr$ in \cref{prop:fun_comp} and \cref{prop:corr_bicat}. We also prove that they are adjoint to each other in \cref{prop:adj}.\par
Our next important result is \cref{cor:corr_comp} that uses the construction of completion from \cite{ayala2018flagged} to provide for complete Segal spaces $\cC$ and $\cD$ an isomorphism between our category of correspondences $\mor_\corr(\cC,\cD)$ and $\cP(\cC\times\cD^\op)$, thus proving that it extends the classical definition. Finally, in \cref{prop:nat} we give a rather technical description of the category $\mor_\corr(\cD,*)$ for $\cD\in\cat_{/[n]}$ that will be used in later sections.
\begin{construction}
Given a pair of Segal spaces $\cC$ and $\cD$, consider the functor $\Delta^\op\times\Delta^\op\rightarrow\cS$ sending $([n],[m])$ to $\cC_n\times\cD_m$. Denote by $\fib(\cC,\cD)$ the corresponding right fibration. Denote by $j:\fib(\cC,\cD)^\infty\hookrightarrow\fib(\cC,\cD)$ the inclusion of the subcategory containing only morphisms lying over $\Delta^\op_\infty\times\Delta^\op_{-\infty}$ and call those morphisms \textit{inert}. Finally, declare objects lying over $([0],[0])$ \textit{elementary} and denote by $i:\fib(\cC,\cD)^\el\hookrightarrow\fib(\cC,\cD)^\infty$ the inclusion of the full subcategory on inert morphisms. Define $\widetilde{\corr}(\cC,\cD)$ to be the category of presheaves $\cF$ on $\fib(\cC,\cD)^\op$ such that $j^*\cF\cong i_*i^*j^*\cF$.
\end{construction}
\begin{prop}\label{prop:corr_mor}
$\widetilde{\corr}(\cC,\cD)\cong\corr(\cC,\cD)$ and the forgetful functor $\widetilde{\corr}(\cC,\cD)\rightarrow\cP(\fib(\cC,\cD)^{\el,\op})$ is monadic.
\end{prop}
\begin{proof}
We first observe that a correspondence $M:\cC\nrightarrow\cD$ can be equivalently described as a functor $F:\Delta^\op\times\Delta^\op\rightarrow\cS$ such that $F([n],[m])\cong \cC_n\times_{\cC_0}M\times_{\cD_0}\cD_m$, where $M\rightarrow\cD_0\times\cC_0$ is the underlying space of the correspondence. Using the equivalence $\mor_\cat(X,\cS)\cong\cS_{/X}$ for $\infty$-groupoids and the fact that $([0],[0])$ is final in $\Delta^\op_\infty\times\Delta^\op_{-\infty}$, we see that those are exactly the presheaves satisfying the conditions of the lemma. To prove the monadicity claim we will prove that $j_!i_*$ restricts to a functor $\cP(\fib(\cC,\cD)^{\el,\op})\rightarrow\widetilde{\corr}(\cC,\cD)$, the monadicity would then follow as in \cite[Proposition 8.1]{chu2019homotopy}.\par
To prove it observe that 
\[j_!i_*M(c_m,d_n)\cong\underset{((c'_l,d'_k)\rightarrow(c_m,d_n))\in\fib(\cC,\cD)^\infty_{/(c_m,d_n)}}{\colim}i_*M(c'_l,d'_n),\]
where $\fib(\cC,\cD)^\infty_{/(c_m,d_n)}$ is the category with objects given by morphisms to $(c_m,d_n)$ and morphisms by morphisms over $(c_m,d_n)$ in $\fib(\cC,\cD)^\infty$. First, observe that all active morphisms belong to $\Delta^\op_\infty\times\Delta^\op_{-\infty}$, so using the fact that $\fib(\cC,\cD)$ admits an active/inert factorization we see that it suffices to take the colimit over the full subcategory $\fib(\cC,\cD)^{\inrt,\infty}_{/(c_m,d_n)}\hookrightarrow\fib(\cC,\cD)^{\infty}_{/(c_m,d_n)}$ on inert morphisms $((c'_l,d'_k)\rightarrowtail(c_m,d_n))$.\par
Observe that any inert morphism that does nor lie over $\Delta^\op_\infty\times\Delta^\op_{-\infty}$ can be uniquely decomposed as $c_\infty\circ f$ where $f$ is a morphism lying over $\Delta^\op_\infty\times\Delta^\op_{-\infty}$ and $c_\infty$ is a morphism lying over $s_\infty:([m],[n])\rightarrow([m+1],[n+1])$ which (considered as a morphism in $\Delta\times\Delta$) sends $([m],[n])$ isomorphically onto a subinterval $([1,m],[1,n])$. This means we can take the colimit over $\fib(\cC,\cD)^{s_\infty,\id}_{/(c_m,d_n)}$ which is the full subcategory of $\fib(\cC,\cD)^{\inrt,\infty}_{/(c_m,d_n)}$ on morphisms lying over $s_\infty$ or $\id$. Finally, observe that there is a morphism $\id\rightarrow s_\infty$ in $\Delta^\op\times\Delta^\op_{/([m],[n])}$, so the category $\fib(\cC,\cD)^{s_\infty}_{/(c_m,d_n)}$ of morphisms lying over $s_\infty$ is cofinal in $\fib(\cC,\cD)^{s_\infty,\id}_{/(c_m,d_n)}$. By untangling definitions, we finally see that $j_!i_*$ sends $M\in\cS_{/\cC_0\times\cD_0}$ to the presheaf sending $([n],[m])$ to $\cC_n\times_{\cC_0}\cC_1\times_{\cC_0}M\times_{\cD_0}\cD_1\times_{\cD_0}\cD_m$, which is obviously and element of $\widetilde{\corr}(\cC,\cD)$.
\end{proof}
\begin{cor}\label{cor:dual}
There is an isomorphism 
\[\corr(\cC,\cD)\cong\corr(*,\cC^\op\times\cD)\]
\end{cor}
\begin{proof}
First, observe that $\fib(\cC,\cD)^\el\cong\fib(*,\cD\times\cC)^\op\cong\cS_{/\cC_0\times\cD_0}$. It follows from the proof of \cref{prop:corr_mor} that the free correspondence monad on $\fib(\cC,\cD)^\el$ is given by $M\mapsto\cC_1\times_{\cC_0}M\times_{\cD_0}\cD_1$ and the claim follows since these monads are manifestly equivalent for $\widetilde{\corr}(\cC,\cD)$ and $\widetilde{\corr}(*,\cD\times\cC^\op)$.
\end{proof}
\begin{construction}\label{constr:tensor}
Observe that given two algebraic patterns $\cO$ and $\cE$, their product $\cO\times\cE$ also obtains a natural structure of an algebraic pattern such that the category of Segal $\cO\times \cE$-spaces is isomorphic to the category of Segal $\cO$-objects in $\seg_{\cE}(\cS)$. In particular, this applies to $\ccat\times[n]$ whose Segal spaces are given by strings of morphisms $\cC_0\rightarrow\cC_1\rightarrow...\rightarrow\cC_n$ of $\ccat$-Segal spaces. We define a functor $L_1:\ccat^\cell_{/[n]}\rightarrow\cP(\ccat\times[n])$ (resp. $L_2:\ccat^\cell_{/[n]}\rightarrow\cP(\ccat\times[n])$) sending $\alpha:[m]\rightarrow[n]$ to $\alpha^{-1}(\{0\})\hookrightarrow\alpha^{-1}([0,1])\hookrightarrow...\hookrightarrow \alpha^{-1}([n])=[m]$ (resp. to $\alpha^{-1}(\{n\})\hookrightarrow\alpha^{-1}([n-1,1])\hookrightarrow...\hookrightarrow \alpha^{-1}([n])=[m]$). 
\end{construction}
\begin{prop}
Both $L_i$ and their right adjoints $R_i$ restrict to functors between the respective categories of Segal objects.
\end{prop}
\begin{proof}
We will only prove the claim for $L_1$ and $R_1$, leaving a similar proof in the other case to the reader. First, observe that we can equivalently describe $\ccat^\cell_{/[n]}\xrightarrow{L_1}\ccat\times[n]\xrightarrow{\{j\}^*}\ccat$ as given by the correspondence
\[
\begin{tikzcd}[row sep=huge, column sep=huge]
{}&\ccat_{/[n]}\arrow[r, "i^*_j"]&\ccat_{/[j]}\arrow[dr, "p_!"]\\
\ccat^\cell_{/[n]}\arrow[ur, "i^\cell_{n,!}"]&{}&{}&\ccat
\end{tikzcd},
\]
where $i_j:[j]\hookrightarrow[n]$ denotes the natural inclusion of $[j]$ as an initial segment and $p:[j]\rightarrow[0]$ is the unique morphism. Since all of the morphisms involved restrict to morphisms between the corresponding Segal categories, we see that the same is true of $L_1$ as well. Conversely, given $f:\cC\rightarrow\cD$ and $\alpha:[m]\rightarrow[n]$ we see that 
\[R_1 f(\alpha)\cong\mor_{\Delta\times [n]}(L_1\alpha,f).\] 
From the description of $L_1$ we see that \[L_1f(\alpha)\cong\underset{e\in\ccat^{\cell,\el}_{\alpha/}}{\colim}L_1f(e),\]
so $R_1$ indeed satisfies the required property.
\end{proof}
\begin{notation}\label{not_cat}
Given $F:\cC\rightarrow\cD$ we will denote by $f_!:\cC\nrightarrow\cD$ the correspondence $R_1 f$ and by $f^*:\cD\nrightarrow\cC$ the correspondence $R_2f$.
\end{notation}
\begin{prop}\label{prop:fun_comp}
The morphism sending $f:\cC\rightarrow\cD$ to $f^*$ (resp. to $f_!$) extends to a functor $\cat\rightarrow \corr$.
\end{prop}
\begin{proof}
We will prove the theorem for $(-)_!$, leaving a similar argument for $(-)^*$ to the reader. We need to prove that for any morphism $\alpha:[n]\rightarrow[m]$ in $\Delta$ we have the following commutative diagram
\[
\begin{tikzcd}[row sep=huge, column sep=huge]
{[m]}\times \ccat\arrow[r, "R_{1,m}"]\arrow[d, "\alpha^*"]&\ccat^\cell_{/[m]}\arrow[d, "\alpha^*_\cell"]\\
{[n]}\times\ccat\arrow[r, "R_{1,n}"]&\ccat^\cell_{/[n]}
\end{tikzcd}.
\]
Observe that we can represent any $\alpha$ as a composition of face and degeneracy maps, so it would suffice to prove the claim for those. First assume that $\alpha$ is a degeneracy map $\sigma_i:[n]\rightarrow[n-1]$. Now observe that for any $\cF\in[m]\times\cat$ and for any $c:[l]\rightarrow[n]$ such that its image does not contain $i$ we have $R_{1,n}\circ \sigma_i^*\cF(c)\cong \sigma^*_{i,\cell}\circ R_{1,n-1}\cF(c)\cong R_{1,n-1}\cF(\sigma_i\circ{c})$. So we only have to prove the statement for the morphisms $c$ whose image contains $i$. Moreover, using that both $R_{1,n}\circ \sigma_i^*\cF$ and $\sigma^*_{i,\cell}\circ R_{1,n-1}\cF$ satisfy the Segal condition, we can further assume that the image of $c$ is contained in the interval $\{i-1,i\}$ or $\{i,i+1\}$. Since in this case $c$ uniquely factors through $[1]\hookrightarrow[n]$ we are finally reduced to proving the statement for $\sigma_):[1]\rightarrow[0]$. In other words, we have to prove that for any Segal space $\cC$ the morphism $R_1$ sends $(\cC=\joinrel=\cC)$ to the identity correspondence, which is immediate.\par
Now we consider the case of the face map $\delta_j:[n-1]\rightarrow[n]$. Just like in the previous paragraph, we are reduced to proving the claim for the unique active morphism $[1]\rightarrow[2]$. In other words, for a composable pair $(\cC\xrightarrow{f}\cD\xrightarrow{g}\cE)$ we need to demonstrate that $g_!\circ f_!\cong (g\circ f)_!$. It would suffice to demonstrate $g_!\circ f_!(\id_{[1]})\cong (g\circ f)_!(\id_{[1]})$. We have seen in the course of the proof of \cref{prop:corr} that $g_!\circ f_!(\id_{[1]})$ is isomorphic to the colimit of a functor $F:\Delta^\op\rightarrow\cS$ given by \[F([n])\bydef(\cC_0\times_{\cD_0}\cD_1)\overbrace{\times_{\cD_0}\cD_1\times_{\cD_0}...\times_{\cD_0}\cD_1\times_{\cD_0}}^\text{n times}(\cD_0\times_{\cE_0}\cE_1).\] 
Note that this functor is a restriction along $\Delta^\op\xrightarrow{[n]\rightarrow[n+1]}\Delta^\op_{-\infty}$ of a functor $\widetilde{F}:\Delta^\op_{-\infty}\rightarrow\cS$ sending $[n]$ to \[\widetilde{F}([n])\bydef(\cC_0)\overbrace{\times_{\cD_0}\cD_1\times_{\cD_0}...\times_{\cD_0}\cD_1\times_{\cD_0}}^\text{n times}(\cD_0\times_{\cE_0}\cE_1).\] 
Since $\Delta^\op_{-\infty}$ has a final object $[0]$, the colimit of this functor is isomorphic to \[\widetilde{F}([0])\cong\cC_0\times_{\cE_0}\cE_1\cong(g\circ f)_!(\id_{[1]}).\]
On the other hand we know from the proof of \cref{prop:corr} that $\Delta^\op\rightarrow\Delta^\op_{-\infty}$ is cofinal, so in the end we get
\[g_!\circ f_!(\id_{[1]})\cong\underset{[n]\in\Delta^\op}{\colim}F([n])\cong\underset{[n]\in\Delta^\op_{-\infty}}{\colim}\widetilde{F}([n])\cong\widetilde{F}([0])\cong (g\circ f)_!(\id_{[1]}).\]
\end{proof}
\begin{prop}\label{prop:fact_1}
Given a correspondence $M:\cC\nrightarrow\cD$, denote by $i_1:\cC\hookrightarrow M$ and $i_2:\cD\hookrightarrow M$ the natural inclusions of fibers over $0$ and $1$. Then we have
\[M\cong i_2^*\circ i_{1,!}.\]
\end{prop}
\begin{proof}
First, observe that by definition 
\begin{equation*}
    \begin{split}
    i_2^*\circ i_{1,!}\cong& \underset{[n]\in\Delta^\op}{\colim}\cC_0\times_{M_0}M_1\times_{M_0}M_n\times_{M_0}M_1\times_{M_0}\cD_0\\
    \cong& \underset{[n]\in\Delta^\op}{\colim}M_1^\cC\times_{M_0}M_n\times_{M_0}M_1^\cD\\
    \cong& \underset{[n]\in\Delta^\op}{\colim}M^{n+2,\sur},
    \end{split}
\end{equation*}
where $M_1^\cC$ (resp. $M_1^\cD$) denotes the space of morphisms of $M$ with source in $\cC$ (resp. with target in $\cD$) and $M^{n,\sur}$ denotes the space of composable strings of $n$ morphisms with source in $\cC$ and target in $\cD$. Observe that $M^{n,\sur}$ is the left Kan extension along $\Delta^\op\times \Delta^\op\xrightarrow{(n,m)\mapsto(n+m+1)}\Delta^\op$ of the presheaf on $\Delta\times\Delta$ sending $(n,m)$ to $\cC_n\times_{\cC_0}M\times_{\cD_0}\cD_m$. By transitivity of Kan extensions, we can continue the above chain of isomorphisms as follows:
\begin{equation*}
    \begin{split}
        i_2^*\circ i_{1,!}\cong&\underset{[n]\in\Delta^\op}{\colim}M^{n+2,\sur}\\
        \cong&\underset{(n,m)\in\Delta^\op\times \Delta^\op}{\colim}\cC_1\times_{\cC_0}(\cC_n\times_{\cC_0}M\times_{\cD_0}\cD_m)\times_{\cD_0}\cD_1\\
        \cong&\underset{(n,m)\in\Delta_{-\infty}^\op\times \Delta_\infty^\op}{\colim}\cC_n\times_{\cC_0}M\times_{\cD_0}\cD_m\\
        \cong&M,
    \end{split}
\end{equation*}
where the first isomorphism follows from the transitivity of Kan extensions, the second since $\Delta^\op\rightarrow\Delta^\op_{\infty}$ is cofinal and the last since $\Delta^\op_{\infty}$ has a final object.
\end{proof}
\begin{prop}\label{prop:adj}
For any morphism $f:\cD\rightarrow\cC$ in $\mathrm{Cat}$ the morphism $f_!$ is left adjoint to $f^*$ in $\mathrm{Corr}$.
\end{prop}
\begin{proof}
We begin by describing the unit and counit morphisms. We start with the counit $\epsilon:f_!f^*\rightarrow\id$:
\begin{equation*}
    \begin{split}
        f_!f^*\cong&\underset{[n]\in\Delta^\op}{\colim}(\cC_1\times_{\cC_0}\cD_0)\times_{\cD_0}\cD_n\times_{\cC_0}(\cD_1\times_{\cD_0}\cC_0)\\
        \cong &\underset{[n]\in\Delta^\op}{\colim}(\cC_1\times_{\cC_0}\cD_n)\times_{\cC_0}\cC_1\\
        \xrightarrow{f}&\underset{[n]\in\Delta^\op}{\colim}(\cC_1\times_{\cC_0}\cC_n)\times_{\cC_0}\cC_1\\
        \cong&\underset{[n]\in\Delta^\op}{\colim}\cC_{n+1}\times_{\cC_0}\cC_1\\
        \cong&\underset{[m]\in\Delta^\op_{-\infty}}{\colim}\cC_{m}\times_{\cC_0}\cC_1\\
        \cong&\cC_1\\
        \cong&\id_\cC,
    \end{split}
\end{equation*}
where the non-trivial isomorphisms follow by the same argument as in \cref{prop:fact_1}.
Now we observe the following isomorphism:
\begin{equation*}
    \begin{split}
        f^*f_!\cong&\underset{[n]\in\Delta^\op}{\colim}(\cD_0\times_{\cC_0}\cC_1)\times_{\cC_0}\cC_n\times_{\cC_0}(\cC_1\times_{\cC_0}\cD_0)\\
        \cong&\underset{[n]\in\Delta^\op}{\colim}\cD_0\times_{\cC_0}\cC_1\times_{\cC_0}\cD_0\times_{\cC_0}\cC_{(n+1)}\\
        \cong&\underset{[m]\in\Delta^\op_{-\infty}}{\colim}\cD_0\times_{\cC_0}\cC_1\times_{\cC_0}\cD_0\times_{\cC_0}\cC_{m}\\
        \cong&\cD_0\times_{\cC_0}\cC_1\times_{\cC_0}\cD_0.
    \end{split}
\end{equation*}
Using this isomorphism, we define $\eta:\id_\cD\rightarrow f^*f_!$ as $\cD_1\xrightarrow{(s,f,t)}\cD_0\times_{\cC_0}\cC_1\times_{\cC_0}\cD_0$. Now it remains to verify the triangle identities. We first verify that $f_!\xrightarrow{\eta}f_!f^*f_!\xrightarrow{\epsilon}f_!$ is isomorphic to the identity:
\begin{equation*}
    \begin{split}
        f_!\cong&\cD_0\times_{\cC_0}\cC_1\\
        \cong&\underset{[n]\in\Delta^\op}{\colim}\cD_1\times_{\cD_0}\cD_n\times_{\cD_0}(\cD_0\times_{\cC_0}\cC_1)\\
        \xrightarrow{f\times\id\times\id}&\underset{[n]\in\Delta^\op}{\colim}\cC_1\times_{\cC_0}\cD_n\times_{\cD_0}(\cD_0\times_{\cC_0}\cC_1)\\
        \cong&\underset{[n]\in\Delta^\op}{\colim}(\cC_1\times_{\cC_0}\cD_n)\times_{\cD_0}(\cD_0\times_{\cC_0}\cC_1)\\
        \xrightarrow{f\times\id}&\underset{[n]\in\Delta^\op}{\colim}\cC_{n+1}\times_{\cD_0}(\cD_0\times_{\cC_0}\cC_1)\\
        \cong &\cD_0\times_{\cC_0}\cC_1.
    \end{split}
\end{equation*}
Since the restriction of all the morphisms to $(\cD_0\times_{\cC_0}\cC_1)$ is isomorphic to identity, we conclude that the whole composition is also isomorphic to identity. The other triangle identity is proved in a similar manner:
\begin{equation*}
    \begin{split}
        f^*\cong&\cC_1\times_{\cC_0} \cD_0\\
        \cong&\underset{[n]\in\Delta^\op}{\colim}(\cC_1\times_{\cC_0} \cD_0)\times_{\cD_0}\cD_n\times_{\cD_0}\cD_1\\
        \xrightarrow{\id\times \id\times f}&\underset{[n]\in\Delta^\op}{\colim}(\cC_1\times_{\cC_0} \cD_0)\times_{\cD_0}\cD_n\times_{\cD_0}\cC_1\\
        \cong&\underset{[n]\in\Delta^\op}{\colim}(\cC_1\times_{\cC_0} \cD_0)\times_{\cD_0}(\cD_n\times_{\cD_0}\cC_1)\\
        \xrightarrow{\id\times f}&\underset{[n]\in\Delta^\op}{\colim}(\cC_1\times_{\cC_0} \cD_0)\times_{\cD_0}\cC_{n+1}\\
        \cong&\cC_1\times_{\cC_0} \cD_0,
    \end{split}
\end{equation*}
which is isomorphic to the identity for the same reason.
\end{proof}
\begin{construction}\label{constr:rfib}
Let $M:\cC\nrightarrow\cD$ be a correspondence. We will construct from it a simplicial space $\widetilde{M}$ as follows: first consider the functor $F:\Delta\xrightarrow{\Delta}\Delta\times\Delta\rightarrow\Delta_{/[1]}$ sending $[n]$ to the morphism $f:[2n+1]\rightarrow[1]$ that sends $[0,n]$ to $0$, the segment $[n,n+1]$ to $[0,1]$ and the rest to $1$. We define $\widetilde{M}\bydef F^*M$, where $M$ is viewed as a presheaf on $\Delta_{/[1]}$. Observe that there are natural transformations $\alpha:j_i\rightarrow F, i\in\{1,2\}$ where $j_i:\Delta\rightarrow\Delta_{/[1]}$ is a functor sending $[n]$ to a constant morphism $[n]\rightarrow[1]$ with the target $\{i\}$. It induces a natural functor $p:\widetilde{M}\rightarrow\cC\times \cD$. 
\end{construction}
\begin{lemma}\label{lem:mon_slice}
Assume we are given an $\infty$-category $\cC$ equipped with a monad $T$ and a fixed object $x\in\cC$. Then giving a structure of a $T$-algebra on $x$ is equivalent to lifting $T$ to a monad $\widetilde{T}$ on $\cC_{/x}$. The category of algebras for the monad $\widetilde{T}$ is then equivalent to the category of $T$-algebras over $x$.
\end{lemma}
\begin{proof}
First assume we have an endofunctor $F:\cC\rightarrow\cC$, then it is easy to see that lifting $F$ to an endofunctor of $\cC_{/x}$ is equivalent to giving a morphism $m:Fx\rightarrow x$. Indeed, this is certainly necessary, and conversely, given a morphism like this we define $F(y\xrightarrow{f} x)\bydef Fy\xrightarrow{Ff} Fx\xrightarrow{m} x$. Similarly, assume we have a natural transformation $\alpha:F\rightarrow G$ where both $F$ and $G$ lift to $\cC_{/x}$. Then $\alpha$ lifts to $\cC_{/x}$ if and only if the following diagram commutes 
\[
\begin{tikzcd}
Fx\arrow[rr,"\alpha_x"]\arrow[dr,"m_F"]&{}&Gx\arrow[dl, "m_G" swap]\\
{}&x
\end{tikzcd}.
\]
Finally, observe that given two endofunctors $F$ and $G$ that both lift to $\cC_{/x}$, the lifting of the composition $G\circ F$ corresponds to $GFx\xrightarrow{Gm_F}Gx\xrightarrow{m_G}x$. Combining this with the fact that a monad on $\cC$ is equivalent by \cref{cor:lax_mon} to a functor of bicategories $\rB\Delta_a\rightarrow \cat$ sending $*\in\rB\Delta_a$ to $\cC$, we see that lifting a monad $T$ to $\cC_{/x}$ is equivalent to giving a morphism $m_x:Tx\rightarrow x$ such that for every $s:[n]\rightarrow[l]$ the following diagram commutes
\[
\begin{tikzcd}
T^n x\arrow[rr,"s_T"]\arrow[dr,"m_x^n"]&{}&T^l x\arrow[dl, "m_x^l" swap]\\
{}&x
\end{tikzcd}.
\]
In other words, this is equivalent to giving $x$ the structure of a $T$-algebra.\par
To prove the second claim, observe that the lifting of $T^n$ to $\cC_{/x}$ sends $f:y\rightarrow x$ to $T^n y\xrightarrow{T^n f}T^n x\xrightarrow{m^n}x$, so giving $f$ the structure of a $T$-algebra is equivalent to giving a morphism $m_y:Ty\rightarrow y$ endowing $y$ with the $T$-algebra structure such that it makes the following diagram commute
\[
\begin{tikzcd}[row sep=huge, column sep=huge]
T^n y\arrow[r, "T^n f"]\arrow[d, "m_y^n"]&T^n x\arrow[d, "m^n"]\\
y\arrow[r, "f"]&x
\end{tikzcd}.
\]
This is manifestly the same data necessary to give $y$ a structure of a $T$-algebra over $x$.
\end{proof}
\begin{prop}\label{prop:rfib}
$\widetilde{M}$ is a Segal space. Moreover, postcomposition with $p_!p^*$ induces an isomorphism between $\mor_{\mathrm{Corr}}(*, \widetilde{M})$ and $\mor_{\mathrm{Corr}}(\cC, \cD)_{/M}$.
\end{prop}
\begin{proof}
First, observe that $F$ is \textit{not} a morphism of algebraic patterns. We will prove that 
\[\widetilde{M}([n])\cong\underset{e\in\mElo{Cat}{[n]}}{\lim}\widetilde{M}(e)\]
by induction on $n$. Observe that $\widetilde{M}([0])\cong M$ and $\widetilde{M}([1])\cong \cC_1\times_{\cC_0}M\times_{\cD_0}\cD_1\cong M\times_{(\cC_0\times \cD_0)}(\cC_1\times\cD_1)$. Assume that the claim is true for $(n-1)$, then we have 
\begin{equation*}
    \begin{split}
        \widetilde{M}([n])\cong& M\times_{(\cC_0\times\cD_0)}(\cC_n\times\cD_n)\\
        \cong&  M\times_{(\cC_0\times\cD_0)}(\cC_{n-1}\times\cD_{n-1})\times_M(M\times_{(\cC_0\times\cD_0)}(\cC_1\times\cD_1))\\
        \cong&\widetilde{M}([n-1])\times_{\widetilde{M}[0]}\widetilde{M}([1]),
    \end{split}
\end{equation*}
so $\widetilde{M}$ is indeed a Segal object.\par
To prove the last statement we will use \cref{prop:corr_mor}. Observe that $\fib(*,\widetilde{M})^\el\cong\cS_{/M}\cong (\cS_{/\cC_0\times \cD_0})_{/M}$. Moreover, the free $\widetilde{M}$-fibration monad is given by 
\[K\mapsto K\times_M M\times_{\cC_0\times \cD_0} (\cC_1\times \cD_1)\cong K\times_{\cC_0\times \cD_0} (\cC_1\times \cD_1).\] 
So in particular it is a lifting of the free $(\cC\times\cD^\op)$-fibration monad to $\fib(\cC,\cD)_{/M}^\el$, so the last claim of the proposition now follows from \cref{lem:mon_slice}.
\end{proof}
\begin{prop}
If we denote by $p_1:\widetilde{M}\rightarrow\cC$ and $p_2:\widetilde{M}\rightarrow\cD$ the natural projections, there is a natural isomorphism
\[M\cong p_{2,!}\circ p_1^*\]
\end{prop}
\begin{proof}
Observe that by definition we have
\begin{equation*}
    \begin{split}
        p_{2,!}\circ p_1^*\cong &\underset{[n]\in\Delta^\op}{\colim}\cC_1\times_{\cC_0} M\times_M M_n\times_m M\times_{\cD_0} \cD_1\\
        \cong &\underset{[n]\in\Delta^\op}{\colim} \cC_1\times_{\cC_0}\cC_n\times_{\cC_0}M\times_{\cD_0}\cD_n\times_{\cD_0}\cD_1\\
        \cong & \underset{[n]\in\Delta_{-\infty}^\op}{\colim}\cC_n\times_{\cC_0} M\times_{\cD_0}\cD_n\\
        \cong &M,
    \end{split}
\end{equation*}
where the isomorphisms follow as in \cref{prop:fact_1}.
\end{proof}
\begin{defn}
Denote by $c_1$ a Segal space with two objects and an invertible morphism between them. We call a Segal space $\cC$ \textit{complete} if any morphism $c_1\rightarrow\cC$ factors through $e:\cC_0\hookrightarrow\cC$. We denote the full subcategory of $\mathrm{Cat}$ on complete Segal spaces by $\cat^\comp$ and the full subcategory of $\corr$ on complete Segal spaces by $\corr^\comp$.
\end{defn}
\begin{prop}\label{prop:compl}
The natural inclusion $i^{\mathrm{comp}}:\mathrm{Cat^{comp}}\hookrightarrow\mathrm{Cat}$ admits a left adjoint $(\widehat{-}):\mathrm{Cat}\rightarrow\mathrm{Cat^{comp}}$. Moreover, for any Segal space $\cC$ we have $\cC_1\cong\cC_0\times_{\widehat{\cC_0}}\widehat{\cC_1}\times_{\widehat{\cC_0}}\cC_0$.
\end{prop}
\begin{proof}
This is a combination of \cite[Lemma 1.9]{ayala2018flagged} and \cite[Lemma 1.14]{ayala2018flagged}.
\end{proof}
\begin{prop}\label{prop:univ}
The natural morphism $p:\cC\rightarrow\widehat{\cC}$ induces an isomorphism $p_!$ in $\corr$.
\end{prop}
\begin{proof}
We will prove that $p^*$ constitutes an inverse to $p_!$. We know from the proof of \cref{prop:adj} that 
\[p^*p_!\cong\cC_0\times_{\widehat{\cC_0}}\widehat{\cC_1}\times_{\widehat{\cC_0}}\cC_0\cong \cC_1\]
where the last isomorphism follows from \cref{prop:compl}. Similarly, we have 
\begin{equation*}
    \begin{split}
        p_!p^*\cong&\underset{[n]\in\Delta^\op}{\colim}(\widehat{\cC}_1\times_{\widehat{\cC}_0}\cC_n)\times_{\widehat{\cC}_0}\widehat{\cC}_1\\
        \cong&\underset{[n]\in\Delta^\op}{\colim}(\widehat{\cC}_0\times_{\cC_0}\cC_{n+1})\times_{\widehat{\cC}_0}\widehat{\cC}_1\\
        \cong&\underset{[m]\in\Delta_{-\infty}^\op}{\colim}(\widehat{\cC}_0\times_{\cC_0}\cC_{m})\times_{\widehat{\cC}_0}\widehat{\cC}_1\\
        \cong&\widehat{\cC}_0\times_{\widehat{\cC}_0}\widehat{\cC}_1\cong\widehat{\cC}_1
    \end{split}
\end{equation*}
where we once again used the isomorphism of \cref{prop:compl}.
\end{proof}
\begin{cor}\label{cor:comp_corr}
The functor $(\widehat{-}):\cat\rightarrow\cat^\comp$ extends to a functor $(\widehat{-}):\corr\rightarrow\corr^\comp$.
\end{cor}
\begin{proof}
Given a correspondence $M\cC\nrightarrow\cD$ we define 
\[\widehat{M}\bydef p_{\cD,!}\circ M\circ p^*_\cC :\widehat{\cC}\nrightarrow\widehat{\cD}.\]
Since by \cref{prop:univ} we have $p^*p_!\cong\id$ and $p_!p^*\cong \id$, we see that this indeed defines a functor.
\end{proof}
\begin{cor}\label{cor:corr_comp}
Given two Segal spaces $\cC$ and $\cD$ we have 
\[\mor_{\corr}(\cC,\cD)\cong\cP(\widehat{\cD}^\op\times \widehat{\cC})\]
\end{cor}
\begin{proof}
We know from \cref{prop:univ} that $\mor_{\corr}(\cC,\cD)\cong\mor_{\corr}(\widehat{\cC},\widehat{\cD})$, so it suffices to prove it for complete Segal spaces, in which case it follows from \cite[Lemma 4.1]{ayala2017fibrations}.
\end{proof}
\begin{prop}\label{prop:monadicity}
Assume that $F:\cC\rightarrow\cD$ is a morphism of Segal spaces such that the induced morphism $F_0:\cC_0\rightarrow\cD_0$ is an effective epimorphism. Then $F_!:\corr(*,\cC)\leftrightarrows\corr(*,\cD):F^*$ is a monadic adjunction.
\end{prop}
\begin{proof}
It follows from \cref{prop:univ} that it suffices to prove the statement fro $\widehat{F}:\widehat{\cC}\rightarrow\widehat{\cD}$. In this case it would follow from \cite[Proposition 8.1.]{chu2019homotopy} once we prove that $\widehat{F}$ induces an effective epimorphism on the space of objects. Observe that we have the following diagram
\[
\begin{tikzcd}[row sep=huge, column sep=huge]
\cC_0\arrow[r]\arrow[d,"F"]&\widehat{\cC}_0\arrow[d,"\widehat{F}"]\\
\cD_0\arrow[r]&\widehat{\cD}_0
\end{tikzcd}
\]
In this diagram all morphisms except possibly the right vertical map are effective epimorphisms, so it now follows from the two-out-of-three property \cite[Corollary 6.2.3.12.]{lurie2009higher} that $\widehat{F}$ is also an effective epimorphism.
\end{proof}
\begin{construction}\label{constr:n_string}
For a given $n\geq0$ denote by $\ccat^\inj_{/[n]}$ the full subcategory of $\ccat_{/[n]}$ on injective morphisms $g:[m]\rightarrow[n]$ and by $\ccat^{\act, \inj}_{/[n]}$ the subcategory of $\ccat^\inj_{/[n]}$ only containing active morphisms. Assume that we have a Segal space $(F:\cD\rightarrow[n])\in\cat_{/[n]}$, denote $\cD_i\bydef F^{-1}(i)$. Also for $i\leq j\leq k\leq n$ denote by $\cD_{i,j}$ the pullback $[1]\times_{[n]}\cD$, where $[1]\rightarrow[n]$ is the morphism sending $0$ to $j$ and $1$ to $k$.  We can view this as a correspondence from $\cD_j$ to $\cD_k$.  Denote by $\ccat^{\infty,\inj}_{/[n]}$ the subcategory of $\ccat^\inj_{/[n]}$ containing the same objects and only those morphisms that preserve the maximal element, observe that $\ccat^{\act,\inj}_{/[n]}$ is a subcategory of $\ccat^{\infty,\inj}_{/[n]}$.\par
Now fix an object $\cF\bydef(\cF_0,\cF_1,...,\cF_n)\in\mor_\corr(\coprod_{i=0}^n\cD_i,*)$. Denote by $\corr_n^\cF$ the category whose objects are given by injective morphisms $g:[m]\rightarrow[n]$ and a morphism from $f:[l]\rightarrow[n]$ to $g$ is given by a morphism $h$ as in the diagram below
\[
\begin{tikzcd}
{[m]}\arrow[dr, "g"]\arrow[rr, "h"]&{}&{[l]}\arrow[dl, "f" swap]\\
{}&{[n]}
\end{tikzcd}
\]
together with the data of morphisms \[\alpha_{h,i}:\cD_{f(h(i+1)-1),g(i+1)}\circ\cD_{f(h(i+1)-2),f(h(i+1)-1)}\circ...\circ\cD_{g(i),f(h(i)+1)}\rightarrow\cD_{g(i),g(i+1)}\]
for $0\leq i\leq m$ in $\mor_\corr(\cD_{g(i)},\cD_{g(i+1)})$ and a morphism \[\alpha_h:\cF_{f(l)}\circ \cD_{f(l)-1,f(l)}\circ...\circ\cD_{g(m),g(m)+1}\rightarrow\cF_{g(m)}\]
in $\mor_{\corr}(\cD_{g(m)},*)$ with the obvious operation of composition.\par
There is a natural forgetful functor $p_\cF:\corr^\cF_n\rightarrow\ccat^\inj_{/[n]}$. Observe that its pullback $p_\cF^{-1}(\ccat^{\infty,\inj}_{/[n]})\rightarrow\ccat^{\infty,\inj}_{/[n]}$ along the natural inclusion $\ccat^{\infty,\inj}_{/[n]}\hookrightarrow\ccat^{\inj}_{/[n]}$ admits a section 
\[s_{\cF,0}:\ccat^{\infty,\inj}_{/[n]}\rightarrow p^{-1}(\ccat^{\infty,\inj}_{/[n]})\]
that is identity on objects and sends every morphism $h$ as above to the morphism in $\corr_n^\cF$ for which $\alpha_h:\cF_{g(m)}\rightarrow\cF_{g(m)}$ is the identity and every $\alpha_{h,i}$ is induced by the composition maps. Now we define the category $M_n$ to have as objects pairs $(\cF,s)$ where $s$ is a section of $p_\cF$ that restricts to $s_{\cF,0}$ on the subcategory $\ccat^{\infty,\inj}_{/[n]}$. To define morphisms in $M_n$, observe that for any morphism $F:\cF\rightarrow\cE$ in $\mor_\corr(\coprod_{i=0}^n\cD_i,*)$ we can define the category $\corr^F_n$ exactly as above except we require the morphisms $\alpha_h$ to be given by the commutative diagrams of the form
\[
\begin{tikzcd}[row sep=huge, column sep=huge]
\cF_{f(l)}\circ \cD_{f(l)-1,f(l)}\circ...\circ\cD_{g(m),g(m)+1}\arrow[r,"\alpha'_h"]\arrow[d, "F_{f(l)}"]&\cF_{g(m)}\arrow[d, "F_{g(m)}"]\\
\cE_{f(l)}\circ \cD_{f(l)-1,f(l)}\circ...\circ\cD_{g(m),g(m)+1}\arrow[r,"\alpha''_h"]&\cE_{g(m)}
\end{tikzcd}.
\]
Observe that there are forgetful functors $p_1:\corr^F_n\rightarrow\corr^\cF_n$ and $p_2:\corr^F_n\rightarrow\corr^\cE_n$. We define the space of morphisms from $(\cF,s_\cF)$ to $(\cE,s_\cE)$ is given by the space of sections of $p_F:\corr^F_n\rightarrow\ccat^\inj_{/[n]}$ that extend $s_{F,0}$ and such that $p_1\circ s_F\cong s_\cF$ and $p_2\circ s_F\cong s_\cE$.
\end{construction}
\begin{construction}\label{constr:monad_alg}
 Assume that we are given a Segal space $\cC$ with a monad $T$ on it and an object $x\in\cC$. Then we can define a Segal space $\cC^T_x$ whose objects are given by numbers $n\geq0$ and such that a morphism from $l$ to $m$ is given by the data of a morphism $h:[m]\rightarrow[l]$ together  with natural transformations $\alpha_{h,i}:T^{h(i+1)-h(i)}\rightarrow T$ and a morphism $\alpha_h:T^{l-h(m)}x\rightarrow x$. Denote by $p_x:\cC^T_x\rightarrow\Delta^\op$, observe that there is a natural section 
 \[s_{x,0}:(\Delta^{\infty})^\op\rightarrow p_x^{-1}((\Delta^{\infty})^\op)\]
 sending every morphism $h$ to the morphism in $\cC^T_x$ for which $\alpha_x\cong\id$ and all $\alpha_{h,i}$ are given by the structural morphisms of $T$.
\end{construction}
\begin{lemma}\label{lem:monad_alg}
Giving an object $x\in\cC$, giving it a structure of a $T$-algebra is equivalent to extending the section $s_{x,0}$ to a section $s_x:\Delta^\op\rightarrow\cC^T_x$ in the notation of \cref{constr:monad_alg}. A monad $T$ on $\cC$ induces a monad $T_\arr$ on $\arr(\cC)$ sending every $f:x\rightarrow y$ to $Tf:Tx\rightarrow Ty$ with the structural morphisms given by the naturality squares for the respective morphisms of $T$. If both $x$ and $y$ are $T$-algebras, then giving $f$ a structure of a morphism of $T$-algebras is equivalent to giving a section $s_f:\Delta^\op\rightarrow\arr(\cC)_f^{T_\arr}$ that restricts to $s_x$ and $s_y$ on $\cC^T_x$ and $\cC^T_y$ respectively.
\end{lemma}
\begin{proof}
By definition, a monad is an algebra in the monoidal category $\edm(\cC)$. For every Segal space $\cD$ the category $\mor_\cat(\cD,\cC)$ is a right module over $\edm(\cC)$, where the action is given by postcomposition with an endomorphism. In particular, this applies to $\mor_\cat(*,\cC)\cong\cC$. Using \cite[Proposition 4.2.2.9.]{luriehigher} and \cite[4.2.2.12.]{luriehigher} and untangling the definitions, we see that this structure can be encoded as a coCartesian fibration $M\rightarrow\Delta^\op\times[1]$ such that $M([n],0)\cong\edm(\cC)^{\times n}\times \cC$, $M([n],1)\cong \edm(\cC)^{\times n}$, the morphisms $([n],0)\rightarrow([n],1)$ are sent to the natural projection $\edm(\cC)^{\times n}\times \cC\rightarrow\edm(\cC)^{\times n}$, the restrictions to $\Delta\times\{1\}$ and $\Delta^{\act,\op}\times\{0\}$ are given by the structure of the monoidal category on $\edm(\cC)$, inert morphisms $i:[m]\rightarrowtail[n]$ in $\Delta^{\op}\times\{0\}$ that preserve the maximal element are given by projections $\edm^{\times m}\times\cC\rightarrow\edm^{\times n}\times \cC$. Giving a monad $T$ with an algebra $x$ is then equivalent to giving the section $s$ of the natural projection $p:M\rightarrow\Delta^\op\times[1]$. Now observe that this section is in fact uniquely determined by its restriction to $\Delta^\op\times\{0\}$. Moreover, if we are given a monad $T$ and an underlying object $x$ of the prospective $T$-algebra, we can in fact restrict to considering sections of the natural projection from the full subcategory of $M$ on $T$ and $x$. It now follows by easy inspection that this space of sections is in fact the same as the one described in the statement of the lemma.\par
To prove the second claim we first observe that by definition the morphisms between $T$-algebras $x$ and $y$ are given by the natural transformations between sections $s_x$ and $s_y$ that restrict to identity on $\edm(\cC)^{\times n}$. For a pair of functors $F_0,F_1:\cD\rightarrow\cE$ over a give category $\cC$, the space of natural transformations between them is the space of morphisms $\cD\rightarrow\arr_\cC(\cE)$ that restrict to $F_1$ and $F_1$, where $\arr_\cC(\cE)$ is given by the following pullback square
\[
\begin{tikzcd}[row sep=huge, column sep=huge]
{\arr_\cC(\cE)}\arrow[d]\arrow[r]&{\arr(\cE)}\arrow[d, "(s{,}t)"]\\
\cC\arrow[r, "\Delta"]&\cC\times\cC
\end{tikzcd}.
\]
It is easy to see that if $\cE\rightarrow\cC$ is a coCartesian fibration, then so is $\arr_\cC(\cE)$ and its fiber over $c\in\cC$ is given by $\arr(\cE_c)$. Applying it to $p:M\rightarrow \Delta^\op\times[1]$ we see that $\arr(\cC)$ also obtains the natural action of $\edm(\cC)$. It is easy to see that this natural action is in fact given by \[\edm(\cC)\times\arr(\cC)\xrightarrow{\Delta}\edm(\arr(\cC))\times\arr(\cC)\rightarrow\arr(\cC)\]
where $\Delta$ is a morphism sending an endomorphism $G$ of $\cC$ to an endomorphism of $\arr(\cC)$ that sends $c\xrightarrow[f]c'$ to $Gc\xrightarrow{Gf}Gc'$. Putting everything together, we see that extending a given morphism $f:x\rightarrow y$ to a morphism of $T$-algebras is equivalent to giving a section $s_f:\Delta^\op\rightarrow \arr(\cC)^{T_\arr}_f$ satisfying the properties described in the claim.
\end{proof}
\begin{prop}\label{prop:nat}
Given a Segal space $(F:\cD\rightarrow[n])$ as in \cref{constr:n_string}, there is a natural isomorphism
\[\mor_\corr(\cD,*)\cong M_n\]
in the notation of \cref{constr:n_string}.
\end{prop}
\begin{proof}
Consider the natural inclusion $j:\coprod_{k=0}^n\cD_k\hookrightarrow \cD$. It is an isomorphism on the space of objects, so by \cref{prop:monadicity} the category $\mor_{\corr}(\cD,*)$ is monadic over $\mor_\corr(\coprod_{j=0}^n\cD_j,*)$. It is easy to see that $\mor_\corr(\coprod_{k=0}^n\cD_k,*)\cong\prod_{k=0}^n\mor_\corr(\cD_k,*)$. Denote by $\ccat^\act_{/[n]}$ (resp. $\ccat^{\infty}_{/[n]}$) the subcategory of $\ccat_{/[n]}$ containing only active (resp. maximal element preserving) morphisms over $[n]$ and by $\widetilde{M}_n$ the category defined exactly like $M_n$ but with $\ccat^{\act,\inj}_{/[n]}$ and $\ccat^{\infty,\inj}_{/[n]}$ replaced with $\ccat^\act_{/[n]}$ and $\ccat^{\infty}_{/[n]}$. Using the calculations of \cref{prop:adj} we see that
\[j^*j_!\cong\coprod_{k=0}^n\cD_{k,0}\times_{\cD_1}\coprod_{k=0}^n\cD_{k,0}\cong\coprod_{f:[1]\rightarrow[n]}\cD_{f(0),f(1)}.\]
Iterating this calculation we see that the $i$th component of $(j^*j_!)^m\cF$ is given by
\[\coprod_{f:[m+1]\rightarrow[n],f(0)=i}\cF_{f(m+1)}\circ\cD_{f(m),f(m+1)}\circ \cD_{f(m-1),f(m)}\circ...\circ \cD_{i, f(1)}\]
and the unit and multiplication morphisms are given by composition and identity morphisms in $\ccat^\act_{/[n]}$.\par
Observe that the left fibration over $\Delta^{\op}$ corresponding to a functor sending $[m]$ to $\mor_\cat([m],[n])$ is given by the natural projection $p:\ccat_{/[n]}\rightarrow\ccat$. It follows from this and \cref{lem:monad_alg} that the category of $j^*j_!$-algebras is isomorphic to $\widetilde{M}_n$. Now it remains to prove that the sections $s_\cF:\ccat_{/[n]}\rightarrow\corr^\cF_n$ in $\widetilde{M}_n$ are uniquely defined by their restriction to $\ccat^{\inj}_{/[n]}$, however this is obvious from the existence of surjective/injective factorization and the fact that all surjective morphisms in $\ccat^\act_{/[n]}$ are sent to identities.
\end{proof}
\begin{remark}\label{rem:nat}
Applying \cref{prop:nat} to the special case $n=1$ we see that given a correspondence $M:\cC\nrightarrow\cD$, the category $\mor_\corr(M,*)$ (where here we denote by $M$ the total category of the correspondence) is isomorphic to the category of triples $(\cF_\cC,\cF_\cD, \alpha:\cF_\cD\circ M\rightarrow\cF_\cC)$ where $\cF_\cC\in\mor_\corr(\cC,*)$ and $\cF_\cD\in\mor_\corr(\cD,*)$.
\end{remark}
\begin{cor}\label{cor:nat}
There is a natural equivalence between that space of natural transformations $\alpha:M\rightarrow N$ where $M$ and $N$ are correspondences from $\cC$ to $\cD$ and the
space of correspondences $\overline{\alpha}\in\corr([1]\times\cC,\cD)$ restricting to $M$ and $N$ on $0$ and $1$ respectively.
\end{cor}
\begin{proof}
Observe that $[1]\times \cC$ is equivalently the correspondence given by $\id_{\cC,!}$. The result now follows from \cref{rem:nat} and \cref{cor:dual}. For future use we will also explicate the construction. Assume that we are given a correspondence $\cC\times[1]\nrightarrow\cD$ restricting to $M$ and $N$ on the ends of the interval. Observe that the objects of $\fib([1]\times\cC,\cD)$ lying over $([m],[n])\in\Delta^\op\times\Delta^\op$ are given by sequences \[\widehat{c}\bydef(c_1\rightarrow...\rightarrow c_{m'}\rightarrow c\rightarrow c'\rightarrow c'_1\rightarrow...\rightarrow c'_{m''},d_1\rightarrow...\rightarrow d_n)\] for all $m'$ and $m''$ such that $m'+m''+1=m$. In particular, the morphisms $i_{1,2}:\fib(\cC,\cD)\rightarrow\fib(\cC\times[1],\cD)$ corresponding to the natural inclusions $i_{1,2}:\cC\hookrightarrow\cC\times [1]$ are given by sending an object \[(c_1\rightarrow...\rightarrow c_m,d_1\rightarrow...\rightarrow d_n)\in\fib(\cC,\cD)\]
to the same string of morphisms viewed as an element of $\fib(\cC\times[1],\cD_n)$ for which $m'=0$ (resp. $m''=0$). To get the value of the corresponding natural transformation $\alpha$ on $(c_1\rightarrow...\rightarrow c_m,d_1\rightarrow...\rightarrow d_n)\in\fib(\cC,\cD)$, consider the objects \[\overline{c}\bydef(c_1\rightarrow...\rightarrow c_m,d_1\rightarrow...\rightarrow d_n)\]
with $m'=m-1$ and $m''=0$ and 
\[\widetilde{c}\bydef(c_1\rightarrow...\rightarrow c_m=\joinrel=c_m,d_1\rightarrow...\rightarrow d_n)\] 
with $m'=m-1$ and $m''=1$. Observe that there is a natural morphism $\widetilde{c}\rightarrow\overline{c}$ ling over $([m],[n])\xrightarrow{(i_m,\id)}([m+1],[n])$ where $i_m:[m]\rightarrowtail[m+1]$ sends $[m]$ isomorphically into the initial subinterval of $[m+1]$. This induces the morphism \[\alpha:\cC_m\times_{\cC_0}M\times_{\cC_0}\cD_n\rightarrow\cC_m\times_{\cC_0}N\times_{\cC_0}\cD_n\]
which we declare the value of the required natural transformation.\par
Conversely, assume that we are given a natural transformation $\alpha:M\rightarrow N$. We define the corresponding functor $\fib(\cC\times[1],\cD)\rightarrow\cS$ by first sending every object of the form of $\widehat{c}$ above with $m''>0$ to $M$ and objects as with $m''=0$ to $N$. Then we need to define the action of morphisms. First, observe that we only need to define the action of morphisms whose image in $\Delta^\op\times\Delta^\op$ does not belong to $\Delta^\op_\infty\times\Delta^\op$. So assume we are given an object $\widehat{c}$ as above and an inert morphism $([m_1],[n])\overset{(i,\id)}{\rightarrowtail}([m],[n])$ such that $i(m_1)\bydef m_2<m'$, then we define the corresponding morphism as a composition
\[\cC_m\times_{\cC_0}M\times_{\cC_0}\cD_n\overset{(j_{m'},\id)}{\rightarrowtail}\cC_{m'+1}\times_{\cC_0}M\times_{\cC_0}\cD_n\xrightarrow{\alpha}\cC_{m'+1}\times_{\cC_0}N\times_{\cC_0}\cD_n\overset{(j_{m_1},\id)}{\rightarrowtail}\cC_{m_1}\times_{\cC_0}N\times_{\cC_0}\cD_n\]
where $j_{m'}:[m'+1]\rightarrowtail[m]$ is the inclusion of $[m'+1]$ as the initial segment of $[m]$ with the corresponding morphism \[\cC_m\times_{\cC_0}M\times_{\cC_0}\cD_n\overset{(j_{m'},\id)}{\rightarrowtail}\cC_{m'+1}\times_{\cC_0}M\times_{\cC_0}\cD_n\]
given by the structure of the correspondence on $M$ and $j_{m_1};[m_1]\rightarrowtail[m'+1]$ is such that $j_{m_1}\circ j_{m'}\cong i$ with the corresponding morphism \[\cC_{m'+1}\times_{\cC_0}N\times_{\cC_0}\cD_n\overset{(j_{m_1},\id)}{\rightarrowtail}\cC_{m_1}\times_{\cC_0}N\times_{\cC_0}\cD_n\]
given by the structure of the correspondence on $N$. 
\end{proof}
\begin{lemma}\label{lem:nat_1}
Given an $\infty$-category $\cC$ together with two morphisms $\cF:\cC\rightarrow\cS$ and $\cE:\cC\rightarrow\cS$, denote by $p_1:\widetilde{\cE}\rightarrow\cC$ (resp. $p_2:\widetilde{\cF}\rightarrow\cS$) the left fibration over $\cC$ corresponding to $\cE$ (resp. to $\cF$) and consider a functor $\overline{\cF}:\widetilde{\cE}\rightarrow\cS$ sending $e\in\widetilde{\cE}$ to $\cF(p(e))$. Then the space of natural transformation $\nat(\cE,\cF)$ is isomorphic to the limit of $\overline{\cF}$.
\end{lemma}
\begin{proof}
First, observe that the left fibration corresponding to $\overline{\cF}$ is $\overline{p}:\widetilde{\cF}\times\widetilde{\cE}\rightarrow\widetilde{\cE}$. By \cite[Corollary 3.3.3.4.]{lurie2009higher} the limit of $\overline{\cF}$ is given by the space of coCartesian sections of $\overline{p}$, which is easily seen to be the space of morphisms $\widetilde{\cE}\rightarrow\widetilde{\cF}$ taking coCartesian edges to coCartesian edges, which is the same as $\nat(\cE,\cF)$.
\end{proof}
\begin{lemma}\label{lem:nat_2}
Given two functors $F,G:\cC\rightarrow\cD$ between Segal spaces, consider the correspondence $\id_\cC$ viewed as a left fibration over $\fib(\cC,\cC)$. Define a functor $\widetilde{G}:\id_\cC\rightarrow\cS$ that sends \[(\overline{c_1}\rightarrow\overline{c_2}\rightarrow...\rightarrow\overline{c_k}\rightarrow c\rightarrow c'\rightarrow\widehat{c_1}\rightarrow...\rightarrow\widehat{c_t})\]
to $\mor_\cD(F(c),G(c'))$. Then the space of natural transformations $\nat(F,G)$ is isomorphic to the limit of this functor.
\end{lemma}
\begin{proof}
A natural transformation between $F$ and $G$ is a functor $[1]\times\cC\rightarrow\cD$ that restricts to $F$ on $\cC\times\{0\}$ and to $G$ on $\cC\times\{1\}$. Using \cref{lem:nat_1}, we can view the space of morphisms $[1]\times\cC\rightarrow\cD$ as the limit over $(\cC\times[1])$ (viewed as a left fibration over $\Delta^\op$) of a functor  $F_\cD:(\cC\times[1])\rightarrow\cS$ sending 
\[\widetilde{c}\bydef(c_1\xrightarrow{f_1} c_2\xrightarrow{f_2}...\xrightarrow{f_{n-1}} c_n)\in(\cC\times[1])_n\]
to $\cD_n$. The space of natural transformations is then given by the limit of $\widetilde{F_\cD}\subset F_\cD$ which sends $\widetilde{c}$ as above to the space of sequences $(d_1\xrightarrow{g_1}d_2\xrightarrow{g_2}...\xrightarrow{g_{n-1}}d_n)$ for which $d_i\cong F(c_i)$ (resp. $G(c_i)$) and $g_j\cong F(f_j)$ (resp. $G(f_j)$) if $c_i$ and $f_j$ are in the image of $\cC\times\{0\}$ (resp. $\cC\times\{1\}$). Observe that since $\cC\times[1]$ is the total category of the identity correspondence, we have $(\cC\times[1])\cong p_!\id_\cC$ where $p:(\Delta^\op\times\Delta^\op)\rightarrow\Delta^\op$ sends $([n],[m])$ to $[n+m+1]$ and $\id_\cC(n,m)\cong \cC_n\times_{\cC_0}\cC_1\times_{\cC_0}\cC_m$. In other words, there is an isomorphism
\[\coprod_{k+t=n-1}\cC_k\times_{\cC_0}\cC_1\times_{\cC_0}\cC_t\cong (\cC\times[1])_n.\]
Finally, observe that the functor
\[\coprod_{k+t=n-1}\cC_k\times_{\cC_0}\cC_1\times_{\cC_0}\cC_t\cong (\cC\times[1])_n\xrightarrow{\widetilde{F_\cD}}\cS\]
sends $(\overline{c_1}\rightarrow\overline{c_2}\rightarrow...\rightarrow\overline{c_k}\rightarrow c\rightarrow c'\rightarrow\widehat{c_1}\rightarrow...\rightarrow\widehat{c_t})$ to $\mor_\cD(F(c),G(c'))$, so the resulting left fibration over $\fib(\cC,\cC)$ is indeed the one described in the claim.
\end{proof}
\begin{prop}\label{prop:corr_bicat}
Given two morphisms $f$ and $g$ from $\cC$ to $\cD$ there is a natural isomorphism
\[\nat(f,g)\cong\mor_{\mor_{\mathrm{Corr}}(\cC,\cD)}(f_!,g_!)\cong\mor_{\mor_{\mathrm{Corr}}(\cC,\cD)}(g^*,f^*)\]
\end{prop}
\begin{proof}
We will prove the proposition for $(-)_!$, the case of $(-)^*$ would then follow from \cref{prop:adj} and the elementary properties of adjunctions. First, observe that by definition $f_!\cong R_1f$ for the functor $R_1$ of \cref{constr:tensor}. It follows that
\[\mor_{\ccat^\cell_{/[1]}}(f_!,g_!)\cong \mor_{\ccat^\cell_{/[1]}}(R_1 f,R_1 g)\cong \mor_{\ccat\times [1]}(L_1 R_1 f,g).\]
Observe that $L_1 R_1 f$ is given by the natural inclusion $i:\cC\rightarrowtail f_!$ of $\cC$ into the total category of the correspondence $f_!$. More explicitly, the value of $L_1 R_1 f$ on $([n],0)$ is given by $\cC_n$ and on $([m],1)$ it is given by \[p_!(f_!)([n])\cong \coprod_{k+t=n-1} \cC_k\times_{\cC_0}\cD_1\times_{\cD_0}\cD_t.\] Denote by $F_g:L_1 R_1 f\rightarrow\cS$ the functor whose limit gives the space $\nat_\corr(f_!,g_!)\bydef \mor_{\mor_{\mathrm{Corr}}(\cC,\cD)}(f_!,g_!)$. By arguing as in the proof of \cref{lem:nat_2}, we see that 
\[F_g(c_1\rightarrow...\rightarrow c_n,0)\cong \{f(c_n)\}\]
and 
\[F_g(d_1\rightarrow...\rightarrow d_t, d_t\rightarrow f(c_1), c_1\rightarrow...\rightarrow c_k)\cong \mor_\cD(d_t, g(c_1)).\] 
Observe that there is a natural morphism $F:\id_\cC\rightarrow L_1 R_1 f$ sending \[(c_1\rightarrow...\rightarrow c_t, c_t\rightarrow c'_1, c'_1\rightarrow...\rightarrow c'_k)\]
to 
\[(f(c_1)\rightarrow...\rightarrow f(c_t), f(c_t)\rightarrow f(c'_1), c'_1\rightarrow...\rightarrow c'_k,1).\] 
It follows by inspection that $\widetilde{g}\cong F_g\circ F$, where $\widetilde{g}:\id_\cC\rightarrow\cS$ is the functor of \cref{lem:nat_2}, and so we have a natural morphism 
\[G:\nat_\corr(f_!,g_!)\cong \underset{L_1 R_1 f}{\lim}F_g\rightarrow\underset{\id_\cC}{\lim}\;\widetilde{g}\cong \nat(f,g).\]
Conversely, given a natural transformation $\alpha:f\rightarrow g$ we can view it as a morphism $\alpha:\cC\times[1]\rightarrow\cD$ and define $H:\nat(f,g)\rightarrow\nat_\corr(f_!,g_!)$ by sending $\alpha$ to $\alpha_!:\cC\times[1]\nrightarrow\cD$ which corresponds to an element of $\nat_\corr(f_!,g_!)$ by \cref{cor:nat}.\par
It is easy to see by untangling the definitions that $G\circ H\cong \id$. Given an element $\beta:f_!\rightarrow g_!$, which we view as a cartesian section of $F_g$, we see that $H\circ G (\beta)$ sends $(d\xrightarrow{h}f(c))$ to 
\[d\xrightarrow{h}f(c)\xrightarrow{\beta(\id_{f(c)})}g(c)\in F_g(h)\cong\mor_{\cD}(d, g(c)).\]
Observe that we have a sequence of morphisms 
\[F_g(f(c)=\joinrel= f(c))\cong F_g(d\xrightarrow{h} f(c),f(c)=\joinrel=f(c))\rightarrow F_g(d\xrightarrow{h} f(c))\]
which in particular shows that for every $\beta\in\underset{L_1 R_1 f}{\lim}F_g$ we have $\beta(h)\cong \beta(\id_{f(c)})\circ h$ proving that indeed $H\circ G\cong \id$.
\end{proof}
\begin{remark}\label{rem:subcat}
\Cref{prop:corr_bicat} combined with \cref{prop:fun_comp} shows that we can view $\cat$ as a faithful subcategory of $\corr$.
\end{remark}
\section{Lax functors}\label{sect:four}
Our goal in this section is to describe the concept of a lax functor between a Segal space $\cC$ and a twofold Segal space $\cB$ and prove their basic properties. In defining the lax functors we follow the approach of \cite{gaitsgory2017study}, however for convenience we first describe in \cref{constr:cart} for a category $\cC$ with a subcategory $\cD$ such that $\cC_0\cong\cD_0$ the pattern $\ccart^\cD(\cC)$ whose Segal spaces correspond to categories over $\cC$ that admit coCartesian lifts of all morphisms in $\cD$. We then use this category to prove in \cref{prop:lax} and \cref{cor:unilax} that the categories of lax and unital lax functors from a given $\cC$ are corepresentable by a twofold Segal spaces $L^\lax\cC$ and $L^{\lax,\un}\cC$ respectively. \par
Finally, in \cref{def:coc} we introduce the concepts of lax cocones and lax colimits that will play a major role in the next section. We believe that our notion of lax colimits generalizes that of \cite{gepner2015lax}, however we do not prove that this is the case. 
\begin{prop}
There is a natural equivalence between the category of two-fold Segal spaces and the category of Cartesian fibrations over $\Delta$ satisfying the Segal condition such that $F([0])$ is a groupoid.
\end{prop}
\begin{proof}
First, observe that the category of Cartesian fibrations over $\Delta$ satisfying the Segal condition can be equivalently described as the category of Segal $\ccat\times\ccat$-spaces in the notation of \cref{constr:tensor}. This is precisely the category of double Segal spaces. The fact that double Segal categories satisfying the conditions of the lemma describe twofold Segal spaces follows from \cite{haugseng2018equivalence}.
\end{proof}
\begin{construction}\label{constr:tricat}
We will denote by $2-\cat$ the Segal space of bicategories and functors between them. We will upgrade it to a flagged tricategory. In order to do this it is enough to endow each $\mor_{2-\cat}(\cA,\cB)$ with the structure of a flagged bicategory in a manner compatible with composition. To do this, first denote for $m\geq 0$ by $\theta_{0,m}$ the singleton category $[0]$ and by $\theta_{1,m}$ the bicategory with two objects $0$ and $1$ such that $\mor_{\theta_{1,m}}(0,0)=\mor_{\theta_{1,m}}(1,1)=\{\id\}$, $\mor_{\theta_{1,m}}(1,0)=\varnothing$ and $\mor_{\theta_{1,m}}(0,1)=[m]$, then define for $n>1$
\[\theta_{n,m}\bydef\overbrace{\theta_{1,m}\coprod_{[0]}\theta_{1,m}\coprod_{[0]}...\coprod_{[0]}\theta_{1,m}}^\text{$n$ times},\]
where in each term of the iterated pushout $[0]$ embeds as the final object of $\theta_{1,m}$ on the left and the initial object of $\theta_{1,m}$ on the right. More explicitly, $\theta_{n,m}$ is a bicategory with $n$ objects such that $\mor_{\theta_{n,m}}(i,j)$ is empty unless $j\geq i$ and in that case has objects given by sequences $(m_1,m_2,...,m_{j-i})$ such that $0\leq m_k\leq m$ with a unique morphism from $(m_1,m_2,...,m_{j-i})$ to $(m'_1,m'_2,...,m'_{j-i})$ if all $m'_k\geq m_k$. Also observe that we have an isomorphism
\[\theta_{n,m}\cong\overbrace{\theta_{n,1}\coprod_{[n]}\theta_{n,1}\coprod_{[n]}...\coprod_{[n]}\theta_{n,1}}^\text{$m$ times},\]
where the inclusions $[n]\hookrightarrow\theta_{n,1}$ are identity on objects and send morphisms in $[n]$ to sequences of the form $(1,1,...,1)$ on the right and $(0,0,...,0)$ on the left. Observe that there is a bicosimplicial object \[\theta_{\bullet,\bullet}:\Delta\times\Delta\rightarrow 2-\cat\]
sending $(i,j)$ to $\theta_{i,j}$. Indeed, using the iterated pushouts above it suffices to only define the action of morphisms $f:e\rightarrow([n],[m])$ where $e\in\{([0],[0]),([1],[0]),([1],[1])\}$. The action on $\theta_{0,0}$ is obvious, to every $([1],[0])\xrightarrow{(ij,k)}([n],[m])$ we assign a morphism $[1]\rightarrow\theta_{n,m}$ sending the unique nontrivial morphism of $[1]$ to $(k,k,...,k)$ viewed as a morphism from $i$ to $j$ in $\theta_{n,m}$ and finally we send every $([1],[1])\xrightarrow{(ij,st)}([n],[m])$ to a morphism $\theta_{1,1}\rightarrow\theta_{n,m}$ sending the unique non-trivial $2$-morphism of $\theta_{1,1}$ to the $2$-morphism $(s,s,...,s)\rightarrow(t,t,...,t)$ in $\mor_{\theta_{n,m}}(i,j)$.\par
Now given twofold Segal spaces $\cA$ and $\cB$ we define a bisimplicial space 
\[\mor_{2-\cat}(\cA,\cB)_{\bullet,\bullet}:([n],[m])\mapsto\mor_{2-\cat}(\cA\times \theta_{n,m},\cB),\]
it follows from the pushout isomorphisms above that it is indeed a twofold Segal space. Given $\alpha:\theta_{1,1}\times\cA\rightarrow\cB$ such that $\alpha_0\cong F$ and $\alpha_1\cong F'$ and $\beta:\theta_{1,1}\times\cB\rightarrow\cC$ such that $\beta_0\cong G$ and $\beta_1\cong G'$ we define their composition to be 
\[\theta_{1,1}\times\cA\xrightarrow{m}\theta_{2,1}\times\cA\xrightarrow{(\widehat{\alpha},\widehat{\beta})}\cC,\]
where $m$ is the morphism $\theta_{1,1}\xrightarrow{(02,01)}\theta_{2,1}$, $\widehat{\alpha}$ is the composition
\[\theta_{1,1}\times\cA\xrightarrow{\alpha}\cB\xrightarrow{G}\cC\]
and $\widehat{\beta}$ is 
\[\theta_{1,1}\times\cA\xrightarrow{F'}\theta_{1,1}\times\cB\xrightarrow{\beta}\cC.\]
\end{construction}
\begin{notation}\label{not_arr}
We will use \cref{constr:tricat} as the default tricategory structure on $2-\cat$, however in some cases we will need a slightly different notion. To describe it, we will first define a bicategory $\theta_{t,n}\otimes \theta_{l,m}$ for objects $\theta_{l,m}$ and $\theta_{t,n}$ of \cref{constr:tricat}. The objects of $\theta_{t,n}\otimes \theta_{l,m}$ are given by pairs $(i,j)$ where $0\leq i\leq t$ and $0\leq j\leq l$ and 
\begin{gather*}
    \mor_{\theta_{t,n}\otimes \theta_{l,m}}((i,j),(i',j'))\bydef
    \begin{cases}
        \varnothing &\text{      if $i'<i$ or $j'<j$}\\
        \arr_{\theta_{l,m}}^{i'-i,\lax}(j,j')&\text{       otherwise}
    \end{cases}
\end{gather*}
 where $\arr_{\theta_{l,m}}^{i'-i,\lax}(j,j')$ is a category with objects given by pairs $(f,s)$, where $f:i\rightarrow i'$ is a morphism in $\theta_{t,n}$ and $s$ is a sequence
 \[j\xrightarrow{s_1} j_1\xrightarrow{s_2}...\xrightarrow{s_{i'-i}} j_{i'-i}\xrightarrow{s_{i'-i+1}} j'\]
 in $\theta_{l,m}$. Observe that by definition every morphism from $j$ to $j'$ in $\theta_{l,m}$ is of the from $s_{j'-j}\circ...\circ s_1$, where each $s_k$ is a morphism from $j+k-1$ to $j+k$. We can thus identify an object of $\arr_{\theta_{l,m}}^{i'-i,\lax}(j,j')$ with the data of a morphism $f:i\rightarrow i'$ in $\theta_{t,n}$, a morphism $s:j\rightarrow j'$ in $\theta_{l,m}$ and $(i'-i)$ marked points $j\leq j_1\leq j_2\leq...\leq j_{i'-i}\leq j'$ in $[j'-j]$. Using this description, we define the space of morphisms from $(f,s,j\leq j_1\leq...\leq j')$ to $(\widetilde{f},\widetilde{s},l\leq\widetilde{j_1}\leq...\leq j')$ to be empty unless $j_k\leq\widetilde{j_k}$ for all $k$ in which case it is given by \[\mor_{\mor_{\theta_{t,n}}(i,i')}(f,f')\times\mor_{\mor_{\theta_{l,m}}(j,j')}(s,s').\]
 The composition of 
 \[(f:i\rightarrow i', s:j\rightarrow j',j\leq j_1\leq...\leq j_{i'-i}\leq j)\]
 and 
 \[(f':i'\rightarrow i'', s':j'\rightarrow j'', j'\leq j'_1\leq...\leq j'_{i''-i'}\leq j'')\]
 is given by 
 \[f'\circ f: i\rightarrow i'',s'\circ s:j\rightarrow j'', j\leq j_1\leq...\leq j_{i'-i}\leq j'_1\leq...\leq j'_{i''-i'}\leq j''.\]
 It is obvious that every morphism $F:\theta_{l_1,m_1}\rightarrow\theta_{l_2,m_2}$ induces morphisms $\arr_{\theta_{l_1,m_1}}^{i'-i}(j,j')\rightarrow\arr_{\theta_{l_2,m_2}}^{i'-i}(f(j),f(j'))$, so we can consider $(\theta_{l_1,m_1}\mapsto\theta_{t,n}\otimes\theta_{l_1,m_1})$ as a functor from $\Delta\times\Delta$ to $\cat_2$.
\end{notation}
\begin{lemma}\label{lem:lax_1}
There is a natural isomorphism
\[\theta_{t,n}\otimes\theta_{l,m}\cong\theta_{l,m}\otimes\theta_{t,n}.\]
\end{lemma}
\begin{proof}
The description of objects and morphisms of $\theta_{t,n}\otimes\theta_{l,m}$ is symmetric in $\theta_{l,m}$ and $\theta_{t,n}$ except possibly for the markings, so to prove the claim it suffices to establish for any $k,p\geq0$ a bijection between $p$-markings of $[k]$ and $k$-markings of $[p]$. Assume we are given $0\leq a_1\leq...\leq a_p\leq k$, then for any $j\in[k]$ define $b_j\in[p]$ to be the maximal $q$ for which $a_q<j$. We can apply the same construction to a $k$ marking of $[p]$, the fact that they are inverse to each other amounts to the statement that $a_r$ can be characterized as the maximal element $i$ of $[k]$ such that $a_{r-1}\leq i\leq a_r$ which is obvious.
\end{proof}
\begin{lemma}\label{lem:lax_2}
There are natural isomorphisms
\begin{align}\label{eq:colim_one}
    \theta_{t,n}\otimes\theta_{l,m}\cong\overbrace{(\theta_{t,n}\otimes\theta_{1,m})\coprod_{\theta_{t,n}\otimes\{1\}}(\theta_{t,n}\otimes\theta_{1,m})\coprod_{\theta_{t,n}\otimes\{2\}}...\coprod_{\theta_{t,n}\otimes\{n-1\}}(\theta_{t,n}\otimes\theta_{1,m})}^\text{$l$ times}
\end{align}
 and 
 \begin{align}\label{eq:colim_two}
     \theta_{t,n}\otimes\theta_{1,m}\cong\overbrace{(\theta_{t,n}\otimes\theta_{1,1})\coprod_{\theta_{t,n}\otimes[1]}(\theta_{t,n}\otimes\theta_{1,1})\coprod_{\theta_{t,n}\otimes[1]}...\coprod_{\theta_{t,n}\otimes[1]}(\theta_{t,n}\otimes\theta_{1,1})}^\text{$m$ times}.
 \end{align}
 \[\]
\end{lemma}
\begin{proof}
First, observe that the sets of objects on both sides of \cref{eq:colim_one} and \cref{eq:colim_two} are isomorphic. An object of $\theta_{t,n}\otimes\theta_{1,m}$ can be identified with a pair $(j,i)$ with $0\leq j\leq t$ and $i\in\{0,1\}$. A morphism from $(j,0)$ to $(j,1)$ is given by the data $(f, j\xrightarrow{s_1}j_1\xrightarrow{s_2}j')$, where $f$ is a morphism from 0 to 1 in $\theta_{1,m}$. It follows that a morphism from $(j,i)$ to $(j',i')$ in 
\[ \overbrace{(\theta_{t,n}\otimes\theta_{1,m})\coprod_{\theta_{t,n}\otimes\{1\}}(\theta_{t,n}\otimes\theta_{1,m})\coprod_{\theta_{t,n}\otimes\{2\}}...\coprod_{\theta_{t,n}\otimes\{n-1\}}(\theta_{t,n}\otimes\theta_{1,m})}^\text{$l$ times}\]
can be identified with the following data:
\[(f:i\rightarrow i',j\xrightarrow{s_1}j_1\xrightarrow{s_2}j_2, j_1\xrightarrow{s_2}j_2\xrightarrow{s_3}j_3,..., j_{i'-i-1}\xrightarrow{s_{i'-i}}j_{i'-i}\xrightarrow{s_{i'-i+1}}j'),\]
which is obviously the same as the data needed to define a morphism in $\theta_{t,n}\otimes\theta_{l,m}$. Similarly, since every 2-morphism in $\mor_{\theta_{t,l}}(j,j')$ is of the form $\alpha_{j'-j}*\alpha_{j'-j-1}*...*\alpha_1$, where $\alpha_k$ is a 2-morphism in $\mor_{\theta_{t,n}}(j+k-1,j+k)$, we see that the spaces of 2-morphisms are also isomorphic. Finally, the isomorphism of \cref{eq:colim_two} follows easily from the description of $\theta_{t,n}\otimes\theta_{1,m}$ given above.
\end{proof}
\begin{defn}
For a twofold Segal space $\cB$ and a pair of non-negative integers $(t,n)$ we can consider a functor $\Delta^\op\rightarrow\Delta^\op\rightarrow\cS$ given by 
\[(l,m)\mapsto \mor_{2-\cat}(\theta_{t,n}\otimes\theta_{l,m},\cB).\]
It follows from \cref{lem:lax_2} that this functor satisfies the Segal condition and so defines a twofold Segal space that we will denote $\arr^{\theta_{t,n},\lax}(\cB)$ and simply $\arr^\lax(\cB)$ for $t=1$ and $n=0$. Moreover, we can consider
\[(t,n)\mapsto\arr^{\theta_{t,n},\lax}(\cB)\]
as a functor $\Delta^\op\times\Delta^\op\rightarrow\cS$ and it once again follows from \cref{lem:lax_1} and \cref{lem:lax_2} that it satisfies the Segal condition. In particular, for any other twofold Segal space $\cE$ we can consider a functor 
\[(n,m)\mapsto\mor_{2-\cat}(\cE,\arr^{\theta_{n,m},\lax}(\cB))\]
and it will be a twofold Segal space as well. We denote it by $\mor_{2-\cat}^\lax(\cE,\cB)$.
\end{defn}
\begin{remark}
The category $[n]\otimes[m]$ is in fact the Gray tensor product of $[n]$ and $[m]$ viewed as $2$-categories. The $\infty$-categorical version of this construction can be found in \cite{gaitsgory2017study}, \cite{haugseng2020fibrational} and \cite{gagna2020gray}.
\end{remark}
\begin{remark}
More explicitly, an object of $\arr^\lax(\cB)$ is an arrow $f:a\rightarrow b$ in $\cB$, a morphism from $f$ to $t$ is given by a diagram
\[
 \begin{tikzcd}[row sep=huge, column sep=huge]
 a\arrow[r, "f"]\arrow[d, "g_1"]&b\arrow[dl, Rightarrow, "\alpha_1"]\arrow[d, "h_1"]\\
 c\arrow[r, "t"]&d
 \end{tikzcd},
 \]
 a 2-morphism between a pair of morphisms $\alpha_{1,2}:f\rightarrow t$ is given by a diagram
 \[
 \begin{tikzcd}[row sep=huge, column sep=huge]
 a\arrow[r, "f"]\arrow[d, "g_2"{name=O2}, bend left=30]\arrow[d, "g_1"{name=O1, swap}, bend right=30]&b\arrow[dl, Rightarrow, "\alpha_{1,2}"]\arrow[d, "h_2"{name=C2}, bend left=30]\arrow[d, "h_1"{name=C1, swap}, bend right=30]\\
 c\arrow[r, "t"]&d\arrow[Rightarrow, from=O1,to=O2, shorten <=1ex, "\gamma_1"]\arrow[Rightarrow, from=C1,to=C2, shorten <=1ex, "\gamma_2"]
 \end{tikzcd}
 \]
 in $\cB$ such that furthermore the following square of 2-morphisms commutes
 \[
 \begin{tikzcd}[row sep=huge, column sep=huge]
 h_1\circ f\arrow[r, "\gamma_2*f"]\arrow[d, "\alpha_1"]&h_2\circ f\arrow[d, "\alpha_2"]\\
 t\circ g_1\arrow[r, "t*\gamma_1"]&t\circ g_2
 \end{tikzcd}.
 \]
\end{remark}
\begin{defn}
Given a Segal space $\cC$ and a twofold Segal space $\cB$, we define a \textit{lax functor} from $\cC$ to $\cB$ to be a morphism from $\cC$ to $\cB$ considered as categories over $\Delta^\op$ that take coCartesian morphisms over inert morphisms in $\Delta^\op$ to coCartesian morphisms. We call a lax functor \textit{unital} if it additionally takes coCartesian morphisms lying over surjective active morphisms to coCartesian morphisms. We denote by $\mor^\lax_\cat(\cC,\cB)$ (resp. by $\mor_\cat^{\lax,\un}(\cC,\cB)$) the space of lax functors (resp. unital lax functors) from $\cC$ to $\cB$ and denote a lax functor by $F:\cC\rightsquigarrow\cB$. We endow this space with the structure of a twofold Segal space by declaring a morphism from $F$ to $G$ to be a morphism $f:\cC\rightsquigarrow \arr^\lax(\cB)$  such that its composition with the natural projection $p_1:\arr^\lax(\cB)\rightarrow \cB$ (resp. with $p_2:\arr^\lax(\cB)\rightarrow\cB$) is isomorphic to $F$ (resp. to $G$) and a 2-morphism $\alpha:f\rightarrow g$ to be a morphism $\cC\rightsquigarrow\arr^{\theta_{1,1},\lax}(\cB)$ whose natural projections to $\mor_\cat^\lax(\cC,\arr^\lax(\cB))$ are isomorphic to $f$ and $g$ respectively.
\end{defn}
\begin{construction}\label{constr:cart}
Given categories $\cE$ and $\cC$ together with a subcategory $i:\cD\hookrightarrow\cC$ such that $i_0:\cD_0\overset{\sim}{\hookrightarrow}\cC_0$ is an isomorphism, denote by $\mor^\cD(\cE,\cC)$ the category whose objects are functors $F:\cE\rightarrow\cC$ and whose morphisms are natural transformations $\alpha:F\rightarrow G$ such that all $\alpha_e:F(e)\rightarrow G(e)$ belong to $\cD$, also denote by $\arr_{\cC,n}(\cD)$ the full subcategory of $\mor_\cat([n],\cC)$ on those morphisms $[n]\rightarrow\cC$ that factor through $\cD$. Denote by $\ccart^\cD(\cC)$ the Segal fibration over $\ccat$ sending $[n]$ to $\mor^\cD([n],\cC)$. We will also denote $\seg_{\ccart^\cD(\cC)}(\cS)$ by $\cart^\cD(\cC)$.
\end{construction}
\begin{prop}\label{prop:cart_def}
$\ccart^\cD(\cC)$ is an extendable Segal pattern. Moreover, each intermediate subcategory inclusion $\cD\overset{j}{\hookrightarrow}\cA\hookrightarrow\cC$ induces an extendable Segal morphism $j:\ccart^\cD(\cC)\rightarrow\ccart^\cA(\cC)$ and the adjunction 
\[L_\seg j_!:\cart^\cD(\cC)\leftrightarrows\cart^\cA(\cC):j^*\]
is monadic.
\end{prop}
\begin{proof}
Observe that for every object $f\in\ccart^\cD(\cC)$ we have \[(\ccart^\cD(\cC))^\el_{f/}\cong(\ccart^\cA(\cC))^\el_{j(f)/}\cong(\ccat_{/\cC})^\el_{f/},\] 
so the fact that $j$ is a Segal morphism follows. The next two statements would follow from \cite[Proposition 9.5]{chu2019homotopy} and \cite[Corollary 9.17]{chu2019homotopy} respectively if we prove that $\ccart^\cD(\cC)\rightarrow\ccat$ is a Segal fibration. However, this reduces to an obvious isomorphism 
\[\mor^\cD([n],\cC)\cong\underset{e\in\ccat^\el_{[n]/}}{\lim} \mor^\cD(e,\cC).\]
The last statement now follows from \cite[Proposition 8.1]{chu2019homotopy}.
\end{proof}
\begin{remark}
Observe that there are two natural subcategories of $\cC$: the entire category $\cC$ and $\widetilde{i}:\cC^\sim\hookrightarrow\cC$. In the first case we will denote $\ccart^\cC(\cC)$ simply by $\ccart(\cC)$. Observe that $\ccart^{\cC^\sim}(\cC)$ is isomorphic to $\ccat_{/\cC}$. Moreover, for every subcategory $\cD$ the inclusion $\widetilde{i}$ factors through $\cD$, and so by \cref{prop:cart_def} we have an induced morphism $L_\seg \widetilde{i}_{\cD,!}:\cat_{/\cC}\rightarrow\cart^\cD(\cC)$ and moreover $\cart^\cD(\cC)$ is isomorphic to the category of algebras for the monad $\widetilde{i}^*_\cD L_\seg \widetilde{i}_{\cD,!}$.
\end{remark}
\begin{prop}\label{prop:cart_mon}
The monad $\widetilde{i}^*_\cD L_\seg \widetilde{i}_{\cD,!}$ on $\cat_{/\cC}$ is isomorphic to 
\[L_\cart^\cD:(p:\cE\rightarrow\cC)\mapsto\cE\times_{\cC}\arr_\cC(\cD)\bydef\cC^\cD_{p/}.\]
\end{prop}
\begin{proof}
Observe that every morphism $f\xrightarrow{\alpha}g$ in $\ccart^\cD(\cC)$ decomposes uniquely as $f\xrightarrow{m}h\xrightarrow{\alpha_\cD}g$ where $m$ is a morphism in $\ccat_{/\cC}$ and $\alpha_\cD$ is a morphism in $\mor^\cD([m],\cC)$ for some $m$ - this follows since $\ccart^\cD(\cC)$ is a coCartesian fibration over $\ccat$. The statement now follows easily since by definition we have 
\[\widetilde{i}^*_\cD\widetilde{i}_{\cD,!}\cE(f)\cong\underset{(\widetilde{i}f'\rightarrow f)\in\widetilde{i}/f}{\colim}\cE(f')\cong \underset{(f'\rightarrow f)\in\arr_{\cC,n}(\cD)}{\colim}\cE(f')\cong\cC^\cD_{p/}(f)\]
for any $(f:[n]\rightarrow\cC)\in\ccat_{/\cC}$, where the second isomorphism follows from the factorization described above. Since this is obviously a Segal $\ccart^\cD(\cC)$-space, we see that $L_\seg \widetilde{i}_{\cD,!}\cong \widetilde{i}_{\cD,!}$.
\end{proof}
\begin{notation}\label{not:cartmor}
We will often identify $\cF\in\cart^\cD(\cC)$ with the underlying overcategory $p:\cF\rightarrow\cC$. Observe that if we have $a\in\cF_x$ and a morphism $f:x\rightarrow y$ belonging to the subcategory $\cD$, then we have a morphism $f_*:x\rightarrow y$ in $\ccart^\cD(\cC)$ and so a corresponding functor $f_*:\cF_x\rightarrow\cF_y$. Moreover, we can view $f$ as a morphism from $x$ to $f:[1]\rightarrow\cC$ in $\ccart^\cD(\cC)$, so we also have a functor $\ob(\cF_x)\rightarrow\cF_f$. We will call the arrows in the essential image of this functor \textit{cocartesian} and denote them by $f_*$; an arrow of this form has $a\in\cF_x$ as the source and $f_*(a)\in\cF_y$ as the target.
\end{notation}
\begin{notation}
We denote by $\Delta^\li_a$ the full subbicategory of $\cat$ on $\varnothing$ and $[n]$ for $n\geq0$. Observe that the underlying category of $\Delta^\li_a$ is $\Delta_a$, in particular there is a natural inclusion $i:\Delta_a\hookrightarrow\Delta^\li_a$. Observe that it also inherits the monoidal structure from $\Delta_a$ given by 
\[([n],[m])\mapsto[m+n+1]\]
with the unit object given by $\varnothing$. In particular, we can consider the tricategory $\rB \Delta^\li_a$. Now let $\cT$ be an arbitrary tricategory. Recall from \cref{cor:lax_mon} that giving a monad in $\cT$ is equivalent to giving a functor $T:\rB\Delta_a\rightarrow\cT$. We will call the monad $T$ \textit{lax-idempotent} if factors through $i:\rB\Delta_a\hookrightarrow\rB\Delta^\li_a$.
\end{notation}
\begin{prop}
A monad $T$ admits a lax idempotent structure if and only if the unit morphism $eT:T\rightarrow TT$ is right adjoint to multiplication $m:TT\rightarrow T$, in which case this structure is unique.
\end{prop}
\begin{proof}
First, observe that a morphism $f:[m]\rightarrow[n]$ in $\Delta^\li_a$ is a right adjoint if it preserves the maximal element and a left adjoint if it preserves the minimal element. In particular, if there is a lax idempotent structure on a monad $T$ then we certainly have the adjunction described in the claim.\par
Conversely, denote by $\theta:i_0\rightarrow i_1$ the natural transformation between the two inert morphisms $i_{0,1}:[0]\rightarrow[1]$. We claim that a lax-idempotent structure is uniquely defined by the image of $\theta$. Indeed, it is easy to see that any natural transformation $\alpha:f\rightarrow g$ where $f$ and $g$ are functors from $[m]$ to $[n]$ has the form $\theta^{g(0)-f(0)}\otimes\theta^{g(1)-f(1)}\otimes...\otimes\theta^{g(m)-f(m)}$. Now observe that $\theta$ itself is isomorphic to the composite
\[i_0\cong\id\circ i_0\xrightarrow{\eta}i_1\circ p\circ i_0\cong i_1,\]
where $p:[1]\rightarrow[0]$ denotes the unique morphism and $\eta$ is the unit of the adjunction $p\dashv i_1$. It follows that the structure of the lax-idempotent monad is uniquely determined by the data of the adjunction $m\dashv eT$ and this data is unique if it exists.
\end{proof}
\begin{prop}\label{prop:lax_uniq}
Given a lax-idempotent monad on a flagged bicategory $\cB$ and an object $b\in\cB$, there is a structure of a $T$-algebra on $b$ if and only if the unit morphism $e:b\rightarrow Tb$ admits a left adjoint $t$ such that $t\circ b\cong\id$, and this structure is unique if it exists.
\end{prop}
\begin{proof}
Consider a functor $F:\Delta_a\cong(\Delta^\act)^\op\rightarrow\cB$ that sends $[n]$ to $T^n b$ and every morphism to the corresponding structural morphism of the monad $T$. It follows from \cref{lem:monad_alg} that giving $b$ a structure of a $T$-algebra is equivalent to extending this to a functor $\widetilde{F}\Delta^\op_{-\infty}\rightarrow\cB$. First assume that $e:b\rightarrow Tb$ has a left adjoint $t:Tb\rightarrow b$. We claim that $t$ gives $b$ the structure of a $T$-algebra. Observe that every inert morphism in $\Delta_{-\infty}$ is a composition of $s_{\infty,n};[n]\hookrightarrow[n+1]$ that maps $[n]$ isomorphically onto the initial segment of $[n+1]$. We declare $\widetilde{F}(s_{\infty,n})$ to be $T^{n+1}b\xrightarrow{t}T^n b$. We now need to check that for every factorization square
\[
\begin{tikzcd}[row sep=huge, column sep=huge]
{[n]}\arrow[d, tail, "s_{\infty,n}"]\arrow[r, two heads, "\widetilde{a}"]&{[l]}\arrow[d, tail, "\widetilde{s}"]\\
{[n+1]}\arrow[r, two heads, "a"]&{[m]}
\end{tikzcd}
\]
in $\Delta_{-\infty}$ we have $\widetilde{F}(s_{\infty,n})\circ F(a)\cong F(\widetilde{a})\circ \widetilde{F}(\widetilde{s})$. Now observe that if $a$ is such that its restriction to $[n,n+1]$ is equal to identity, then the commutativity is trivial. Using the equivalence 
\[\Act_\ccat([n])\cong\underset{([n]\rightarrowtail e)\in\ccat^\el_{[n]/}}{\lim}\Act(e)\]
we are thus reduced to the case $n=0$. We are then reduced to proving the commutativity of the diagrams
\[
\begin{tikzcd}[row sep=huge, column sep=huge]
b\arrow[r, "e"]\arrow[d, equal]&Tb\arrow[d, "t"]\\
b\arrow[r, equal]&b
\end{tikzcd},
\begin{tikzcd}[row sep=huge, column sep=huge]
T^m b\arrow[r, "m"]\arrow[d, "t^m"]&Tb\arrow[d, "t"]\\
b\arrow[r, equal]&b
\end{tikzcd}.
\]
The first of these diagrams commute by assumption, to prove he commutativity of the second diagram, observe that every morphism in it has a right adjoint (by the definition of $t$ and since $T$ is lax-idempotent), so it suffices to prove the commutativity of the corresponding diagram of right adjoints. It is given by
\[
\begin{tikzcd}[row sep=huge, column sep=huge]
b\arrow[r, equal]\arrow[d, "e"]&b\arrow[d, "e^m"]\\
Tb\arrow[r, "e^{m-1}"]&T^m b
\end{tikzcd}
\]
and it obviously commutes.\par
Now assume that $b$ has a $T$-algebra structure $t:Tb\rightarrow b$. The isomorphism $t\circ e$ then follows from the definition. Moreover, it is elementary to see that the composition 
\[\id_{Tb}\cong Tt \circ Te\xrightarrow{\eta_m} Tt\circ eT\circ m\circ Te\cong Tt\circ eT\cong e\circ t,\]
where $\eta_m:\id_{T^2b}\rightarrow eT\circ m$ is the unit of the adjunction $m\dashv eT$ provides a unit for the adjunction $t\dashv e$. It follows that the structure of a $T$-algebra is necessarily given by the left adjoint to $e:b\rightarrow Tb$, and this morphism is unique if it exists, which concludes the proof of the claim.
\end{proof}
\begin{remark}
Lax-idempotent monads (also frequently called \textit{KZ-monads}) were introduced for ordinary categories in \cite{kock1972monads} and \cite{zoberlein1976doctrines}, a significantly more detailed account of them in classical terms can be found in \cite{bunge2006singular}.
\end{remark}
\begin{prop}\label{prop:cart_uniq}
The structure of a $\cD$-coCartesian fibration on a given $(p:\cE\rightarrow\cC)\in\ccat_{\cC}$ exists if and only if the natural inclusion $e:\cE\hookrightarrow \cE\times_{\cC}\arr_\cC(\cD)$ admits a left adjoint.
\end{prop}
\begin{proof}
We will consider $\cat_{/\cC}$ as a flagged bicategory. The category of $\cD$-coCartesian fibrations over $\cC$ is then isomorphic to the category of algebras for the monad that sends $p:\cF\rightarrow\cC$ to $\cC^\cD_{p/}$. Observe that $e:\cE\hookrightarrow \cE\times_{\cC}\arr_\cC(\cD)$ is the unit morphism for the free $\cD$-coCartesian fibration monad. The morphism $e$ is easily seen to be fully faithful: indeed, it sends an object $x\in\cE$ to the pair $(x,\id_{p(x)})$ and it follows directly from the definitions that
\[\mor_{\cC^\cD_{p/}}((x,\id_{p(x)}),(y,\id_{p(y)}))\cong\mor_\cE(x,y).\]
Now observe that this monad is lax-idempotent: indeed, by definition $T^n_{\ccart^\cD(\cC)}(\cE\xrightarrow{p}\cC)\cong\cE\times_\cC \arr^n_\cC(\cD)$; since $\arr_\cC(\cD)$ is defined as a full subcategory of $\mor_\cat([n],\cC)$, we see that this functor indeed extends to a functor $\Delta^\li_a\rightarrow \edm_{2-\cat}(\cat_{/\cC})$.\par
By \cref{prop:lax_uniq} the only thing we need to show is that for  any left adjoint $m:\cC^\cD_{p/}\rightarrow p$ we have $m\circ e\cong\id_p$. Indeed, for any pair of objects $(x,y)$ of $\cF$ we have
\[\mor_{\cF}(m\circ e x,y)\cong\mor_{\cC^\cD_{p/}}(ex,ey))\cong\mor_\cF(x,y),\]
where the first isomorphism follows from the definition and the second follows since $e$ is fully faithful, and so the claim follows by Yoneda lemma.
\end{proof}
\begin{remark}\label{rem:cart_univ}
Given a $\cD$-coCartesian fibration $p:\cF\rightarrow\cC$, the unit morphism $e:\cF\rightarrow\cC^\cD_{/p}$ is given by sending $x$ to $(x,\id_x)\in\cC^\cD_{/p}$. The left adjoint $m:\cC^\cD_{/p}\rightarrow \cF$ sends $(x,f:p(x)\rightarrow y)$ to $f_*x\in\cF_y$. Observe that 
\[\mor_{\cC^\cD_{/p}}((x,f),e(z))\cong \mor_\cC(y,p(z))\times_{\mor_\cC(p(x),p(z))}\mor_\cF(x,z).\]
The universal property of left adjoint then gives 
\[\mor_\cF(f_*x,z)\cong \mor_\cC(y,p(z))\times_{\mor_\cC(p(x),p(z))}\mor_\cF(x,z).\]
\end{remark}
\begin{cor}\label{cor:cartlift}
Given a $\cD$-coCartesian fibration $p:\cF\rightarrow\cC$, a morphism $\widetilde{f}:x\rightarrow y$ in $\cF$ lying over $f:p(x)\rightarrow p(y)$ in $\cD$ is coCartesian in the sense of \cref{not:cartmor} if and only if for any object $z\in\cF$ the following commutative square
\[
\begin{tikzcd}[column sep=huge, row sep=huge]
    \mor_{\cF}(y,z)\arrow[r]\arrow[d]&\mor_{\cF}(x,z)\arrow[d]\\
    \mor_{\cC}(p(y),p(z))\arrow[r]& \mor_{\cC}(p(x),p(z))
\end{tikzcd}
\]
is a pullback square. Moreover, $p$ is a $\cD$-coCartesian fibration if and only if every $\cD$-morphism of $\cC$ admits a coCartesian lift.
\end{cor}
\begin{proof}
The claimed universal property is simply a restatement of \cref{rem:cart_univ}. Assuming every morphism $f$ of $\cC$ admits a coCartesian lift, we can define a left adjoint to $e:p\rightarrow\cC^\cD_{p/}$ by sending $(x, f:px\rightarrow y)$ to the target of the coCartesian lift of $f$.
\end{proof}
\begin{remark}
\Cref{cor:cartlift} and \cite[Proposition 2.4.4.3.]{lurie2009higher} allows us to identify the category of $\cD$-coCartesian fibrations in the sense of \cref{prop:cart_def} and the category with objects $p:\cF\rightarrow\cC$ for which all morphisms in $\cD$ admit $p$=coCartesian lifts and with morphisms given by those morphisms over $\cC$ that preserve those lifts. In the case $\cD\cong\cC$ we can alternatively conclude by observing the isomorphism between $L^\cC_{\cart}$ and the free fibration monad of \cite{gepner2015lax}.
\end{remark}
\begin{prop}\label{prop:lax}
Given a Segal space $\cC$, there exists a twofold Segal space $L^\lax\cC$ such that for every twofold Segal space $\cB$
\[\mor^\lax_{2-\cat}(L^\lax\cC,\cB)\cong \mor_\cat^\lax(\cC,\cB).\]
\end{prop}
\begin{proof}
First, observe that it is enough to present an isomorphism on the level of underlying spaces, since the isomorphism on morphisms and 2-morphisms would then follow from the same argument applied to $\arr^\lax(\cB)$ and $\arr^{\theta_{1,1},\lax}(\cB)$. We begin by defining 
\[L^\lax\cC\bydef \cC\times_{\Delta^\op}\arr^\act(\Delta^\op),\]
where $\arr^\act(\Delta^\op)$ is a coCartesian fibration over $\Delta^\op\cong\ccat$ corresponding to a functor that sends $s\in\ccat$ to $\ccat^\act_{/s}$.\par
To prove that this object satisfies the required universal property, first observe that a morphism of twofold Segal spaces can be identified with a morphism of the corresponding coCartesian fibrations over $\Delta^\op$, while a lax functor can be identified with a morphism in $\cart^\inrt(\Delta^\op)$. Denote by $i^\inrt:(\Delta^\inrt)^\op\hookrightarrow\Delta^\op$ the natural inclusion, then by \cref{prop:cart_def} the natural forgetful functor $U:\cart(\Delta^\op)\rightarrow\cart^\inrt(\Delta^\op)$ admits a left adjoint $L_\seg i^\inrt_!$. We will prove that $i^\inrt_!\cC\cong L^\lax\cC$, in particular it would follow that it satisfies the $\ccart(\Delta^\op)$ Segal condition, so in particular $L_\seg i^\inrt_!\cC\cong i^\inrt_! \cC$. For any morphism $f:[n]\rightarrow\cC$ we have by definition
\[i^\inrt_!\cC(f)\cong \underset{(g\rightarrow f)\in(\arr^n(\Delta^\op))^\inrt_{/f}}{\colim}\cC(g)\cong\underset{(g\rightarrow f)\in(\arr^n(\Delta^\op))^\inrt_{/f}}{\colim}\cC(g(0)),\]
where $g:[n]\rightarrow\Delta^\op$ is a functor, the isomorphism $\cC(g)\cong \cC(g(0))$ follows since $\cC$ is a right fibration over $\Delta^\op$ and $(\arr^n(\Delta^\op))^\inrt_{/f}$ is a category whose objects are morphisms commutative diagrams 
\[
\begin{tikzcd}[row sep=huge,column sep=huge]
g(0)\arrow[r, "g_1"]\arrow[d, "s_0"]&g(1)\arrow[r, "..." description]\arrow[d, "s_{1}"]&g(n-1)\arrow[d, "s_{n-1}"]\arrow[r, "g_n"]&g(n)\arrow[d, "s_n"]\\
f(0)\arrow[r, "f_1"]&f(1)\arrow[r, "..." description]&f(n-1)\arrow[r,"f_n"]&f(n)
\end{tikzcd}
\]
in $\Delta^\op$ and morphisms $i:g\rightarrow h$ are given by diagrams 
\[
\begin{tikzcd}[row sep=huge,column sep=huge]
g(0)\arrow[r, "g_1"]\arrow[d, tail, "i_0"]&g(1)\arrow[r,  "..." description]\arrow[d, tail, "i_{1}"]&g(n-1)\arrow[d, tail, "i_{n-1}"]\arrow[r, "g_n"]&g(n)\arrow[d, tail, "i_n"]\\
h(0)\arrow[r, "g_1"]\arrow[d, "t_0"]&h(1)\arrow[r, "..." description]\arrow[d, "t_1"]&h(n-1)\arrow[d, "t_{n-1}"]\arrow[r, "g_n"]&h(n)\arrow[d, "t_n"]\\
f(0)\arrow[r, "f_1"]&f(1)\arrow[r, "..." description]&f(n-1)\arrow[r,"f_n"]&f(n)
\end{tikzcd}
\]
where for all $j$ we have $t_j\circ i_j\cong s_j$ and all $i_j$ are inert.\par
Denote by $(\Act^n(\Delta^\op))_{/f}$ the groupoid whose objects are objects $s:g\rightarrow f$ of $(\arr^n(\Delta^\op))^\inrt_{/f}$ as above for which all $s_j$ are active morphisms. Observe that $(\Act^n(\Delta^\op))_{/f}$ is a full subcategory of $(\arr^n(\Delta^\op))^\inrt_{/f}$ and also that we have an isomorphism
\[L^\lax\cC(f)\cong \underset{(g\overset{a}{\twoheadrightarrow}f)\in(\Act^n(\Delta^\op))_{/f}}{\colim}\cC(g(0)).\] 
We will now prove that the inclusion $(\Act^n(\Delta^\op))_{/f}\hookrightarrow(\arr^n(\Delta^\op))^\inrt_{/f}$ is cofinal, it would imply that $i^\inrt_!\cC\cong L^\lax\cC$ using the isomorphism above. In fact, for every object $(g\xrightarrow{s}f)\in(\arr^n(\Delta^\op))^\inrt_{/f}$ we will produce a unique up to isomorphism morphism to an object of $(\Act^n(\Delta^\op))_{/f}$. To do this first denote by $f(k-1)\overset{f^i_k}{\rightarrowtail}c_k\overset{f^a_k}{\twoheadrightarrow}f(k)$ the active/inert factorization of $f_k$ and introduce similar notation for $g_k$ and $s_k$, then observe that we have the following diagram
\[
\begin{tikzcd}[row sep=huge,column sep=huge]
g(0)\arrow[r, tail, "g^i_1"]\arrow[d, tail, "s^i_0"]&a_1\arrow[d, tail, "u_1"]\arrow[r, two heads, "g^a_1"]&g(1)\arrow[r, tail, "g^i_2"]\arrow[d, tail, "s^i_1"]&a_2\arrow[d, tail, "u_2"]\arrow[r, two heads, "..." description]&g(n-1)\arrow[r, tail, "g^i_n"]\arrow[d, tail, "s^i_{n-1}"]&a_n\arrow[d, tail, "u_n"]\arrow[r, two heads, "g^a_n"]&g(n)\arrow[d, tail, "s^i_n"]\\
h(0)\arrow[d, two heads, "s^a_1"]\arrow[r, tail, "h^i_1"]&b_1\arrow[d, two heads, "v_1"]\arrow[r, two heads, "h^a_1"]&h(1)\arrow[d, two heads, "s^a_2"]\arrow[r, tail, "h^i_2"]&b_2\arrow[d, two heads, "v_2"]\arrow[r, two heads, "..." description]&h(n-1)\arrow[d, two heads, "s^a_{n-1}"]\arrow[r, tail, "h^i_{n}"]&b_n\arrow[d, two heads, "v_n"]\arrow[r, two heads, "h^a_n"]&h(n)\arrow[d, "s^a_n"]\\
f(0)\arrow[r, tail, "f^i_1"]&c_1\arrow[r, two heads, "f^a_1"]&f(1)\arrow[r, tail, "f^i_2"]&c_2\arrow[r, two heads, "..." description]&f(n-1)\arrow[r, tail, "f^i_n"]&c_n\arrow[r, two heads, "f^a_n"]&f(n)
\end{tikzcd}
\]
in which every square is given by active/inert factorization. The fact that $b_k$ that appears in the active/inert factorization of both $s^i_k\circ g^a_k$ and $f^i_k\circ s^a_k$ is the same object follows since $(f^a_k\circ v_k)\circ(h^i_k\circ s^i_{k-1})$ and $(s_k^a\circ h^a_k)\circ(u_i\circ g^i_k)$ represent the active/inert factorization of the same morphism. The object $(h\overset{s_1}{\twoheadrightarrow}f)$ obviously belongs to $(\Act^n(\Delta^\op))_{/f}$ and the morphism $(g\overset{s_i}{\rightarrowtail}h)$ is unique by the uniqueness of the active/inert factorization.\par
It follows that $i^\inrt_! \cC\cong L^\lax\cC$. It remains to prove that this is indeed a twofold Segal space. For this we need to show that both $\Act_\ccat([n])$ and $\Act^2_\ccat([n])$ are Segal spaces, which easily follows since $\ccat$ is extendable; this proves that $L^\lax\cC$ is a double Segal space. Moreover, since there are no non-trivial active morphisms with target $[0]$, we also see that $L^\lax\cC([0])\cong \cC([0])$ is a groupoid, which concludes the proof.
\end{proof}
\begin{lemma}\label{lem:int_sur}
Morphisms of the form $i\circ s$ where $i$ is an inert morphism and $s$ is a surjective active morphism form a subcategory of $\Delta$.
\end{lemma}
\begin{proof}
It would suffice to prove that if 
\[
\begin{tikzcd}[row sep=huge, column sep=huge]
x\arrow[d, tail, "i"]\arrow[r, two heads, "s"]&y\arrow[d, tail, "j"]\\
z\arrow[r, two heads, "t"]&w
\end{tikzcd}
\]
is a factorization square in $\Delta$ in which $i$ is inert and $t$ is surjective, then $s$ is also surjective. However, this is obvious: $i$ is an inclusion of a subinterval, $s$ is the restriction of $t$ to that subinterval and $j$ is the inclusion of the image of $x$ under $t$, so $s$ is obviously surjective.
\end{proof}
\begin{notation}
We will denote by $\Delta^{\sur,\inrt}$ the subcategory described in \cref{lem:int_sur}.
\end{notation}
\begin{cor}\label{cor:unilax}
Given a Segal space $\cC$ there exists a twofold Segal space $L^{\lax,\un}\cC$ such that for every two-fold Segal space $\cB$
\[\mor^\lax_{2-\cat}(L^{\lax,\un}\cC,\cB)\cong \mor_\cat^{\lax,\un}(\cC,\cB).\]
\end{cor}
\begin{proof}
First, observe that just like in \cref{prop:lax} it suffices to prove the statement at the level of underlying spaces. A unital lax functor is equivalently a morphism in $\cart^{\sur,\inrt}(\Delta^\op)$ and we know by \cref{prop:cart_def} that the natural forgetful functor $U:\cart(\Delta^\op)\rightarrow\cart^{sur,\inrt}(\Delta^\op)$ admits a left adjoint $L_\seg i^{\sur,\inrt}_!$. It is \textit{a priori} not a twofold Segal space, but it is at least a presheaf on $\Delta^\op\times\Delta^\op$, so it factors through a universal twofold Segal space $L_{\mathrm{2-Seg}}(L_\seg i^{\sur,\inrt}_!\cC)$ and we define $L^{\lax,\un}\cC\bydef L_{\mathrm{2-Seg}}(L_\seg i^{\sur,\inrt}_!\cC)$.
\end{proof}
\begin{lemma}\label{lem:pullback}
Given an injective morphism $i:[l]\hookrightarrow[n]$ and a morphism $s:[m]\rightarrow[n]$ in $\Delta$ such that $\ima(s)\cap\ima(i)\neq\varnothing$ we have that the $x$ in the following pullback square
\[
\begin{tikzcd}[row sep=huge, column sep=huge]
x\arrow[d, "s'"]\arrow[r, hook, "i'"]&{[m]}\arrow[d, "s"]\\
{[l]}\arrow[r, hook, "i"]&{[n]}
\end{tikzcd}
\]
in $\cat$ is isomorphic to an object of $\Delta$.
\end{lemma}
\begin{proof}
Indeed, the set of objects of $x$ is obviously finite and nonempty, its fiber over every $\{j\}\in[l]$ is a (possibly empty) linearly ordered set. Moreover, if $\{j<j'\}\in[l]$ and the corresponding fibers are nonempty, then there is a single morphism from any element of $x_j$ to the smallest element of $x_{j'}$. It follows that $x$ is indeed a finite linearly ordered set.
\end{proof}
\begin{construction}\label{constr:lax_mon}
Define the bicategory $\bimod_n$ for $n\geq0$ as follows: first define its objects to be natural numbers $0\leq i\leq n$. define $\mathrm{End}_{\bimod_n}(i)\bydef \Delta_a$. Define $\mor_{\bimod_n}(i,j)\bydef \varnothing$ if $i>j$. To define $\mor_{\bimod_n}(i,j)$ in the other cases first consider the category $(\Delta^{\act,\inj}_{/[j-i]})^\op$ whose objects are active injective morphisms $[n]\twoheadrightarrow [j-i]$. Consider a functor $F_{i,j}:(\Delta^{\act,\inj}_{/[j-i]})^\op\rightarrow\cat$ that sends $[n]\twoheadrightarrow [j-i]$ to $\Delta^n_a$ and sends an injective morphism $i:[m]\hookrightarrow[n]$ to the natural projection $i^*:\Delta^n_a\rightarrow\Delta^m_a$. Finally, define $\mor_{\bimod_n}(i,j)$ to be the total category corresponding to $F_{i,j}$ if $j>i$. Given two composable morphisms $(a:[n]\twoheadrightarrow[j-i],c\in\Delta^n_a)$ and $(b:[m]\twoheadrightarrow[k-j],d\in\Delta^m_a)$, define their composition to be a pair $(a+b:[m+n]\twoheadrightarrow[k-i],c+d\in\Delta^{(m+n)}_a)$ where $(a+b)$ is an injective active morphism obtained by gluing $a$ and $b$ and $(c+d)$ has the same projections to the first $(j-i-1)$ components as $c$ and the same projection to the last $(k-j-1)$ components as $d$ and whose projection to the $(j-i)$-th component is given by the sum of the corresponding components of $c$ and $d$.
\end{construction}
\begin{lemma}\label{lem:bimod}
There is a natural equivalence 
\[L^\lax [n]\cong \bimod_n\]
\end{lemma}
\begin{proof}
By applying \cref{prop:lax} we see that the objects of $L^\lax [n]$ are in bijection with the objects of $[n]$. Given a pair $i\leq j$, denote by $f_{i,j}:[1]\rightarrow[j-i]$ the unique active morphism in $\Delta$, then it is easy to see that the category $\mor_{L^\lax[n]}(i,j)$ is isomorphic to $(\Delta^\act_{/[n]})^\op_{f_{i,j}/}\cong(\Delta^\act_{/[j-i]})^\op$.\par
We first prove the claim in the special case $i=j$. The claim is then reduced to
\[\Delta^{\act,\op}\cong\Delta_a\]
To prove it recall the classical fact that $\Delta$ is generated by morphisms $\delta^n_i:[n-1]\rightarrow[n]$  and $\sigma^n_i:[n+1]\rightarrow[n]$ for $0\leq i\leq n$ that satisfy
\begin{gather*}
    \delta^{n+1}_i\circ \delta^n_j\cong \delta^{n+1}_{j+1}\circ \delta^{n}_i\;\text{    if }i\leq j\\
    \sigma^{n}_j\circ \sigma^{n+1}_i\cong \sigma^n_{i}\circ \sigma^{n+1}_{j+1}\; \text{ if }i\leq j\\
    \sigma^n_j\circ \delta^{n+1}_i\cong 
    \begin{cases}
        \delta^n_i\circ \sigma^{n-1}_{j-1}&\text{    if }i< j\\
        \id_{[n]} \;\text{    if }i=j&\text{ or }i=j+1\\
        \delta^n_{i-1}\circ \sigma^{n-1}_{j}&\text{    if }j+1< i\\
    \end{cases}.
\end{gather*}
Out of these morphisms only $\delta^n_0$ and $\delta^n_n$ are inert and the rest are active. We define a morphism $F:\Delta^{\act,\op}\rightarrow\Delta_a$ by sending $[n]$ to $[n-1]$, sending $\delta^n_{o<i<n}:[n-1]\rightarrow[n]$ to $\sigma^{n-2}_{i-1}:[n-1]\rightarrow[n-2]$ and sending $\sigma^n_i:[n+1]\rightarrow[n]$ to $\delta^n_i:[n-1]\rightarrow[n]$. It is elementary to check that this map preserves the necessary relations, so since it induces an isomorphism on the set of generators we see that $F$ is indeed an isomorphism.\par
Now assume $j>i$. In this case first observe that any active morphism $a:[n]\twoheadrightarrow[j-i]$ in $\Delta$ decomposes as $a\cong t\circ s$ where $t$ is injective and $s$ is surjective. In particular, there is a functor $F:(\Delta^\act_{/[j-i]})^\op\rightarrow(\Delta^{\act,\inj}_{/[j-i]})^\op$ sending each active morphism to its the injective component. Given an active injective morphism $t:[m]\hookrightarrow[j-i]$, it is easy to see that the fiber of $F$ over $t$ is isomorphic to $(\Delta^{\act,\op})^m\cong\Delta_a^m$. Moreover, given a morphism $q:t\rightarrow t'$ observe that we can form the following diagram in $\Delta$
\[
\begin{tikzcd}[row sep=huge]
{[n']}\arrow[d, "\overline{s}"]\arrow[rr,"\overline{q}"]&{}&{[n]}\arrow[d, "s"]\\
{[m']}\arrow[rr,"q"]\arrow[dr,"t'"]&{}&{[m]}\arrow[dl, "t" swap]\\
{}&{[n]}
\end{tikzcd}
\]
where the top square is a pullback (which exists by \cref{lem:pullback}). Using the universal property of the pullback we see that $F:(\Delta^\act_{/[j-i]})^\op\rightarrow(\Delta^{\act,\inj}_{/[j-i]})^\op$ is a coCartesian fibration. Moreover, it is easy to see that under the isomorphism $F_t\cong \Delta^m_a$ and $F_{t'}\cong \Delta^{m'}_a$ the functor $q_*:F_t\rightarrow F_{t'}$ is exactly the same as described in \cref{constr:lax_mon}.
\end{proof}
\begin{cor}\label{cor:uni}
The bicategory $L^{\lax,\un}[n]$ has objects given by natural numbers $0\leq i\leq n$ and the morphisms categories are given by $\mor_{L^{\lax,\un}}(i,j)\cong(\Delta^{\act,\inj}_{/[j-i]})^\op$ with composition given by gluing active morphisms.
\end{cor}
\begin{proof}
Observe that by definition a unital lax functor $L^{\lax,\un}[n]\rightarrow\cB$ is a lax functor $F:\bimod_n\rightarrow\cB$ that sends all surjective active morphisms $s:[n]\twoheadrightarrow[j-i]$ to identities. By examining the construction of $\bimod_n$ and the proof of \cref{lem:bimod}, we see that those are functors $F$ sending all endomorphisms in $\edm_{\bimod_n}(i)\cong\Delta_a$ for all $i$ to identities. It now follows easily that the space of those functors is indeed corepresented by the category described in the claim.
\end{proof}
\begin{cor}\label{cor:lax_mon}
Giving a lax functor from a point to a twofold Segal space $\cB$ is equivalent to giving a monad in $\cB$. In particular, giving a monad in $\cB$ is equivalent to giving a morphism \[\rB\Delta_a\rightarrow\cB.\]
\end{cor}
\begin{proof}
Indeed, for a given object $b\in\cB$ giving a monad with the underlying object $b$ is by definition equivalent to giving an algebra in the monoidal category $\edm_\cB(b)$. If we view $\edm_\cB(b)$ as a coCartesian fibration over $\Delta^\op$, then this in turn is equivalent to giving a morphism $*\rightarrow\edm_\cB(b)$ over $\Delta^\op$ that sends coCartesian morphisms lying over inert morphisms to coCartesian morphisms. This in turn is by definition the same thing as a lax functor $*\rightsquigarrow\cB$ sending $*$ to $b$.\par
The second statement now follows from \cref{lem:bimod} and the description of $\bimod_0$.
\end{proof}
\begin{construction}
Observe that since $\ccart^\cD(\cC)$ is a coCartesian fibration over $\Delta^\op$ we have that every morphism in $\ccart^\cD(\cC)$ canonically decomposes as $a\circ s$ where $a:([n]\xrightarrow{f}\cC)\rightarrow ([m]\xrightarrow{g'}\cC)$ is a morphism of categories over $\cC$ and $s$ is a morphism in $\mor^\cD([m],\cC)$. Denote by  \[p:\ccart^\cD(\cC)\rightarrow\cC\]
the morphism that sends $(c_0\xrightarrow{f_1}c_1\xrightarrow{f_2}...\xrightarrow{f_n}c_n)$ to $c_0$, sends every morphism $s$ given by the diagram
\[
\begin{tikzcd}[row sep=huge,column sep=huge]
c_0\arrow[r, "f_0"]\arrow[d,  "s_0"]&c_1\arrow[r, "..." description]\arrow[d,  "s_{1}"]&c_n\arrow[d,  "s_n"]\\
\overline{c_0}\arrow[r, "g_0"]&\overline{c_1}\arrow[r, "..." description]&\overline{c_n}
\end{tikzcd}
\]
in which all $s_i$ belong to $\cD$ to $s_0:c_0\rightarrow\overline{c_0}$, sends every active morphism in $\ccat_{/\cC}$ to the identity and sends every inert morphism of the form
\[(c_k\xrightarrow{f_{k+1}}c_{k+1}\xrightarrow{f_{k+2}}...\xrightarrow{f_l}c_l)\subset(c_0\xrightarrow{f_1}c_1\xrightarrow{f_2}...\xrightarrow{f_n}c_n)\]
to \[(f_k\circ f_{k-1}\circ...\circ f_1):c_0\rightarrow c_k.\]
\end{construction}
\begin{lemma}\label{lem:fib_cart}
The morphism 
\[p^*:\mor_\cat(\cC,\cS)\rightarrow\mor_\cat(\ccart^\cD(\cC),\cS)\]
restricts to a morphism 
\[p^*:\mor_\cat(\cC,\cS)\rightarrow\cart^\cD(\cC).\]
\end{lemma}
\begin{proof}
Denote by $\cF$ an arbitrary presheaf on $\cC^\op$, then we need to prove that $p^*\cF$ satisfies the Segal condition. The objects of $\ccart^\cD(\cC)$ are given by morphisms $f:[n]\rightarrow\cC$, so we will prove the claim by induction on $n$, the case $n=0$ being obvious. Now let $f:[n]\rightarrow\cC$ be given by $(c_0\xrightarrow{f_1}c_1\xrightarrow{f_2}...\xrightarrow{f_n}c_n)$ and assume that we have proved the claims for all chains of length $\leq n-1$, then we have the following isomorphisms
\begin{align*}
    p^*\cF(f)\cong&\cF(c_0)\\
    \cong&\cF(c_0)\times_{\cF(c_1)}\cF(c_1)\\
    \cong& p^*\cF(c_0\xrightarrow{f_1}c_1)\times_{p^*\cF(c_1)}p^*\cF(c_1\xrightarrow{f_2}...\xrightarrow{f_n}c_n)\\
    \cong&p^*\cF(c_0\xrightarrow{f_1}c_1)\times_{p^*\cF(c_1)}p^*\cF(c_1\xrightarrow{f_2}c_2)\times_{p^*\cF(c_2)}...\times_{p^*\cF(c_{n-1})}\cF(c_{n-1}\xrightarrow{f_n}c_n),
\end{align*}
where the last isomorphism uses the inductive hypothesis. This concludes the proof of the claim.
\end{proof}
\begin{prop}\label{prop:colim_lax}
For any Segal space $\cC$ and a twofold Segal space $\cB$ there is a natural isomorphism
\[\mor^\lax_\cat(\cC,\cB)\cong\underset{([n]\rightarrow\cC)\in\ccat_{/\cC}}{\lim}\mor^\lax_\cat([n],\cB).\]
Similarly, there is a natural isomorphism 
\[\mor^{\lax,\un}_\cat(\cC,\cB)\cong\underset{([n]\rightarrow\cC)\in\ccat_{/\cC}}{\lim}\mor^{\lax,\un}_\cat([n],\cB).\]
\end{prop}
\begin{proof}
The proofs of both claims are very similar and we will only prove the second one. First, observe that it would suffice to prove the statement for the underlying space on both sides, since the isomorphism on morphisms would then follow by applying the same argument to $\arr^\lax(\cB)$.\par
To prove the claim, begin by observing that $\cC$, when  viewed as an object of $\cart^{\sur,\inrt}(\Delta^\op)$, corresponds to the presheaf $p^*\cC$ on $\ccart^{\sur,\inrt}(\Delta^\op)^\op$ in the notation of \cref{lem:fib_cart}. The claim now follows from the following sequence of isomorphisms
\begin{align*}
    \mor^\lax_\cat(\cC,\cB)\cong&\mor_{\cart^{\sur,\inrt}(\Delta^\op)}(p^*\cC,\cB)\\
    \cong&\mor_{\cart^{\sur,\inrt}(\Delta^\op)}(p^*(\underset{\Delta^\op}{\colim}\;\cC_n\otimes h_{[n]}),\cB)\\
    \cong&\mor_{\cart^{\sur,\inrt}(\Delta^\op)}(\underset{\Delta^\op}{\colim}\;\cC_n\otimes p^*h_{[n]},\cB)\\
    \cong&\underset{\Delta^\op}{\lim}\;\mor_{\cart^{\sur,\inrt}(\Delta^\op)}(\cC_n\otimes h_{[n]},\cB)\\
    \cong&\underset{(f:[n]\rightarrow\cC)\in\ccat_{/\cC}}{\lim}\mor_{\cart^{\sur,\inrt}(\Delta^\op)}(h_{[n]},\cB)\\
    \cong &\underset{(f:[n]\rightarrow\cC)\in\ccat_{/\cC}}{\lim}\mor_\cat^\lax([n],\cB).
\end{align*}
\end{proof}
\begin{remark}
Combining \cref{cor:uni} with \cref{prop:colim_lax} we obtain an explicit description of $L^{\lax,\un}\cC$ for an arbitrary Segal space $\cC$.
\end{remark}
\begin{notation}\label{not:unitalization}
Observe that there is a natural inclusion $L^{\lax,\un}[n]\hookrightarrow L^\lax[n]$ which is identity on objects and sends $(f:[n]\rightarrow[j-i])\in\mor_{L^{\lax,\un}}(i,j)$ to $(f,(\varnothing,\varnothing,...,\varnothing)\in\Delta_a^n)\in\mor_{L^{\lax}}(i,j)$. Moreover, those morphisms are functorial with respect to morphisms in $\Delta$, so in particular using \cref{prop:colim_lax} we obtain for any twofold Segal space $\cB$ a functor
\[(-)^\un:\mor^\lax_\cat(\cC,\cB)\cong\underset{([n]\rightarrow\cC)\in\ccat_{/\cC}}{\lim}\mor^\lax_\cat([n],\cB)\rightarrow\underset{([n]\rightarrow\cC)\in\ccat_{/\cC}}{\lim}\mor^{\lax,\un}_\cat([n],\cB)\cong \mor^{\lax,\un}_\cat(\cC,\cB)\]
which we will call the \textit{unitalization} functor.
\end{notation}
\begin{defn}\label{def:coc}
For a twofold Segal space $\cB$, a lax functor $F:\cC\rightsquigarrow\cB$ and an object $b\in\cB$ we will define \textit{lax cocone} of $F$ over $b$ to be a lax functor $G_F:\cC^\triangleright\rightsquigarrow\cB$ such that $G_F|_{\cC}\cong F$ and the composition  $*\rightarrow \cC^\triangleright \overset{G_F}{\rightsquigarrow}\cB$ is the constant functor with the value $b$ (where $*$ denotes the vertex of the cone). We will endow the space of lax cocones over $b$ with the structure of a Segal space by declaring a morphism $\alpha:G_F\rightarrow H_F$ to be a natural transformation between $G_F$ and $H_F$ in $\mor_\cat^\lax(\cC^\triangleright,\cB)$ such that the morphisms $\alpha_c$ for $c\in\cC$ and $\alpha_*$ are all identities. We will denote this Segal space by $\coc_\cB^\lax(F,b)$. We will call $\widetilde{b}\in\cB$ a \textit{lax colimit} of $F$ and write 
\[b\cong\underset{c\in\cC}{\colim^\lax}F(c)\]
if $\coc^\lax(F,b)\cong\mor_\cB(\widetilde{b},b)$. We define lax cones and lax limits to be the lax cocones and lax colimits in $\cB^\op$.
\end{defn}
\begin{prop}\label{prop:colim_cocone}
There is a natural equivalence
\[\coc_\cB^\lax(F,b)\cong \underset{f:[n]\rightarrow\cC}{\lim}\coc_\cB^\lax(F\circ f,b).\]
\end{prop}
\begin{proof}
First, observe that just like in the proof of \cref{prop:colim_lax} it is enough to prove the statement for the underlying space of objects. Now recall the equivalence
\[\mor_\cat^\lax(\cC^\triangleright,\cB)\cong\underset{f:[n]\rightarrow\cC^\triangleright}{\lim}\mor_\cat^\lax([n],\cB)\]
from the same proposition. By \cite[Corollary 3.3.3.4.]{lurie2009higher} this limit is isomorphic to the space of coCartesian sections $s$ of the left fibration $p:\cF\rightarrow\ccat_{/\cC^\triangleright}$ sending $(f:[n]\rightarrow\cC^\triangleright)$ to $\mor_\cat^\lax([n],\cB)$. Observe that the space $\coc_\cB^\lax(F,b)$ is isomorphic to the space of those sections that send $f$ to $\{F\circ f\}$ if $f([n])\subset\cC$ and for which $s(f)$ factors through $\coc^\lax_\cB(F\circ f,b)\rightarrow\mor_\cat^\lax([n],\cB)$ if $f(n)=*$. The required isomorphism easily follows.
\end{proof}
\begin{remark}\label{rem:monad_mod}
It follows from the definitions that for a lax functor $*\overset{T}{\rightsquigarrow}\cB$ (which is equivalently a monad by \cref{cor:lax_mon}) and an object $b$ the category $\coc_\cB^\lax(T,b)$ is equivalent to the category of $T$-algebras in $\mor_\cB(T(*),b)$. Indeed, denote by $\ccat^\triangleright_{/[1]}$ the subcategory of $\ccat_{/[1]}$ on those $g:[n+1]\rightarrow[1]$ for which $g^{-1}(\{1\})\cong\{n+1\}$. Then by definition giving an object of $\coc_\cB^\lax(T,b)$ is equivalent to giving a morphism $\ccat^\triangleright_{/[1]}$ that send $g$ as above to $T^n f$ for some $f\in\mor_\cB(T(*),b)$. It is easy to see that $\ccat^\triangleright_{/[1]}\cong (\Delta_{-\infty})^\op$. Using this equivalence and the structure of a monad on $T$, it is easy to see that functors like this can be equivalently defined as sections $s:\Delta^\op\rightarrow\mor_{\cB}(T(*),b)^T_f$ satisfying the conditions described in the proof of \cref{prop:nat} and those correspond to the $T$-algebra structures on $f$.
\end{remark}
\begin{construction}
For an integer $n\geq0$ denote by $[1]\widetilde{\otimes}[n]$ the category whose objects are pairs $(i,t)$ where $0\leq i\leq n$ and $t\in\{0,1\}$ and such that $\mor_{[1]\widetilde{\otimes}[n]}((i,0),(i',0))\cong\mor_{[n]}(i,i')$, $\mor_{[1]\widetilde{\otimes}[n]}((j,1),(j',1))\cong\mor_{[n]}(j',j)$, the category $\mor_{[1]\widetilde{\otimes}[n]}((i,1),(j,0))$ is empty and the category $\mor_{[1]\widetilde{\otimes}[n]}((i,0),(j,1))$ has object given by cospans $i\xrightarrow{f}k\xleftarrow{g}j$ and morphisms given by diagrams
\[
\begin{tikzcd}
{}&k\arrow[dd, "h"]\\
i\arrow[ur, "f"]\arrow[dr, "f'"]&{}&j\arrow[ul, "g" swap]\arrow[dl, "g'" swap]\\
{}&k'
\end{tikzcd}.
\]
The construction is obviously functorial in $[n]$, and so for any twofold Segal space $\cB$ we can define a functor $\Delta^\op\rightarrow\cS$ sending $[n]$ to $\mor_\cat([1]\widetilde{\otimes}[n],\cB)$ which is easily seen to be a twofold Segal space. We will call the resulting Segal space the \textit{lax twisted arrow category} and denote it by $\twar^\lax(\cB)$.
\end{construction}
\begin{lemma}\label{lem:twar}
For a twofold Segal space $\cB$ denote by $\arr^{\lax,\radj}(\cB)$ (resp. $\arr^{\lax,\ladj}(\cB)$) the subcategory of $\arr^\lax(\cB)$ with the same objects but with morphisms given by $[1]\otimes[n]\rightarrow\cB$ that send all morphisms in $\mor_{[1]\otimes[n]}((i,0),(j,0))$ and in $\mor_{[1]\otimes[n]}((i',1),(j',1))$  to right (resp. left) adjoints and similarly denote by $\twar^{\lax,\ladj}(\cB)$ (resp. $\twar^{\lax,\radj}(\cB)$) the subcategory of $\twar^\lax(\cB)$ with the same objects but with morphisms given by $[1]\widetilde{\otimes}[n]$ that send the morphism in  $\mor_{[1]\widetilde{\otimes}[n]}((i,0),(j,0))$ to right (resp. left) adjoints and morphisms in $\mor_{[1]\widetilde{\otimes}[n]}((i',1),(j',1))$ to left (resp. right) adjoints. Then there are isomorphisms
\[\arr^{\lax,\radj}(\cB)\cong\twar^{\lax,\ladj}(\cB)\]
and 
\[\arr^{\lax,\ladj}(\cB)\cong\twar^{\lax,\radj}(\cB).\]
\end{lemma}
\begin{proof}
We will only prove the first isomorphism, since the proof of the second is completely analogous. For every $n\geq0$ denote by $[n]^\radj$ (resp. $[n]^\ladj$) the category obtained by adding all right (resp. left) adjoints to $[n]$. Observe that there is an isomorphism $\cI:[n]^{\ladj,\op}\xrightarrow{\sim}[n]^\radj$ that sends every arrow in $[n]^\ladj$ to its right adjoint. By definition, all morphisms $[n]\otimes[1]\rightarrow\cB$ in $\arr^{\lax,\radj}(\cB)_n$ factor through $[n]^\radj\otimes[1]$, where $[n]^\radj\otimes[1]$ is the category with objects $(i,t)$ where $i\in[n]^\radj$ and $t\in\{0,1\}$ such that \[\mor_{[n]^\radj\otimes[1]}((i,t),(j,t))\cong\mor_{[n]^\radj}(i,j)\]
for $t=0$ and $t=1$, the category $\mor_{[n]^\radj\otimes[1]}((i,1),(j,0))$ is empty and $\mor_{[n]^\radj\otimes[1]}((i,0),(j,1))$ has objects given by strings $i\xrightarrow{f}k\xrightarrow{g}j$ in $[n]^\radj$ and morphisms by diagrams
\begin{equation}\label{eq:two}
    \begin{tikzcd}[row sep=huge, column sep=huge]
    i\arrow[r, "f"]\arrow[dr, "f'" swap]&k\arrow[d, "h"]\arrow[r, "g"]\arrow[ld, Rightarrow, shorten >=40 pt, "\alpha" swap, near start]\arrow[dr, Rightarrow, shorten >=40 pt, "\beta", near start]&j\\
    {}&k'\arrow[ur, "g'" swap]&{}
    \end{tikzcd}
\end{equation}
in $[n]^\radj$.\par
Similarly, every morphism $[n]\widetilde{\otimes}[1]\rightarrow\cB$ in $\twar^{\lax,\ladj}(\cB)_n$ factor through $[n]^\ladj\widetilde{\otimes}[1]$, where $[n]^\ladj\widetilde{\otimes}[1]$ is the category with objects $(i,t)$ where $i\in[n]^\ladj$ and $t\in\{0,1\}$ such that $\mor_{[n]^\ladj\widetilde{\otimes}[1]}((i,0),(j,0))\cong\mor_{[n]^\ladj}(i,j)$ and $\mor_{[n]^\ladj\widetilde{\otimes}[1]}((i,1),(j,1))\cong\mor_{[n]^\ladj}(j,i)$, the category $\mor_{[n]^\ladj\widetilde{\otimes}[1]}((i,1),(j,0))$ is empty and $\mor_{[n]^\ladj\widetilde{\otimes}[1]}((i,0),(j,1))$ has objects given by cospans $i\xrightarrow{s}l\xleftarrow{t}j$ in $[n]^\ladj$ and morphisms by diagrams
\begin{equation}\label{eq:three}
    \begin{tikzcd}[row sep=huge, column sep=huge]
    i\arrow[r, "s"]\arrow[dr, "s'" swap]&l\arrow[d, "v"]\arrow[ld, Rightarrow, shorten >=40 pt, "\gamma" swap, near start]\arrow[dr, Leftarrow, shorten >=40 pt, "\delta", near start]&j\arrow[l, "t" swap]\arrow[ld, "t'"]\\
    {}&l'&{}
    \end{tikzcd}.
\end{equation}
To finish the proof of the lemma, we will provide an isomorphism $\widetilde{\cI}:[n]^\ladj\widetilde{\otimes}[1]\xrightarrow{\sim}[n]^\radj\otimes[1]$. The restriction of $\widetilde{\cI}$ to the fiber over $0$ is given by the identity morphism, restriction to the fiber over $1$ is given by $\cI$. Finally, a morphism $i\xrightarrow{s}l\xleftarrow{t}j$ is sent to $i\xrightarrow{s}l\xrightarrow{t^R}j$ where $t^R$ denotes the left adjoint. The 2-morphism represented by the diagram \ref{eq:three} is then sent to 
\[(v\circ s\xrightarrow{\gamma}s', t^R\xrightarrow{\eta}t'^R\circ t'\circ t^R\xrightarrow{\delta}t'^R\circ v\circ t\circ t^R\xrightarrow{\epsilon}t'^R\circ v).\] The inverse $\widetilde{\cI}^{-1}$ restricts to $\id$ on the fiber over $0$, to $\cI^{-1}$ on the fiber over $1$ and sends diagram \ref{eq:two} to 
\[(h\circ f\xrightarrow{\alpha}f',g'^L\xrightarrow{\eta}g'^L\circ g\circ g^L\xrightarrow{\beta}g'^L\circ g'\circ h\circ g^L\xrightarrow{\epsilon}h\circ g^L).\]
The fact that those morphisms are inverse to each other follows from elementary calculations with triangle identities. 
\end{proof}
\begin{cor}\label{cor:lax_adj}
Denote by $\arr^{\lax,!}(\corr)$ the full flagged subcategory of (the underlying Segal space of) $\arr^\lax(\corr)$ on the objects of the form $f_!$. Then for all $k\geq0$ we have an isomorphism 
\[\arr^{\lax,!}(\corr)^\sim_k\cong\mor_\cat([1],\corr_k)^\sim.\]
\end{cor}
\begin{proof}
First, it follows from the Segal condition that it would suffice to provide an isomorphism for $k=0$ and $k=1$. Observe that for $k=0$ both sides are obviously isomorphic to $\mor_\cat([1],\cat)^\sim$. For $k=1$ observe that the objects of $\arr^{\lax,!}(\corr)^\sim_1$ are given by diagrams of the form 
\[
\begin{tikzcd}[row sep=huge, column sep=huge]
\cC\arrow[r, "M"]\arrow[d, "f_!"]&\cD\arrow[d, "g_!"]\arrow[dl, Rightarrow, "\alpha"]\\
\cE\arrow[r, "N"]&\cA
\end{tikzcd}
\]
in $\corr$. It follows from the mate correspondence of \cref{lem:twar} that the space of such diagrams is isomorphic to the space 
of natural transformations $\widetilde{\alpha}:M\rightarrow g^*\circ N\circ f_!$. A calculation similar to the one in \cref{prop:fact_1} shows that 
\[g^*\circ N\circ f_!\cong \cC_0\times_{\cE_0}N\times_{\cK_0}\cD_0,\]
so it follows that this data indeed is equivalent to a morphism in $\corr_1$.
\end{proof}
\section{Lax functors to Corr}\label{sect:five}
This section is the technical heart of the paper. Our main result is \cref{thm:lax_corr} that describes the category of lax functors from $\cC$ to $\corr$ in terms of the category $\cat^\wrr_{/\cC}$ whose description can be found in \cref{prop:cat_ic}. Our second main result \cref{thm:lax_colim} provides an explicit description for a lax colimit of a lax functor $\cC\rightsquigarrow\corr$. 
\begin{construction}
Denote by $\ccat^\rightarrow$ the opposite of the category whose objects are intervals $\langle m\rangle$ such that every elementary morphism $\{i,i+1\}$ is marked with either $0$ or $1$ and whose morphisms from $\langle m\rangle$ to $\langle n \rangle$ are given by ordinary morphisms of intervals $([m]\rightarrow[n])$ that carry each elementary morphism of $[m]$ marked with $0$ to a composition of morphisms of $[n]$ all of which are marked with $0$. Call a morphism $i:\langle m\rangle\rightarrowtail \langle n \rangle$ inert if the underlying morphism of intervals $([m]\overset{i}{\rightarrowtail}[n])$ is inert and active if the underlying morphism of intervals is active and in addition only carries morphisms marked with $1$ to the composition of morphisms marked with $1$. Call an object elementary if the underlying object of $\ccat$ is elementary. We will denote by $[1]$ the elementary segment marked with $0$ and by $\langle1\rangle$ the elementary segment marked with 1.
\end{construction}
\begin{lemma}
The above definition endows $\ccat^\rightarrow$ with the structure of an algebraic pattern.
\end{lemma}
\begin{proof}
We need to demonstrate that active and inert morphism indeed form a factorization system on $\ccat^\rightarrow$. Let $(\langle m \rangle\xrightarrow{f}\langle n\rangle)$ be a morphism in $(\ccat^\rightarrow)^\op$, we will provide an active/inert factorization for it. First, let 
\[[m]\overset{a}{\twoheadrightarrow}[l]\overset{i}{\rightarrowtail}[n]\]
be the usual active/inert factorization of $f$ in $\Delta$. Observe that there is a unique way of promoting $[l]$ to an object $\langle l\rangle$ such that $a$ remains an active morphism in $(\ccat^\rightarrow)^\op$, namely we mark an interval in $[l]$ with $0$ if it is in the image of a $0$-marked interval in $\langle m\rangle$ and with $1$ if it is in the image of a $1$-marked interval in $\langle m \rangle$. The morphism $i$ then automatically lifts to an inert morphism in $(\ccat^\rightarrow)^\op$ and it is obvious that this decomposition is unique.
\end{proof}
\begin{remark}\label{rem:seg_cat}
For $\bm\in\ccat^\rightarrow$ denote by $(\ccat^\rightarrow)^{\el,\mc}_{\bm/}$ the full subcategory of $(\ccat^\rightarrow)^\el_{/[n]}$ on those inert morphisms $\bm\rightarrowtail e$ that take edges marked with $i$ to edges marked with $i$ for $i\in\{0,1\}$, then it is easy to see that $(\ccat^\rightarrow)^{\el,\mc}_{\bm/}$ is coinitial. In other words, a morphism $\cF:\ccat^\rightarrow\rightarrow\cS$ is Segal if and only if
\[\cF(\bm)\cong\overbrace{\cF(e_1)\times_{\cF(1)}\cF(e_2)\times_{\cF(2)}...\times_{\cF(m-1)}\cF(e_m)}^\text{$m$ times},\]
where $\cF(i)$ denotes the value of $\cF$ on the morphism $[0]\xrightarrow{\{i\}}[m]$ and $e_j$ denotes the $j$th elementary edge of $\bm$.
\end{remark}
\begin{construction}
Denote by $\widetilde{\ccat^\rightarrow}$ the full subcategory of $[1]\times\ccat$ containing all objects with the exception of $(0,[0])$. Denote by $i:\widetilde{\ccat^\rightarrow}\hookrightarrow[1]\times\ccat$ the natural inclusion and declare a morphism in $\widetilde{\ccat^\rightarrow}$ active (resp. inert) if its image in $[1]\times \ccat$ is active (resp. inert) and declare an object elementary if its image under $i$ is elementary (here we endow $[1]\times \ccat$ with the structure of an algebraic pattern given by \cref{constr:tensor}).
\end{construction}
\begin{lemma}\label{lem:pat_bo}
$\widetilde{\ccat^\rightarrow}$ is an algebraic pattern such that the category of Segal spaces for it is isomorphic to the category of triples $(\cC,\cD,F:\cC\rightarrow\cD)$ where both $\cC$ and $\cD$ are Segal $\ccat$-spaces and $F$ is an identity on objects functor.
\end{lemma}
\begin{proof}
Observe that every morphism in $\widetilde{\ccat^\rightarrow}$ can be decomposed as $(f,\id_{[1]})\circ (\id_\ccat,t)$ where $t$ is the unique nontrivial morphism in $[1]$. Since $t$ is considered to be inert, we see that $\widetilde{\ccat^\rightarrow}$ is indeed an algebraic pattern. It is elementary to see that $i_*$ restricts to a functor $\seg_{\widetilde{\ccat^\rightarrow}}(\cS)\rightarrow\seg_{[1]\times\ccat}(\cS)$ and moreover $i_*\cF(i(a))\cong\cF(a)$ for every $\cF\in\seg_{\widetilde{\ccat^\rightarrow}}(\cS)$ and $a\in\widetilde{\ccat^\rightarrow}$ and that $i_*\cF((0,[0]))\cong\cF((1,[0]))$. Since $i$ is fully faithful, we see that $i_*$ induces an isomorphism between $\seg_{\widetilde{\ccat^\rightarrow}}(\cS)$ and the full subcategory of $\seg_{[1]\times\ccat}(\cS)$ on those objects $\cE$ such that $i_*i^*\cE\cong\cE$. Finally, observe that an object $\cE$ of $\seg_{[1]\times\ccat}(\cS)$ can be identified with a pair of Segal spaces $\cC$ and $\cD$ together with a functor $F:\cC\rightarrow\cD$ and the condition $i_*i^*\cE\cong\cE$ means precisely that $F$ induces an isomorphism on objects.
\end{proof}
\begin{prop}\label{prop:pat_bo}
The category of Segal spaces for $\ccat^\rightarrow$ is isomorphic to the category of triples $(\cC,\cD,F:\cC\rightarrow\cD)$ where both $\cC$ and $\cD$ are Segal $\ccat$-spaces and $F$ is an isomorphism-on-objects functor.
\end{prop}
\begin{proof}
Observe that there is a fully faithful functor $\widetilde{i}:\widetilde{\ccat^\rightarrow}\hookrightarrow\ccat^\rightarrow$ that sends $([n],0)$ to the interval of length $n$ where every morphism is marked with $0$ and sends $([m],1)$ to the interval of length $m$ where every morphism is marked with $1$. Also observe that 
\[(\ccat^\rightarrow)^\el_{\widetilde{i}(x)/}\cong(\widetilde{\ccat^\rightarrow})^\el_{x/}\]
for any $x\in\widetilde{\ccat^\rightarrow}$, so $\widetilde{i}$ is indeed a Segal morphism. Since active morphisms are required to preserve the markings, we also see that $\widetilde{i}$ lifts active morphisms uniquely, so by \cite[Proposition 6.3]{chu2019homotopy} $\widetilde{i}_*$ restricts to a functor between the categories of respective Segal objects and moreover 
\[\widetilde{i}_*\cF(a)\cong\underset{(a\rightarrowtail\widetilde{i}(b))\in(\ccat^\rightarrow)^\inrt_{/\widetilde{i}}}{\lim}\cF(b).\]
Since $\widetilde{i}$ is fully faithful, we have $i^*i_*\cong\id$, so it suffices to prove that $i_*i^*\cong\id$. For this, first denote by $(\ccat^\rightarrow)^{\inrt,\maxm}_{/\widetilde{i}}$ the full subcategory of $(\ccat^\rightarrow)^{\inrt}_{/\widetilde{i}}$ on those $(a\rightarrowtail\widetilde{i}(b))$ that correspond to the inclusions of maximal "monochromatic" subintervals, i.e. to inclusions $\langle b\rangle\hookrightarrow\langle a\rangle$ such that every element of $\langle b\rangle$ is marked with the same number and there is no strictly bigger subinterval of $\langle a\rangle$ with this property. Then it is obvious that the inclusion $(\ccat^\rightarrow)^{\inrt,\maxm}_{/\widetilde{i}}\hookrightarrow(\ccat^\rightarrow)^{\inrt}_{/\widetilde{i}}$ is coinitial. Now it suffices to prove that for $\cE\in\seg_{\ccat^\rightarrow}(\cS)$ we have
\[\cE(a)\cong \underset{(a\rightarrowtail\widetilde{i}(b))\in(\ccat^\rightarrow)^{\inrt,\maxm}_{/\widetilde{i}}}{\lim}\cE(b).\]
Using that $\cE$ is a Segal space, we can rewrite the right term of the equation as 
\[\underset{(a\rightarrowtail\widetilde{i}(b))\in(\ccat^\rightarrow)^{\inrt,\maxm}_{/\widetilde{i}}}{\lim}\;(\underset{(b\rightarrowtail e)\in(\ccat^\rightarrow)^\el_{b/}}{\lim}\cE(e))\cong\underset{(\ccat^\rightarrow)^{\el,\maxm}_{a/}}{\lim}\cE(e)\]
where $(\ccat^\rightarrow)^{\el,\maxm}_{a/}$ denotes the category whose objects are composable pairs $a\rightarrowtail b\rightarrowtail e$ where $e$ is elementary and $b$ is a maximal monochromatic subinterval of $a$. Finally, observe that $(\ccat^\rightarrow)^{\el,\maxm}_{a/}\cong (\ccat^\rightarrow)^\el_{a/}$ since maximal monochromatic subintervals cover $\langle a\rangle$ and only intersect along boundary points.
\end{proof}
\begin{construction}\label{constr:overset}
Given a Segal space $\cC$, denote by $\cC^\rightarrow$ the following presheaf on $(\ccat^\rightarrow)^\op$: for $\bn\in\ccat^\rightarrow$ let $\cC^\rightarrow(\bn)$ be the space of morphisms $f:[n]\rightarrow\cC$ in $\cat$ that send edges marked with $0$ to identity morphisms. It is easy to see that, if $g:\bm\rightarrow\bn$ is a morphism in $(\ccat^\rightarrow)^\op$, then $f\circ g$ also satisfies the conditions to be an element of $\cC^\rightarrow(\bm)$.
\end{construction}
\begin{notation}\label{not:nun}
Denote by $\ccat^\nun_{/\cC}$ the full subcategory of $\ccat_{/\cC}$ on those morphisms $[n]\rightarrow\cC$ that take all non-identity morphisms in $[n]$ to non-identity morphisms in $\cC$. Observe that the natural inclusion $\ccat_{/\cC}^\nun\hookrightarrow\ccat_{/\cC}$ is cofinal since every $[n]\xrightarrow{f}\cC$ uniquely factors as $[n]\xrightarrow{s}[m]\xrightarrow{\widetilde{f}}\cC$, where $\widetilde{f}\in\ccat^\nun_{/\cC}$ and $s$ is a surjective morphism.
\end{notation}
\begin{prop}\label{prop:wrr_colim}
The $\cC^\rightarrow$ is a Segal space, so in particular its total space obtains a structure of an algebraic pattern, which we will denote by $\ccat^\wrr_{/\cC}$. Moreover, we have 
\[\cat^\wrr_{/\cC}\cong \underset{([n]\rightarrow\cC)\in\ccat^\nun_{/\cC}}{\lim}\cat^\wrr_{/[n]}.\]
\end{prop}
\begin{proof}
The fact that $\cC^\rightarrow$ is a Segal space follows from the isomorphism \[\mor_\cat([n],\cC)\cong\underset{([n]\rightarrowtail e)\in\ccat^\el_{[n]/}}{\lim}\mor_\cat(e,\cC)\] 
together with the observation that if $(\bn\xrightarrow{f}\cC)$ satisfies the condition of \cref{constr:overset}, then so do all the elementary subintervals $\langle1\rangle\hookrightarrow\bn$ of $\bn$.\par
To prove the second claim, it would suffice to prove that 
\[\cC^\rightarrow\cong\underset{([n]\rightarrow\cC)\in\ccat^\nun_{/\cC}}{\colim}[n]^\rightarrow\]
as presheaves on $(\ccat^\rightarrow)^\op$ since then we can use the same argument as in \cref{cor:slice}. To prove this, first observe that both $\cC^\rightarrow$ and each of the $[n]^\rightarrow$ are colimits of their Yoneda cocones in $\cP((\ccat^\rightarrow)^\op)$. Denote by $\ccat^{\wrr,\nun}_{/\cC}$ the category whose objects are morphisms $\bn\rightarrow[m]\rightarrow\cC$, where the first morphism belongs to $\ccat^\wrr_{/[m]^\rightarrow}$ and the second to $\ccat^\nun_{/\cC}$, and the morphisms are diagrams
\[
\begin{tikzcd}[column sep=huge]
\bn\arrow[r]\arrow[dd,"f"]&{[m]}\arrow[dd, "g"]\arrow[dr]\\
{}&{}&\cC\\
\langle n'\rangle\arrow[r]&{[m']}\arrow[ur]
\end{tikzcd},
\]
where $f$ is a morphism in $\ccat^\rightarrow$. Observe that from the definition of $\ccat^\nun_{/\cC}$ it follows that the composition of arrows restricts to a functor $m:\ccat^{\wrr,\nun}_{/\cC}\rightarrow \ccat^\rightarrow_{/\cC^\rightarrow}$ and it now suffices to prove that it is cofinal, since we would then have the following string of isomorphisms
\[\cC^\rightarrow\cong\underset{(\bn\rightarrow\cC)\in\ccat^\rightarrow_{/\cC^\rightarrow}}{\colim}\bn\cong \underset{(\bn\rightarrow[m]\rightarrow\cC)\in\ccat^\wrr_{/\cC}}{\colim}\bn\cong \underset{([m]\rightarrow\cC)\in\ccat^\nun_{/\cC}}{\colim}\underset{(\bn\rightarrow[m])\in\ccat^\rightarrow_{/[m]^\rightarrow}}{\colim}\bn\cong \underset{([m]\rightarrow\cC)\in\ccat^\nun_{/\cC}}{\colim}[m]^\rightarrow.\]
In other words, we need to prove that for every $(\bm\xrightarrow{f}\cC)\in\cC^\rightarrow$ the category of diagrams of the form 
\[
\begin{tikzcd}[row sep=huge, column sep=huge]
\bn\arrow[d]\arrow[r]&{[n']}\arrow[dr,"g"]\\
\bm\arrow[rr,"f"]&{}&\cC
\end{tikzcd},
\]
where $g\in\ccat^\nun_{/\cC}$ with the obvious notion of morphisms between them is contractible. Observe that this category has an initial object given by $(\bm\xrightarrow{s}[m']\xrightarrow{\widetilde{f}}\cC)$, where $[m']$ is obtained from $[m]$ by contracting all intervals marked with $0$ and $f\cong \widetilde{f}\circ s$ is the factorization from \cref{not:nun}.
\end{proof}
\begin{prop}\label{prop:cat_ic}
$\cat^\wrr_{/\cC}\bydef\seg_{\ccat^\wrr_{/\cC}}(\cS)$ is the category of triples \[(F:\cE\rightarrow\cC,G:\cD:\rightarrow\cC,H:\cE\rightarrow\cD),\] 
where $H$ is a functor over $\cC$ that induces an isomorphism on the space of objects and also an isomorphism $(\cE_f)^\sim\overset{\sim}{\rightarrow}(\cD_f)^\sim$ for any non-identity morphism $f$ in $\cC$ (where $(\cD_f)^\sim$ and $(\cE_f)^\sim$ denote the spaces of morphisms lying over $f$).
\end{prop}
\begin{proof}
First, observe that there is a fully faithful inclusion $i_\cC:\ccat^\wrr_{/\cC}\hookrightarrow\ccat^\rightarrow_{/\cC}$, where in the codomain we view $\cC$ as a Segal $\ccat^\rightarrow$-space corresponding to the triple $(\cC,\cC,\id_{\cC})$ using the notation of \cref{lem:pat_bo}. Using \cref{prop:slice} we see that the Segal $\ccat^\rightarrow_{/\cC}$-spaces can be identified with triples $(F,G,H)$ as in the statement of the proposition, except $H$ is only required to be an isomorphism on objects. Observe that since active morphisms preserve markings, $i_\cC$ admits unique lifting of active morphisms, so by \cite[Proposition 6.3.]{chu2019homotopy} $i_{\cC,*}$ restricts to a functor between the categories of Segal spaces. Since $i$ is an inclusion of a full subcategory, we see that $\cat^\wrr_{/\cC}$ can be identified with those Segal spaces $\cF\in\seg_{\ccat^\rightarrow_{/\cC}}(\cS)$ for which $\cF\cong i_{\cC,*}i^*_\cC \cF$. It is easy to see that $i_{\cC,*}i^*_\cC \cF(e)\cong \cF(e)$ for elementary objects $e$ of $\ccat^\rightarrow_{/\cC}$ of the form $([0]\xrightarrow{\{c\}}\cC)$ for $c\in\cC$, $(\langle1\rangle\xrightarrow{g}\cC)$ for all morphisms $g$ in $\cC$ and $([1]\xrightarrow{\id_{c'}}\cC)$ for all $c'\in\cC$, but \[i_{\cC,*}i^*_\cC\cF([1]\xrightarrow{f}\cC)\cong \cF(\langle1\rangle\xrightarrow{\overline{f}}\cC),\]
where $f$ sends an edge marked with $0$ to a non-identity morphism and $\overline{f}$ has the same underlying morphism but for which the edge is marked with 1. In other words, the full subcategory in question consists of those $\cF$ that have the same value on non-identity morphisms $\langle1\rangle\rightarrow\cC$ regardless of the marking, which concludes the proof.
\end{proof}
\begin{notation}\label{not:forget}
Given $\bm\in\ccat^\rightarrow$, denote by $[\widetilde{m}]$ the interval obtained from $[m]$ by contracting all edges marked with $0$. If we are additionally given $(\bm\xrightarrow{f}[n])\in\ccat^\wrr_{/[n]}$, we see that, since those edges are all sent to identity morphisms, $f$ gives rise to $\widetilde{f}:[\widetilde{m}]\rightarrow[n]$ that sends all edges in $[\widetilde{m}]$ to the image in $[n]$ of their unique preimage in $\bm$. Similarly, every morphism $\bm\rightarrow\bl$ induces a morphism $[\widetilde{m}]\rightarrow[\widetilde{l}]$, and so we obtain a functor $U:\ccat^\wrr_{/[n]}\rightarrow\ccat_{/[n]}$. Denote by \[(\ccat_{/[n]})^\inrt_{/U}\bydef \arr^\inrt(\ccat_{/[n]})\times_{\ccat_{/[n]}}\ccat^\wrr_{/[n]}\]
the category whose objects are pairs $(f,g\rightarrowtail U(f))$ and morphisms are given by commutative diagrams of the form 
\[
\begin{tikzcd}[row sep=huge, column sep=huge]
g'\arrow[d,tail]\arrow[r, "\alpha"]&g\arrow[d, tail]\\
U(f')\arrow[r, "U(\beta)"]&f
\end{tikzcd}.
\] Denote by $p:(\ccat_{/[n]})^\inrt_{/U}\rightarrow\ccat_{/[n]}$ the natural projection.
\end{notation}
\begin{prop}\label{prop:forget}
The fiber of $p$ over $g:[k]\rightarrow[n]$ is isomorphic to $\ccat^\cell_{/[k]}$ and moreover for any morphism $h:[k]\rightarrow[s]$ fitting in the commutative diagram
\[
\begin{tikzcd}
{[k]}\arrow[rr, "h"]\arrow[rd, "g"]&{}&{[s]}\arrow[ld, "v" swap]\\
{}&{[n]}
\end{tikzcd}
\]
the fiber 
\[(\ccat^\wrr_{/[n]})_h\bydef\ccat^\wrr_{/[n]}\times_{\ccat_{/[n]}}[1],\] 
when viewed as a correspondence from $(\ccat^\wrr_{/[n]})_{[s]}$ to $(\ccat^\wrr_{/[n]})_{[k]}$, is isomorphic to $h^*_\cell$ in the notation of \cref{rem:corr_corr}.
\end{prop}
\begin{proof}
Observe that if we view $f:\bm\rightarrow[n]$ as an ordinary morphism in $\cat$, then it factors as $([m]\xrightarrow{s_f}[\widetilde{m}]\xrightarrow{\widetilde{f}}[n])$, where $s_f$ sends all morphisms marked with $0$ to identity morphisms in $[\widetilde{m}]$. It is immediate from the definition that the fiber of $p$ over $g$ has pairs $(\bm\xrightarrow{f}[n],[\widetilde{m}]\overset{i}{\rightarrowtail}[k])$ where $i$ is an inert morphism over $[n]$ in $\Delta$ and $g\circ i\cong \widetilde{f}$ as objects and a morphism from $(\bl\xrightarrow{t}[n],[\widetilde{l}]\overset{j}{\rightarrowtail}[k])$ to $(\bm\xrightarrow{f}[n],[\widetilde{m}]\overset{i}{\rightarrowtail}[k])$ given by a diagram in $\Delta$ of the form
\[
\begin{tikzcd}[column sep=huge]
\bm\arrow[r,"s_f"]\arrow[dd, "u"]&{[\widetilde{m}]}\arrow[dd, "U(u)"]\arrow[rd, tail, "i" swap]\arrow[rrd,"\widetilde{f}"]\\
{}&{}&{[k]}\arrow[r, "g"]&{[n]}\\
\bl\arrow[r, "s_t"]&{[\widetilde{l}]}\arrow[ur, tail, "j"]\arrow[urr, "\widetilde{t}" swap]
\end{tikzcd},
\]
where $u$ is a morphism in $\ccat^\wrr_{/[n]}$. Observe that morphisms $i\circ s_f$ and $j\circ s_t$ in the diagram above can be viewed as elements of $\ccat^\cell_{/[k]}$, and a morphism $u$ can also be viewed as a morphism in this category, so this defines a functor $F:p_g\rightarrow\ccat^\cell_{/[k]}$.\par
Conversely, given a cellular morphism $c:[m]\rightarrow[k]$ we can first decompose it as $([m]\xrightarrow{s_c}[\widetilde{m}]\overset{i_c}{\rightarrowtail}[k])$, where $s_c$ is surjective and $i_c$ is inert. Now we mark an edge of $[m]$ with $1$ if its image under $c$ is a nontrivial edge in $[k]$ and with $0$ if $c$ sends it to an identity morphism. It is then easy to see that with these markings the morphism $g\circ c:\bm\rightarrow[n]$ can be viewed as an element of $\cat^\wrr_{/[n]}$ and the pair $(g\circ c:\bm\rightarrow [n], [\widetilde{m}]\overset{i_c}{\rightarrowtail}[k])$ as an element of $p_g$. Moreover, since edges marked with $0$ are precisely those that are sent to identity by $c$, we see that every morphism $h:[m]\rightarrow[l]$ over $[k]$ can also be uniquely lifted to a morphism in $p_g$, so this construction indeed defines a functor $G:\ccat^\cell_{/[k]}\rightarrow p_g$ and it is easy to see that $F$ and $G$ are mutually inverse.\par
To prove the second claim, first let $(\bm\xrightarrow{f}[n],[\widetilde{m}]\overset{i}{\rightarrowtail}[k])\in p_g$ and $(\langle o\rangle\xrightarrow{q}[n],[\widetilde{o}]\overset{j}{\rightarrowtail}[s])\in p_v$, then the correspondence in question assigns to them the space of morphisms $b$ making the following diagram commute
\[
\begin{tikzcd}[column sep=huge]
\bm\arrow[r, "s_f"]\arrow[dd, dotted, "b"]&{[\widetilde{m}]}\arrow[r, tail, "i"]\arrow[dd, dotted, "U(b)"]&{[k]}\arrow[dd,  "h"]\arrow[dr, "g"]\\
{}&{}&{}&{[n]}\\
\langle o\rangle\arrow[r, "s_o"]&{[\widetilde{o}]}\arrow[r, tail, "j"]&{[s]}\arrow[ur, "v"]
\end{tikzcd}.
\]
Similarly, given $([m]\xrightarrow{\overline{f}}[k])\in\ccat^\cell_{/[k]}$ and $[o]\xrightarrow{\overline{q}}[s]$, it follows from \cref{rem:corr_corr} that the correspondence $h_\cell^*$ sends them to the space of fillers $\overline{b}$ in the commutative diagram
\[
\begin{tikzcd}[row sep=huge, column sep=huge]
{[m]}\arrow[d, dotted, "\overline{b}"]\arrow[r, "\overline{f}"]&{[k]}\arrow[d, "h"]\\
{[o]}\arrow[r, "\overline{q}"]&{[s]}
\end{tikzcd}.
\]
However, it is easy to see that every morphism $\overline{b}$ as above uniquely lifts to $b:\bm\rightarrow\langle o\rangle$ making the other diagram commute, thus providing an isomorphism between the two correspondences.
\end{proof}
\begin{remark}\label{rem:lax_fun}
$\ccat^\wrr_{/[n]}$ can be viewed as a full subcategory of $(\ccat_{/[n]})^\inrt_{/U}$ on objects of the form $(f,\id_{U(f)})$. Using this fact, we can view $\cP((\ccat^\wrr_{/[n]})^\op)$ as a full subcategory of $\cP(((\ccat_{/[n]})^\inrt_{/U})^\op)$ on those $\cF$ for which $\cF(f,g\rightarrowtail U(f))\cong\cF(f,\id_{U(f)})$. This condition is equivalent to requiring that for an inert morphism $i:[k']\hookrightarrow[k]$ we have $\cF|_{p_{[k']\xrightarrow{g\circ i}[n]}}\cong i^*_\cell\cF|_{p_{[k]\xrightarrow{g}[n]}}$. Indeed, if this condition holds then we get \[\cF(f,g\rightarrowtail U(f))\cong\cF(f,\id_{U(f)})\]
by applying it to the inert morphism $g\rightarrowtail U(f)$. Conversely, if we assume $\cF(f,g\rightarrowtail U(f))\cong\cF(f,\id_{U(f)})$, then given $i:h\rightarrowtail g$ and an object of $(\ccat_{/[n]})^\inrt_{/U}$ of the form $(f, h\overset{i}{\rightarrowtail}g\rightarrowtail U(f))$ we see that 
\[\cF(f, h\overset{i}{\rightarrowtail}g\rightarrowtail U(f))\cong \cF(f, g\rightarrowtail U(f)),\] 
which is exactly the condition $\cF|_{p_{[k']\xrightarrow{g\circ i}[n]}}\cong i^*_\cell\cF|_{p_{[k]\xrightarrow{g}[n]}}$.
\end{remark}
\begin{construction}\label{constr:act_right}
Denote by $\ccat^{(\act,\rightarrow)}_{/[n]}$ the full subcategory of $\mor_\cat([2],\cat)$ on objects of the form $[m]\overset{a}{\twoheadrightarrow}[k]\xrightarrow{f}[n]$ and by $\ccat^{(\act,\wrr)}_{/[n]}$ the category with objects given by $\bm\overset{a}{\twoheadrightarrow}[k]\xrightarrow{f}[n]$, where $a$ takes edges marked with $0$ to identity, and morphisms are given by diagrams
\[
\begin{tikzcd}[column sep=huge]
\bm\arrow[r, two heads, "a"]\arrow[dd, "h"]&{[k]}\arrow[dd, "t"]\arrow[dr, "f"]\\
{}&{}&{[n]}\\
\bl\arrow[r, two heads, "a'"]&{[s]}\arrow[ur, "f'"]
\end{tikzcd},
\]
where $h$ is a morphism in $\ccat^\rightarrow$. Denote by $U:\ccat^{(\act,\wrr)}_{/[n]}\rightarrow\ccat^{(\act,\rightarrow)}_{/[n]}$ the functor sending $\bm\overset{a}{\twoheadrightarrow}[k]\xrightarrow{f}[n]$ to $[\widetilde{m}]\overset{\widetilde{a}}{\twoheadrightarrow}[k]\xrightarrow{f}[n]$ where $[\widetilde{m}]$ is obtained from $\bm$ by contracting all edges marked with $0$; the same argument as in \cref{not:forget} shows that $U$ is indeed a functor. Denote by \[(\ccat^{(\act,\rightarrow)}_{/[n]})^\inrt_{/U}\bydef \arr^\inrt(\ccat^{(\act,\rightarrow)}_{/[n]})\times_{\ccat^{(\act,\rightarrow)}_{/[n]}}\ccat^{(\act,\wrr)}_{/[n]}\]
the analog of the category from \cref{not:forget} and by $\widehat{p}:(\ccat^{(\act,\rightarrow)}_{/[n]})^\inrt_{/U}\rightarrow\ccat^{(\act,\rightarrow)}_{/[n]}$ the obvious projection. Finally, observe that there is a fully faithful embedding $i:(\ccat_{/[n]})^\inrt_{/U}\hookrightarrow(\ccat^{(\act,\rightarrow)}_{/[n]})^\inrt_{/U}$ sending $(\bm\xrightarrow{f}[n],[\widetilde{m}]\overset{j}{\rightarrowtail}[q])$ to the pair consisting of $(\bm\overset{s_f}{\twoheadrightarrow}[\widetilde{m}]\xrightarrow{\widetilde{f}}[n])$ and the diagram
\[
\begin{tikzcd}[column sep=huge]
{[\widetilde{m}]}\arrow[dd, tail, "j"]\arrow[r, equal]&{[\widetilde{m}]}\arrow[dr, "\widetilde{f}"]\arrow[dd, tail, "j"]\\
{}&{}&{[n]}\\
{[k]}\arrow[r, equal]&{[k]}\arrow[ur, "g"]
\end{tikzcd}.
\]
We will also denote by $\overline{i}:\ccat_{/[n]}\hookrightarrow\ccat^{(\act,\rightarrow)}_{/[n]}$ its restriction to $\ccat_{/[n]}$.
\end{construction}
\begin{lemma}\label{lem:forget_2}
If $([k]\overset{b}{\twoheadrightarrow}[s]\xrightarrow{t}[n])$ represents an object of $\ccat^{(\act,\rightarrow)}_{/[n]}$, then we have a sequence of isomorphisms
\[\widehat{p}_{([k]\overset{b}{\twoheadrightarrow}[s]\xrightarrow{t}[n])}\cong p_{([k]\xrightarrow{t\circ b}[n])}\cong\ccat^\cell_{/[k]}.\]
Moreover, for any morphism
\[
\begin{tikzcd}[column sep=huge]
{[k]}\arrow[dd, "w"]\arrow[r,two heads, "b"]&{[s]}\arrow[dd, "u"]\arrow[dr, "t"]\\
{}&{}&{[n]}\\
{[k']}\arrow[r, two heads, "b'"]&{[s']}\arrow[ur, "f'"]
\end{tikzcd}
\]
in $\ccat^{(\act,\rightarrow)}_{/[n]}$ we have that the space of morphisms in $(\ccat^{(\act,\rightarrow)}_{/[n]})^\inrt_{/U}$ lying over it, viewed as a correspondence from $\ccat^\cell_{/[k']}$ to $\ccat^\cell_{/[k]}$, is isomorphic to $w^*_\cell$.
\end{lemma}
\begin{proof}
First, observe that to prove the first claim it suffices to prove the isomorphism \[\widehat{p}_{([k]\overset{b}{\twoheadrightarrow}[s]\xrightarrow{t}[n])}\cong p_{([k]\xrightarrow{t\circ b}[n])}\]
since the second isomorphism would then follow from \cref{prop:forget}. Observe that objects of $\widehat{p}_{([k]\overset{b}{\twoheadrightarrow}[s]\xrightarrow{t}[n])}$ are given by diagrams of the form
\[
\begin{tikzcd}[column sep=huge]
\bm\arrow[r, "s_f"]&{[\widetilde{m}]}\arrow[r, two heads, "a"]\arrow[dd, tail, "i"]&{[l]}\arrow[dd, tail, "j"]\arrow[dr, "f"]\\
{}&{}&{}&{[n]}\\
{}&{[k]}\arrow[r, two heads, "b"]&{[s]}\arrow[ur, "t"]
\end{tikzcd}.
\]
Observe that, since active and inert morphisms form a factorization system in $\Delta$,   the object $[l]$ and morphisms $j$ and $a$ are uniquely determined by $i:[\widetilde{m}]\rightarrowtail[k]$; in other words we have an isomorphism between the objects of $\widehat{p}_{([k]\overset{b}{\twoheadrightarrow}[s]\xrightarrow{t}[n])}$ and $p_{([k]\xrightarrow{t\circ b}[n])}$. A morphism in $\widehat{p}_{([k]\overset{b}{\twoheadrightarrow}[s]\xrightarrow{t}[n])}$ is given by a diagram of the form
\[
\begin{tikzcd}[column sep=huge]
\bm\arrow[dd, "h"]\arrow[r, "s_f"]&{[\widetilde{m}]}\arrow[dd, "U(h)"]\arrow[drr, tail, "i" near start, swap]\arrow[r, two heads, "a"]&{[l]}\arrow[dd, "g" swap]\arrow[drr, tail, "\overline{i}" swap]\arrow[drrr, "f"]\\
{}&{}&{}&{[k]}\arrow[r, two heads, "b"]&{[s]}\arrow[r, "t" very near start]&{[n]}\\
\langle p\rangle\arrow[r, "s_v" swap]&{[\widetilde{p}]}\arrow[r, two heads, "c" swap]\arrow[urr, tail, "j" near start]&{[q]}\arrow[urr,tail, "\overline{j}"]\arrow[urrr, "v" swap]
\end{tikzcd}.
\]
The only morphism in this diagram that is not obviously determined by $h$ is $g:[l]\rightarrow[q]$, so to provide the isomorphism between morphisms in $\widehat{p}_{([k]\overset{b}{\twoheadrightarrow}[s]\xrightarrow{t}[n])}$ and $p_{([k]\xrightarrow{t\circ b}[n])}$ it suffices to prove that it is uniquely defined by $h$. To do so, observe that since $j\circ U(h)\cong i$ we have that $U(h)$ is inert (otherwise it would contradict the uniqueness of the active/inert factorization). Similarly, since $\widetilde{j}\circ g\cong \widetilde{i}$ we see that $g$ is also an inert morphism, and so it follows that it must be the inert part of the active/inert decomposition of $c\circ U(h)$.\par
To prove the final claim simply observe that the space of morphisms in question is obviously isomorphic to the space of morphisms lying over 
\[
\begin{tikzcd}
{[k]}\arrow[rr, "w"]\arrow[rd, "t\circ b"]&{}&{[k']}\arrow[ld, "t'\circ b'"]\\
{}&{[n]}
\end{tikzcd}
\]
in $(\ccat_{/[n]})^\inrt_{/U}$ and this is isomorphic to $w^*_\cell$ by \cref{prop:forget}.
\end{proof}
\begin{lemma}\label{lem:lax_fun}
$i:(\ccat_{/[n]})^\inrt_{/U}\hookrightarrow(\ccat^{(\act,\rightarrow)}_{/[n]})^\inrt_{/U}$ identifies $\mor_\cat((\ccat_{/[n]})^\inrt_{/U},\cS)$ with the full subcategory of $\cF\in\mor_\cat((\ccat^{(\act,\rightarrow)}_{/[n]})^\inrt_{/U},\cS)$ such that for any object $([k]\overset{b}{\twoheadrightarrow}[s]\xrightarrow{t}[n])$ we have that
\[\cF|_{\widehat{p}_{([k]\overset{b}{\twoheadrightarrow}[s]\xrightarrow{t}[n])}}\cong b^*_\cell \cF|_{\widehat{p}_{i([s]=\joinrel=[s]\xrightarrow{t}[n])}}.\]
\end{lemma}
\begin{proof}
Since $i$ is fully faithful it induces an isomorphism between $\mor_\cat((\ccat_{/[n]})^\inrt_{/U},\cS)$ and the full subcategory of $\cF\in\mor_\cat((\ccat^{(\act,\rightarrow)}_{/[n]})^\inrt_{/U},\cS)$ for which $\cF\cong i_!i^*\cF$, so to prove the claim it suffices to calculate $i_!$. Fix an object $x\bydef([k]\overset{b}{\twoheadrightarrow}[s]\xrightarrow{t}[n],[\widetilde{m}]\rightarrowtail[k])$, then by definition
\[i_!i^*\cF(x)\cong \underset{(i(y)\rightarrow x)\in(i/(\ccat^{(\act,\rightarrow)}_{/[n]})^\inrt_{/U})}{\colim}\cF(y).\]
Observe that any morphism to $([k]\overset{b}{\twoheadrightarrow}[s]\xrightarrow{t}[n])$ from the object in the image of $\overline{i}$ has the form 
\[
\begin{tikzcd}[column sep=huge]
{[k]}\arrow[r, two heads, "b"]\arrow[dd, "h\circ b"]&{[s]}\arrow[dd, "h"]\arrow[dr, "t"]\\
{}&{}&{[n]}\\
{[q]}\arrow[r, equal]&{[q]}\arrow[ur, "g"]
\end{tikzcd},
\]
which we will denote by $\widehat{h}$. We will denote by \[(i/(\ccat^{(\act,\rightarrow)}_{/[n]})^{\inrt}_{/U})^\id\hookrightarrow i/(\ccat^{(\act,\rightarrow)}_{/[n]})^\inrt_{/U}\]
the full subcategory on those morphisms $(i(y)\rightarrow x)$ for which the underlying morphism in $\ccat^{\act,\rightarrow}_{/[n]}$ has $h\cong\id$ in the notation of the diagram above. We will prove that this subcategory is cofinal. To do this, first observe that we can uniquely decompose the morphism from $i([q]\xrightarrow{g}[n])$ to $([k]\overset{b}{\twoheadrightarrow}[s]\xrightarrow{t}[n])$ as
\[
\begin{tikzcd}[column sep=huge]
{[k]}\arrow[r, two heads, "b"]\arrow[d, two heads, "b"]&{[s]}\arrow[d, equal]\arrow[dr, "t"]\\
{[s]}\arrow[r, equal]\arrow[d, "h"]&{[s]}\arrow[d, "h"]\arrow[r, "t"]&{[n]}\\
{[q]}\arrow[r, equal]&{[q]}\arrow[ur, "g"]
\end{tikzcd}.
\]
Observe that the bottom morphism lies in the image of $\overline{i}$, so this gives a decomposition $\widehat{h}\cong \widehat{b}\circ i(h)$ in $\ccat^{(\act,\rightarrow)}_{/[n]}$. Moreover, by \cref{lem:forget_2} the correspondence of morphisms in $(i/(\ccat^{(\act,\rightarrow)}_{/[n]})^\inrt_{/U}$ lying over $\widehat{h}$ is isomorphic to $(h\circ b)^*_\cell$, and this in turn is isomorphic to $h^*_\cell\circ b^*_\cell$. This implies that for a given $(i(y)\xrightarrow{f}x)$ the category of decompositions of the form $(i(y)\xrightarrow{i(\widetilde{f})}z\xrightarrow{a}x)$ where $a$ belongs to $(i/(\ccat^{(\act,\rightarrow)}_{/[n]})^{\inrt})^\id$ is contractible, proving that this category is indeed cofinal.\par
Finally, if we represent $x$ by a cellular morphism $(\bm\xrightarrow{c_m}[k])$, then it is easy to see that $(i/(\ccat^{(\act,\rightarrow)}_{/[n]})^{\inrt,\id}$ is a category of commutative diagrams
\[
\begin{tikzcd}[row sep=huge, column sep=huge]
\bm\arrow[r, "c_m"]\arrow[d, "h"]&{[k]}\arrow[d, "b"]\\
\bl\arrow[r, "c_l"]&{[s]}
\end{tikzcd},
\]
where $c_l$ is a cellular morphism representing an object in $\widehat{p}_{[s]=\joinrel=[s]\xrightarrow{t}[n]}\cong\ccat^\cell_{/[s]}$, and so the claim follows by \cref{rem:corr_corr}.
\end{proof}
\begin{remark}\label{rem:seg_con}
We can view $\seg_{\cat^\wrr_{/[n]}}(\cS)$ as a full subcategory of $\mor_\cat((\ccat_{/[n]}^{(\act,\rightarrow)})^\inrt_{/U},\cS)$ through a composition of full embeddings 
\[\seg_{\cat^\wrr_{/[n]}}(\cS)\hookrightarrow\mor_\cat(\ccat^{\wrr}_{/[n]},\cS)\overset{j_!}{\hookrightarrow}\mor_\cat((\ccat_{/[n]})^\inrt_{/U},\cS)\overset{i_!}{\hookrightarrow}\mor_\cat((\ccat^{\act,\rightarrow}_{/[n]})^\inrt_{/U},\cS)\] 
where $j$ denotes the inclusion of \cref{rem:lax_fun}. Observe that in this case the restriction of $\cF\in\seg_{\cat^\wrr_{/[n]}}(\cS)$ to any $\widehat{p}_{([k]\overset{b}{\twoheadrightarrow}[s]\xrightarrow{t}[n])}\cong\ccat^\cell_{/[k]}$ is an object of $\corr_k$. Indeed, by \cref{lem:lax_fun} we have \[\cF|_{\widehat{p}_{([k]\overset{b}{\twoheadrightarrow}[s]\xrightarrow{t}[n])}}\cong b^*_\cell \cF|_{\widehat{p}_{i([s]=\joinrel=[s]\xrightarrow{t}[n])}},\]
and so, since $b^*_\cell$ takes $\seg_{\ccat^\cell_{/[s]}}(\cS)$ to $\seg_{\ccat^\cell_{/[k]}}(\cS)$, we are reduced to proving the claim for the restriction of $\cF$ to $(\ccat_{/[n]})^\inrt_{/U}$. Similarly, we can use \cref{rem:lax_fun} to reduce further to proving the statement for $\ccat^\wrr_{/[n]}$, in which case it follows from the Segal condition for $\cF$ described in \cref{rem:seg_cat}.
\end{remark}
\begin{theorem}\label{thm:lax_corr}
For a given Segal space $\cC$ denote by $\mor^{\lax,!}_\cat(\cC,\corr)$ the flagged subcategory of (the underlying Segal space of) $\mor^\lax_\cat(\cC,\corr)$ having the same objects, but only those morphisms $L^\lax\cC\rightarrow\arr^\lax(\corr)$ that factor through $\arr^{\lax,!}(\corr)$ in the notation of \cref{cor:lax_adj}. Then there is an equivalence of categories between $\mor^{\lax,!}_\cat(\cC,\corr)$ and $\cat^\wrr_{/\cC}$.
\end{theorem}
\begin{proof}
First, observe that by \cref{prop:colim_lax} and \cref{prop:wrr_colim} it suffices to prove the theorem for $\cC\cong[n]$. We will first provide an isomorphism between the underling groupoids of both categories. Observe that the category $(L^\lax[n])_m$ is by definition isomorphic to the  subcategory of $\ccat^{\act,\rightarrow}_{/[n]}$ containing objects of the form $([m]\overset{a}{\twoheadrightarrow}[l]\xrightarrow{f}[n])$ and morphisms given by diagrams
\[
\begin{tikzcd}
{}&{[l']}\arrow[dd, "v"]\arrow[dr, "f'"]\\
{[m]}\arrow[ur, two heads, "a'"]\arrow[dr, two heads,  "a"]&{}&{[n]}\\
{}&{[l]}\arrow[ur, "f"]
\end{tikzcd}.
\]
Denote by $p_1:\ccat_{/[n]}^{\act,\rightarrow}\rightarrow\Delta^\op$ the morphism sending $[k]\overset{b}{\twoheadrightarrow}[s]\xrightarrow{g}[n]$ to $[k]$, then this category is isomorphic to the fiber $p_1^{-1}([m])$. Given an active morphism $[k]\overset{b}{\twoheadrightarrow}[m]$, the induced morphism $b:(L^\lax[n])_m\rightarrow(L^\lax[n])_k$ sends $([m]\overset{a}{\twoheadrightarrow}[l]\xrightarrow{f}[n])$ to $([k]\overset{a\circ b}{\twoheadrightarrow}[l]\xrightarrow{f}[n])$. Similarly, for an inert morphism $i:[s]\rightarrowtail[m]$ the induced morphism sends it to $([m]\overset{\widetilde{a}}{\twoheadrightarrow}[l]\xrightarrow{f\circ \widetilde{i}}[n])$, where $a\circ i\cong \widetilde{i}\circ \widetilde{a}$ is the active/inert decomposition in $\Delta$. Generally we will denote by $f^*:p_1^{-1}([m])\rightarrow p_1^{-1}([l])$ the induced functor for a morphism $f:[l]\rightarrow[m]$ in $\Delta$.\par
Now assume that we are given a Segal $\ccat^\wrr_{/[n]}$-space $\cF$, we will associate to it a functor $L^\lax[n]\rightarrow\corr$. First, observe that using  \cref{lem:lax_fun} and \cref{rem:lax_fun} we can view $\cF$ as a presheaf on $((\ccat^{(\act,\rightarrow)}_{/[n]})^\inrt_{/U})^\op$. Next, observe that a functor $(L^\lax[n])_m\rightarrow\corr_m\cong\seg_{\ccat^\cell_{/[m]}}(\cS)$ is an element of  
\begin{equation}\label{eq:four}
    \mor_\cat((L^\lax[n])_m,\mor_\cat(\ccat^\cell_{/[m]},\cS))\cong\mor_\cat((L^\lax[n])_m\times \ccat^\cell_{/[m]},\cS).
\end{equation}
Denote by $\overline{p}$ the composition $p_1\circ\widehat{p}:(\ccat_{/[n]}^{\act,\rightarrow})^\inrt_{/U}\rightarrow\Delta^\op$, then it follows from \cref{lem:forget_2} that 
\[(L^\lax[n])_m\times \ccat^\cell_{/[m]})\cong (p_1^{-1}([m])\times \ccat^\cell_{/[m]})\cong \overline{p}^{-1}([m]).\]
For each $m$ we then define the required morphism $(L^\lax[n])_m\rightarrow\corr_m$ to be given by the restriction of $\cF$ to $\overline{p}^{-1}([m])$ (the fact that this functor lands in $\corr_m$ follows from \cref{rem:seg_con}). We then send every morphism $f:[l]\rightarrow[m]$ in $\Delta$ to $f^*_\cell:\corr_m\rightarrow\corr_l$. \par
For this definition to be consistent we need to prove that for every morphism $f$ as above the following diagram commutes
\[
\begin{tikzcd}[row sep=huge, column sep=huge]
p_1^{-1}({[m]})\arrow[r, "\cF_m"]\arrow[d, "f^*"]&\corr_m\arrow[d, "f^*_\cell"]\\
p_1^{-1}({[l]})\arrow[r, "\cF_l"]&\corr_l
\end{tikzcd}.
\]
Now fix an object $([k]\overset{b}{\twoheadrightarrow}[s]\xrightarrow{g}[n])\in\ccat^{\act,\rightarrow}_{/[n]}$. Assume we are given an active morphism $a:[l]\twoheadrightarrow[k]$ in $\Delta$, then we need to prove that 
\[\cF|_{\widehat{p}_{([l]\overset{b\circ a}{\twoheadrightarrow}[s]\xrightarrow{g}[n])}}\cong a^*_\cell\cF|_{\widehat{p}_{({[f]\overset{a}{\twoheadrightarrow}[s]\xrightarrow{g}[n]})}},\]
however this follows from 
\[\cF|_{\widehat{p}_{([l]\overset{b\circ a}{\twoheadrightarrow}[s]\xrightarrow{g}[n])}}\cong b^*_\cell( a^*_\cell\cF|_{\widehat{p}_{([s]=\joinrel=[s]\xrightarrow{g}[n])})}\cong a^*_\cell\cF|_{\widehat{p}_{([f]\overset{a}{\twoheadrightarrow}[s]\xrightarrow{g}[n])}}.\]
Now assume that we are given an inert morphism $i:[t]\rightarrowtail[k]$, consider the following commutative diagram
\[
\begin{tikzcd}[row sep=huge, column sep=huge]
{[k]}\arrow[r, two heads, "b"]&{[s]}\arrow[r, "g"]&{[n]}\\
{[l]}\arrow[u, tail, "i"]\arrow[r, two heads, "a"]&{[t]}\arrow[u, tail, "j"]
\end{tikzcd}
\]
obtained from the active/inert factorization in $\Delta$. We need to prove that \[\cF|_{\widehat{p}_{([l]\overset{a}{\twoheadrightarrow}[t]\xrightarrow{g\circ j}[n])}}\cong i^*_\cell\cF|_{\widehat{p}_{([k]\overset{b}{\twoheadrightarrow}[s]\xrightarrow{g}[n])}},\] 
which follows from \cref{rem:lax_fun}, \cref{lem:lax_fun} and the following string of isomorphisms
\[\cF|_{\widehat{p}_{([l]\overset{a}{\twoheadrightarrow}[t]\xrightarrow{g\circ j}[n])}}\cong a^*_\cell(j^*_\cell\cF|_{\widehat{p}_{([s]=\joinrel=[s]\xrightarrow{g\circ j}[n])}})\cong i^*_\cell(b^*_\cell\cF|_{\widehat{p}_{([s]=\joinrel=[s]\xrightarrow{g\circ j}[n])}})\cong i^*_\cell\cF|_{\widehat{p}_{([k]\overset{b}{\twoheadrightarrow}[s]\xrightarrow{g}[n])}}.\]
This defines a morphism \[H:(\cat^\wrr_{/[n]})^\sim\rightarrow\mor_\cat(L^\lax[n],\corr)^\sim.\]\par
Now assume that we are given a functor $F:L^\lax[n]\rightarrow\corr$, we will construct from it an object of $\ccat^\wrr_{/[n]}$ which we identify with a functor $\cF:(\ccat^{(\act,\rightarrow)}_{/[n]})^\inrt_{/U}\rightarrow\cS$. First we define it on $\coprod_{m\geq0}\overline{p}^{-1}([m])$ using \cref{eq:four} above. Observe that, since $F$ was assumed to be a functor to $\corr$, $\cF$ satisfies the conditions of both \cref{lem:lax_fun} and \cref{rem:lax_fun}, so the only thing we need to do is to extend it to $(\ccat^{(\act,\rightarrow)}_{/[n]})^\inrt_{/U}$.  For a morphism $f:[l]\rightarrow[k]$ in $\Delta$ we can view $\overline{p}^{-1}(f)$ as a correspondence from $\overline{p}^{-1}([k])$ to $\overline{p}^{-1}([l])$ which we will denote by $f^{-1}$. It follows from our previous considerations that $f^{-1}$ is isomorphic to 
\[(f^*_\cell,f^*):\corr_k\times p_1^{-1}([k])\rightarrow\corr_l\times p_1^{-1}([l]).\]
In particular, it follows that $(g\circ f)^{-1}\cF\cong f^{-1}\circ g^{-1}\cF$ for any $\cF\in\cat^\wrr_{/[n]}$. Fix an object $(x,y)\in\corr_l\times p_1^{-1}([l])$, then we see that
\[f^{-1}\cF(x,y)\cong \underset{(f^*t\rightarrow y)\in f^*/y}{\colim}f^*_\cell\cF(x,t)\cong \underset{(f^*t\rightarrow y)\in f^*/y}{\colim}\cF(x,f^*t),\]
where the second isomorphism follows from the fact that $F$ was assumed to be a functor to $\corr$.\par
Observe that we have
\[(\ccat^{(\act,\rightarrow)}_{/[n]})^\inrt_{/U}\cong\underset{g:[m]\rightarrow\Delta^\op}{\colim}\overline{p}^{-1}(g),\]
so it suffices to extend $\cF$ to various $\overline{p}^{-1}(g)$ in a compatible way. Recall from \cref{prop:nat} that in order to extend $\cF$ to $\overline{p}^{-1}(g)$ we need to define a section $\ccat^\inj_{/[m]}\rightarrow\corr^\cF_m$.  Observe that, since $s$ is already defined on morphisms in $\ccat^{\infty,\inj}_{/[m]}$, it suffices to define it for inert morphism that preserve the minimal element. First, assume that we are given a morphism $[1]\xrightarrow{[s,t]}[m]$ whose image in $\Delta$ is given by $[a]\xrightarrow{g} [b]$ and let $i:[0]\rightarrowtail[1]$ be the inert morphism given by the inclusion of $\{0\}$. To this data we need to associate a morphism $\alpha_g:g^{-1}\cF_{[b]}\rightarrow\cF_{[a]}$, where $\cF_{[a]}$ and $\cF_{[b]}$ denote the restrictions of $\cF$ to the fibers over $[a]$ and $[b]$. We define it as
\[g^{-1}\cF_{[b]}(x,y)\cong\underset{(g^*t\rightarrow y)\in g^*/y}{\colim}\cF_{[a]}(x,g^*t)\rightarrow \underset{(w\rightarrow y)\in p_1^{-1}([a])/y}{\colim}\cF_{[a]}(x,w)\cong\id_!\cF_{[a]}(x,y)\cong\cF_{[a]}(x,y),\]
where the morphism of colimits is induced by the natural inclusion between the underlying diagrams.\par
Now assume that we are given an inert morphism morphism
\[
\begin{tikzcd}
{[l']}\arrow[rr, tail, "i"]\arrow[dr, "h'"]&{}&{[l]}\arrow[dl, "h" swap]\\
{}&{[m]}
\end{tikzcd}.
\]
Denote by $(x_0\xrightarrow{g_1}x_1\xrightarrow{g_2}...\xrightarrow{g_l}x_l)$ the image of $[l]$ in $(\ccat^{(\act,\rightarrow)}_{/[n]})^\inrt_{/U}$, then we define the required morphism as
\begin{align*}
    g_1^{-1}g_2^{-1}...g_l^{-1}\cF_l\cong& g_1^{-1}...g_{l'}^{-1}(g_l\circ...\circ g_{l'+1})^{-1}\cF_l\xrightarrow{\alpha_{g_l\circ...\circ g_{l'+1}}}g_1^{-1}...g_{l'}^{-1}\cF_{l'}.
\end{align*}
To prove that this indeed defines a section, we need only to show that for any composable pair of morphisms $(x\xrightarrow{f}y\xrightarrow{g}z)$ we have $\alpha_f\circ f^{-1}\alpha_g\cong\alpha_{g\circ f}$. In other words, we need to prove the following isomorphism
\[\underset{(g^*w\rightarrow t)\in g^*/t}{\colim}\underset{(f^*t\rightarrow y)\in f^*/y}{\colim}\cF(x,f^*g^*w)\cong\underset{(f^*g^*w\rightarrow y)\in f^*g^*/y}{\colim}\cF(x,f^*g^*w),\]
which follows from \cref{prop:fun_comp}. It follows from the construction that this system of correspondences on various $\overline{p}^{-1}(g)$ is compatible, i.e. it defines an element of 
\[\underset{g:[m]\rightarrow\Delta^\op}{\lim}\mor_\corr(\overline{p}^{-1}(g),*)\cong\mor_\corr((\ccat^{(\act,\rightarrow)}_{/[n]})^\inrt_{/U},*).\]
This defines a morphism \[G:\mor_\cat(L^\lax[n],\corr)^\sim\rightarrow(\cat^\wrr_{/[n]})^\sim.\]\par
Observe that $H\circ G (F)\cong F$ for any $F\in\mor_\cat(L^\lax[n],\corr)$. Indeed, it follows easily from the construction that the restrictions of $F$ and $H\circ G(F)$ to every $(L^\lax[n])_k$ are isomorphic, which means that $F$ and $H\circ G (F)$ agree on objects and morphisms of $L^\lax[n]$ meaning that they are isomorphic as functors. To prove that $G\circ H\cong\id$, first observe that $\cF$ and $G\circ H(\cF)$ by definition agree on $\coprod_{m\geq0}\overline{p}^{-1}([m])$, so it suffices to prove that they agree on $\overline{p}^{-1}(f)$ for all morphisms $f:[m]\rightarrow[l]$ in $\Delta$. To prove this observe that for a category $\cC\rightarrow\cD$ with a presheaf $\cE$ the restriction of $\cE$ to $\cC_f$ for a morphism $f$ in $\cD$ is given by 
\[\underset{(y\rightarrow x)\in\cC_f}{\colim}\cE(y)\rightarrow\underset{(y\rightarrow x)\in\cC_{/x}}{\colim}\cE(y)\cong \cE(x).\]
Applying it to our case, we see that the required morphism 
\[f^{-1}G\circ H(\cF)_{[l]}\rightarrow G\circ H(\cF)_{[m]}\]
is given by
\begin{align*}
    f^{-1}G\circ H(\cF)(x,y)\cong& \underset{f^*t\rightarrow y}{\colim}f^*_\cell G\circ H(\cF)(x,t)\\
    \cong& \underset{f^*t\rightarrow y}{\colim}f^*_\cell\cF(x,t)\\
    \cong&\underset{f^*t\rightarrow y}{\colim}\cF(x,f^*t) \rightarrow\\
    \rightarrow&\underset{w\rightarrow y}{\colim}\cF(x,w)\cong\cF(x,y),
\end{align*}
where the second isomorphisms follows from the observation that $\cF$ and $G\circ H(\cF)$ coincide on $(L^\lax[n])_l$. This is exactly the same as the corresponding morphism for $\cF$, which proves the claim.\par
To prove the isomorphism on morphisms, observe that it follows from the definition of $\mor^{\lax,!}_\cat([n],\corr)$ and \cref{cor:lax_adj} that a morphism between two functors $\cF,\cG:[n]\rightsquigarrow\corr$ is given by the data of a natural transformation
\[\alpha_i:(L^\lax[n])\rightarrow\mor_\cat([1],\corr_i)\]
between functors $\Delta^\op\rightarrow\cat$ for which the compositions $s\circ \alpha_i$ (resp $t\circ \alpha_i$)  with the source morphism $s:\mor_\cat([1],\corr_i)\rightarrow\corr_i$ (resp. the target morphism $t:\mor_\cat([1],\corr_i)\rightarrow\corr_i$) are isomorphic to the lax functor $\cF$ (resp. $\cG$). The construction considered above then easily identifies this data with a morphism between the corresponding objects $G(\cF)$ and $G(\cG)$ of $\cat^\wrr_{/[n]}$.
\end{proof}
\begin{cor}\label{cor:monad}
Given a Segal space $\cC$ there is an equivalence between the category of monads on $\cC$ in $\corr$ and the category $\cat^{\bo}_{\cC/}$ of morphisms $f:\cC\rightarrow\cD$ that induce an isomorphism on the groupoid of objects. \par
\end{cor}
\begin{proof}
This follows from \cref{thm:lax_corr} and \cref{cor:lax_mon}.
\end{proof}
\begin{cor}\label{cor:lax_span}
For a Segal space $\cC$ there is an equivalence
\[\mor_\cat^{\lax,!}(\cC,\spanc)\cong\cat_{/\cC}.\]
\end{cor}
\begin{proof}
Recall from \cref{prop:span_subcat} that $\spanc$ can be viewed as a full subcategory of $\corr$ on constant Segal spaces. The isomorphism of \cref{thm:lax_corr} then identifies $\mor_\cat^{\lax,!}(\cC,\spanc)$ with the subcategory of $\cat^\wrr_{/\cC}$ on the objects of the form $(X\xrightarrow{f}\cD)$ where $X$ is a constant simplicial space. Since $f$ is supposed to be an isomorphism on objects, it follows that $X\cong\cD_0$, and so it follows that this subcategory is isomorphic to $\cat_{/\cC}$.
\end{proof}
\begin{cor}
For a Segal space $\cC$ there is an equivalence
\[\mor_\cat^{\lax,\un,!}(\cC,\corr)\cong\cat_{/\cC}.\]
\end{cor}
\begin{proof}
First of all, observe that by \cref{prop:colim_lax} and \cref{cor:slice} it suffices to prove the claim for $\cC\cong[n]$. Now observe that the unital lax functors $[n]\rightsquigarrow\corr$ can be identified with those morphisms $L^\lax[n]\rightarrow\corr$ that send objects of the form $([k]\overset{a}{\twoheadrightarrow}[l]\xrightarrow{\{i\}}[n])$ where $k>0$ and $\{i\}$ denotes the constant morphism with image $i\in[n]$ to strings of identity correspondences. Untangling the constructions of \cref{thm:lax_corr} and preceding propositions, we see that those functors correspond to $\cF\in\cat^\wrr_{/[n]}$ for which $\cF(\bm\xrightarrow{f}[n])\cong\cF([m]\xrightarrow{\widetilde{f}}[n])$, where $\widetilde{f}$ has the same underlying object in $\ccat_{/[n]}$ except all the edges are marked with $0$. Observe that there is a natural inclusion $i:\ccat_{/[n]}\hookrightarrow\ccat^\wrr_{/[n]}$ sending each $[l]\xrightarrow{g}[n]$ to the corresponding object of $\ccat^\wrr_{/[n]}$ in which all edges are marked with $0$. This morphism induces a fully faithful inclusion $i_!:\cat_{/[n]}\hookrightarrow\cat^\wrr_{/[n]}$, and objects $\cF$ as above are easily seen to be precisely the objects in the image of $i_!$, which concludes the proof.
\end{proof}
\begin{notation}\label{not:comp}
For a given Segal space $\cC$ we denote by $\cat^{\lcomp,\wrr}_{/\cC}$ the full subcategory of $\cat^\wrr_{/\cC}$ on the morphisms over $\cC$ of the form $(\cE\rightarrow\cD)$ where $\cE$ is complete and by $\cat^{\comp,\wrr}_{/\cC}$ the full subcategory on the objects as above for which both $\cE$ and $\cD$ are complete.
\end{notation}
\begin{prop}
The natural inclusion $\cat^{\lcomp,\wrr}_{/\cC}\hookrightarrow\cat^\wrr_{/\cC}$ admits a left adjoint.
\end{prop}
\begin{proof}
Indeed, the left adjoint in question is easily seen to be the functor sending $\cC\rightsquigarrow\corr$ to $\cC\rightsquigarrow\corr\xrightarrow{\widehat{(-)}}\corr$, where $\widehat{(-)}$ is the functor of \cref{cor:comp_corr}.
\end{proof}
\begin{remark}
It follows from \cref{rem:bicat} that, given a monad $T$ in $\cat$ over a Segal space $\cC$, we have an induced monad $T_!$ in $\corr$. Applying to it the construction of \cref{thm:lax_corr} we obtain a "naive" version of the Kleisli category $\cK(T)$. More specifically, it is a Segal space with the same space of objects as $\cC$ whose morphisms are given by \[\mor_{\cK(T)}(x,y)\bydef\mor_\cC(x,Ty)\]
with composition given by $f\xrightarrow{f}Ty\xrightarrow{Tg}T^2z\xrightarrow{m}Tz$. The problem with this construction is that, even if $\cC$ is a complete Segal space, $\cK(T)$ generally is not.
\end{remark}
\begin{prop}\label{prop:monad_v2}
Given a complete Segal space $\cC$ there is an equivalence between $\cat^{\lcomp,\wrr}_{/\cC}$ and the category $\cat^{\overset{\eso}{\rightarrow}}_{/\cC}$ whose objects are morphisms $\cD\xrightarrow{f}\cE$ over $\cC$ that induce an effective epimorphism on the space of objects and such that for every non-unital morphism $f:x\rightarrow y$ in $\cC$ we have an isomorphism
\[(\cD_f)^\sim\cong \cD_{0,x}\times_{\cE_{0,x}}(\cE_f)^\sim\times_{\cE_{0,y}}\cD_{0,y}.\]
\end{prop}
\begin{proof}
Given an object $\cD\xrightarrow{g}\cK$ of $\cat^{\lcomp,\wrr}_{/\cC}$ observe that since $\cC$ is complete, the morphism $\cK\rightarrow\cC$ factors through $\widehat{\cK}$. Observe that the composite $\cD\rightarrow\cK\rightarrow\widehat{\cK}$ is essentially surjective on objects and moreover we have 
\[\cD_{0,x}\times_{\widehat{\cK}_{0,x}}\widehat{\cK}^\sim_f\times_{\widehat{\cK}_{0,y}}\cD_{0,y}\cong \cD_{0,x}\times_{\cK_{0,x}}\cK^\sim_f\times_{\cK_{0,y}}\cD_{0,y}\cong\cD^\sim_f,\]
where we have used the isomorphism \[\cK_{0,x}\times_{\widehat{\cK}_{0,x}}\widehat{\cK}^\sim_f\times_{\widehat{\cK}_{0,x}}\cK_{0,x}\cong\cK^\sim_f\]
of \cref{prop:compl} and the isomorphisms $\cD_{0,x}\cong\cK_{0,x}$ and $\cD_{0,y}\cong\cK_{0,y}$ which follow from the definition of $\cat^\wrr_{/\cC}$. It follows that $\cD\rightarrow\widehat{\cK}$ is an object of $\cat^{\overset{\eso}{\rightarrow}}_{/\cC}$, and so we have defined a functor $F:\cat^{\lcomp,\wrr}_{/\cC}\rightarrow\cat^{\lcomp,\overset{\eso}{\rightarrow}}_{/\cC}$. Conversely, given an object $\cD\xrightarrow{f}\cE$ of $\cat^{\overset{\eso}{\rightarrow}}_{/\cC}$ we will define a functor $\ccat^\wrr_{/\cC}\rightarrow\cS$ by sending every $0$-marked edge $a:[1]\rightarrow\cC$ to $\cD^\sim_f$ and every $1$-marked edge $b:\langle1\rangle\rightarrow\cC$ to $\cD_{0,s(b)}\times_{\cE_{0,s(b)}}\cE^\sim_b\times_{\cE_{0,t(b)}}\cD_{0,t(b)}$ with the obvious operation of composition. This defines a functor $G:\cat^{\lcomp,\overset{\eso}{\rightarrow}}_{/\cC}\rightarrow\cat^{\lcomp,\wrr}_{/\cC}$.\par
The fact that $G\circ F\cong\id$ once again follows from the isomorphism \[\cK_{0,x}\times_{\widehat{\cK}_{0,x}}\widehat{\cK}^\sim_f\times_{\widehat{\cK}_{0,x}}\cK_{0,x}\cong\cK^\sim_f\]
of \cref{prop:compl}. It remains to prove that $F\circ G\cong\id$. So assume we are given an object $\cD\xrightarrow{f}\cE$ of $\cat^{\overset{\eso}{\rightarrow}}_{/\cC}$, denote by $\cD\xrightarrow{\overline{f}}\overline{\cE}$ the objects $G(\cD\xrightarrow{f}\cE)$. First, observe that the space of morphisms of $\overline{\cE}$ is isomorphic to $\cD_0\times_{\cE_0}\cE_1\times_{\cE_0}\cD_0$. In particular, since $\cE$ is complete we see that the subcategory of invertible morphisms $(\overline{\cE})^\sim$ has $\cD_0\times_{\cE_0}\cD_0$ as the space of morphisms. By examining the constructions of \cite{ayala2018flagged}, we see that 
\[\widehat{\overline{\cE}}_0\cong\underset{[n]\in\Delta^\op}{\colim}(\overline{\cE})^\sim_n\cong \underset{[n]\in\Delta^\op}{\colim}\overbrace{\cD_0\times_{\cE_0}...\times_{\cE_0}\cD_0}^\text{n times}\cong \cE_0,\]
where the last isomorphism follows since $\cD_0\rightarrow\cE_0$ is an epimorphism. Similarly we see that for a morphism $b:x\rightarrow y$ in $\cC$ we have
\begin{align*}
        \widehat{\overline{\cE}}_b\cong&\underset{([n],[m])\in\Delta^\op\times\Delta^\op}{\colim}(\overline{\cE})^\sim_n\times_{\cD_0}\overline{\cE}_b\times_{\cD_0} (\overline{\cE})^\sim_m\\
        \cong&\underset{([n],[m])\in\Delta^\op\times\Delta^\op}{\colim}(\overline{\cE})^\sim_n\times_{\cD_0}\cD_0\times_{\cE_0}\cE_b\times_{\cE_0}\cD_0\times_{\cD_0}(\overline{\cE})^\sim_m\\
        \cong&\underset{([n],[m])\in\Delta^\op\times\Delta^\op}{\colim}(\overline{\cE})^\sim_n\times_{\cE_0}\cE_b\times_{\cE_0}(\overline{\cE})^\sim_m\\
        \cong& \cE_0\times_{\cE_0}\cE_b\times_{\cE_0}\cE_0\cong\cE_b
\end{align*}
which concludes the proof that $F\circ G\cong \id$.
\end{proof}
\begin{theorem}\label{thm:lax_colim}
Given an object $(\cD\xrightarrow{f}\cE)$ of $\cat^{\wrr}_{/\cC}$, the category $\cE$ can be characterized as the lax colimit of the corresponding lax functor $F:\cC\rightsquigarrow\corr$.
\end{theorem}
\begin{proof}
First, observe that since 
\[\mor_\corr(\cE,\cK)\cong\mor_\corr(\cE\times\cK^\op,*)\]
by \cref{cor:dual}, it suffices to prove that
\[\coc_\corr^\lax(F,*)\cong\mor_\corr(\cE,*).\]
For $g:[n]\rightarrow\cC$ denote by $\cE_g$ the pullback $\cE\times_\cC [n]$, then we have
\[\cE\cong\underset{g:[n]\rightarrow\cC}{\colim}\cE_g.\]
Combining this with \cref{prop:colim_cocone}, we see that it suffices to prove the proposition for $\cC\cong[n]$.\par
We first prove the claim under the assumption that the functor $F:[n]\rightsquigarrow\corr$ actually lands in $\spanc$. Then it follows from the definition of the lax cocone together with \cref{prop:span_subcat} and the fact that $*\in\spanc$ that $\coc_\corr^\lax(F,*)\cong\coc_\spanc^\lax(F,*)$. It now follows from \cref{cor:lax_span} that giving a lax cocone under $F$ is equivalent to giving an object $\cF\in\ccat_{/[n+1]}$ such that its restriction to $[0,n]\subset[n+1]$ is isomorphic to $\cE\in\ccat_{/[n]}$ and such that $\cF_{n+1}\cong *$.\par
Denote by $\widetilde{j}:\widetilde{\ccat}_{/[n+1]}\hookrightarrow\ccat_{/[n+1]}$ the subcategory having the same objects and such that for $h:[m]\rightarrow[n]$ and $g:[l]\rightarrow[n]$ we have \[\mor_{\widetilde{\ccat}_{/[n]}}(h,g)\cong \mor_{\ccat_{/[n]}}(h,g)\]
if $h(m)<n+1$ and if $h(m)=n+1$ denote by $k_g$ the object of $[l]$ for which $g(k_g)=j<n+1$ but $g(k_g+1)=n+1$ if such an object exists (and introduce similar notation for $h$), then $\mor_{\widetilde{\ccat}_{/[n]}}(h,g)$ is given by those morphisms $g\rightarrow h$ over $[n]$ in $\Delta$ for which $g(l)=n+1$ and $g(k_g)=h(k_f)$. Observe that the category of functors $\widetilde{\ccat}_{/[n+1]}\rightarrow\cS$ that satisfy the Segal condition and whose restriction to $\ccat_{/[n]}\subset\widetilde{\ccat}_{/[n+1]}$ is given by $\cE$ is isomorphic to $\mor_\corr(\coprod_{i=1}^n\cE_i,*)$. Indeed, it is easy to see from the definition of $\widetilde{\ccat}_{/[n+1]}$ that the functors $\widetilde{\cF}$ satisfying these conditions are determined by their values $\cF_i$ on $[1]\xrightarrow{[i,n+1]}[n+1]$ an those $\cF_i$ have no additional structure except for morphisms $(\cE_{i})_n\times_{(\cE_i)_0}\cF_i\rightarrow\cF_i$ that assemble to give them the structure of an object of $\mor_\corr(\cE_i,*)$. It will follow from the calculation below that $\widetilde{j}$ takes $\mor_\corr(\coprod_{i=1}^n\cE_i,*)$ to $\coc_\spanc^\lax(F,*)$ and that $\widetilde{j}^*\widetilde{j}_!\widetilde{\cF}\cong j^*j_!\widetilde{\cF}$ in the notation of \cref{prop:nat}. Since $\widetilde{j}$ is surjective on objects, it will then follow from \cite[Theorem 4.7.3.5.]{luriehigher} that $\coc_\spanc^\lax(F,*)$ is isomorphic to the category of algebras for the monad $\widetilde{j}^*\widetilde{j}_!\cong j^*j_!$, which is in turn isomorphic to $\mor_\corr(\cE,*)$ by the proof of \cref{prop:nat}.\par
Finally, observe that 
\[(\widetilde{j}^*\widetilde{j}_!\cF)_i\cong\underset{g\in\ccat_{i,/[n]}}{\colim}\cE_{i,g(1)}\times_{(\cE_{g(1)})_0}\cE_{g(1),g(2)}\times_{(\cE_{g(1)})_0}...\times_{(\cE_{g(m)})_0}\cF_{g(m)}\]
where $\ccat_{i,/[n]}$ is the category whose objects are morphisms $g:[m+1]\rightarrow[n+1]$ for which $g(m+1)=n+1$ and whose morphisms are morphisms $g\rightarrow f$ in $\widetilde{\ccat}_{/[n]}$. Denote by $\overline{\ccat}_{i,/[n]}$ the full subcategory of $\ccat_{i,/[n]}$ on those $g$ for which $g(1)=g(k_g)$, then $\overline{j}:\overline{\ccat}_{i,/[n]}\hookrightarrow\ccat_{i,/[n]}$ is a cofinal subcategory. Indeed, for every $g\in\ccat_{i/[n]}$ the category $g/\overline{j}$ contains the final object given by the natural morphism $\overline{g}\rightarrow g$ over $[n]$ in $\Delta$, where $\overline{g}$ is obtained from $g$ by deleting all the points $k>0$ of $[m+1]$ such that $g(k)<g(k_g)$. It follows that
\begin{align*}
    (\widetilde{j}^*\widetilde{j}_!\cF)_i\cong&\underset{g\in\overline{\ccat}_{i,/[n]}}{\colim}\cE_{i,g(1)}\times_{\cE_{g(1)}}\cE_{g(1),g(2)}\times_{\cE_{g(1)}}...\times_{\cE_{g(m)}}\cF_{g(m)}\\
    \cong&\underset{(g:[1]\rightarrow[n], g(0)=i)}{\colim}\underset{l\in\Delta^\op}{\colim}\;\cE_{i,g(1)}\times_{(\cE_{g(1)})_0}(\cE_{g(1)})_l\times_{(\cE_{g(1)})_0}\cF_{g(1)}\\
    \cong&\coprod_{g:[1]\rightarrow[n],g(0)=i}\cF_{g(1)}\circ \cE_{g(0),g(1)},
\end{align*}
where in the last line we view $\cE_{g(0),g(1)}$ as a correspondence and $\circ$ denotes the composition of correspondences. It now follows from the proof of \cref{prop:nat} that $\widetilde{j}^*\widetilde{j}_!\cF\cong j^*j_!\cF$ which concludes the proof in this special case.\par
Now assume that $(\cD\xrightarrow{f}\cE)$ is a general object of $\ccat^\wrr_{/[n]}$. Just like in the special case treated above, it follows from \cref{thm:lax_corr} that giving an object of $\coc^\lax_\corr(F,*)$ is equivalent to providing an object of $\ccat^\wrr_{/[n]}$ of the form $((\cD,*)\xrightarrow{(f,\id)}(\cE,*))$. We will now show that this object is uniquely determined by its restriction $(\cE,*)$ to $\ccat_{/[n]}$. Assume that we are given such an object, so in particular we are given the spaces $\cF_i$ lying over $[i,n+1]$ together with compatible system of actions $\cE_{i,j}\times_{(\cE_j)_0}\cF_j\rightarrow\cF_i$. We will prove that this uniquely determines the action of $\cD_{i,j}$ on $\cF_j$ for all $i$ and $j$, which would conclude the proof. Indeed, it follows from the definition of $\ccat^\wrr_{/[n]}$ that we have the following commutative diagram
\[
\begin{tikzcd}[row sep=huge, column sep=huge]
\cE_{i,j}\times_{(\cE_j)_0}\cF_j\arrow[r]&\cF_i\\
\cD_{i,j}\times_{(\cD_j)_0}\cF_j\arrow[u, "f"]\arrow[r]&\cF_i\arrow[u, "\sim"]
\end{tikzcd}
\]
in which the right vertical arrow is an isomorphism, which obviously proves our claim and concludes the proof.
\end{proof}
\begin{remark}
Since $\corr\cong\corr^\op$, the category $\cE$ is also a lax limit of the corresponding functor.
\end{remark}
\begin{lemma}\label{lem:fact_cat}
Fix a Segal space $\cC$ together with an object $(\cD\rightarrow\cE)$ of $\cat^\wrr_{/\cC}$ such that the image of all morphisms of $\cD$ is invertible in $\cE$. Then the left Kan extension $p_!:\mor_\cat((\ccat_{/\cC})^\inrt_{/U},\cS)\rightarrow\mor_\cat(\ccat_{/\cC},\cS)$ (where $U$ is the functor of \cref{not:forget}) takes $(\cD\rightarrow\cE)$ to an object of $\cat_{/\cC}$, where we view $\cat^\wrr_{/\cC}$ as a full subcategory of $\mor_\cat((\ccat_{/\cC})^\inrt_{/U},\cS)$ using \cref{rem:lax_fun}.
\end{lemma}
\begin{proof}
First, observe that it is enough to prove the lemma for $\cC\cong[n]$. Indeed, it follows from \cref{prop:wrr_colim} and \cref{cor:slice} that a functor $\cF:\ccat^\wrr_{/\cC}\rightarrow\cS$ (resp. $\cE:\ccat_{/\cC}\rightarrow\cS$) satisfies the respective Segal condition if and only if this is true for all $v^*\cF$ (resp. $v^*\cE$) for $v\in\mor_\cat([n],\cC)$ and $n\geq0$. The claim now follows from the Beck Chevalley isomorphism for the pullback square
\[
\begin{tikzcd}[row sep=huge, column sep=huge]
(\ccat_{/[n]})^\inrt_{/U}\arrow[d, "v"]\arrow[r, "p"]&\ccat_{/[n]}\arrow[d, "v"]\\
(\ccat_{/\cC})^\inrt_{/U}\arrow[r, "p"]&\ccat_{/\cC}
\end{tikzcd}.
\]
The value of the left Kan extension on a given $w:[m]\rightarrow[n]$ is given by 
\[p_!\cF(w)\cong\underset{(p(x)\rightarrow w)\in(p/w)}{\colim}\cF(x).\]
Observe that there is a natural inclusion $i:p^{-1}(w)\hookrightarrow(p/w)$, we will prove that it is cofinal. To do this, first recall from \cref{prop:forget} that the fiber of $p$ over $w$ is given by $\ccat^\cell_{/[m]}$. It follows that the objects of $(p/w)$ are given by pairs $y\bydef(f:[m]\rightarrow[l],h:[k]\rightarrow[l])$ where $f$ is a morphism over $[n]$ and $h$ is a cellular morphism. The objects of $(y/i)$ are then given by pairs $(g:[k]\rightarrow[m], \alpha:f\circ g\cong h)$. Observe that we then have the following commutative diagram 
\[
\begin{tikzcd}[row sep=huge, column sep=huge]
(y/i)\arrow[r, "m"]\arrow[d, hook, "i"]&\{h\}\arrow[d, hook]\\
i^\cell_l/f/g\arrow[r, "m"]&i^\cell_l/f\circ g
\end{tikzcd}
\]
in the notation of \cref{prop:corr}. Now observe that since $h$ and $g$ are assumed to be cellular, the vertical maps are homotopy equivalences. The bottom horizontal map is cofinal by the proof of \cref{prop:corr}, so in particular it is a homotopy equivalence. It follows that the top horizontal map is also a homotopy equivalence, which means exactly that $(y/i)$ is contractible, so $i$ is cofinal.\par
Now observe that the natural inclusion $i:\ccat^\wrr_{/[n]}\hookrightarrow(\ccat_{/[n]})^\inrt_{/U}$ is also cofinal, so it suffices to take colimit over the fiber of $U$. Finally, observe that this fiber is isomorphic to $(\Delta^\op)^{m}$, in particular it follows that \[U_w\cong\underset{(w\rightarrowtail e)\in(\ccat_{/[n]})^\el_{w/}}{\lim}U_e.\] 
Moreover, the Segal condition for $\ccat^\wrr_{/[n]}$ implies that \[\cF(x_w)\cong\underset{(w\overset{i}{\rightarrowtail} e)\in(\ccat_{/[n]})^\el_{w/}}{\lim}\cF(i_!x),\]
where $x_w$ is an element of $U_w$ and $\{i_!x\}$ are its projections to various $U_e$ under the isomorphism above. The only thing that remains to be proved is the fact that $U_w$-colimits distribute over $(\ccat_{/[n]})^\el_{w/}$-limits. We first make the condition more explicit. Denote by $\cE^i_m$ (resp. by $\cD_m^i$) for $i\in\{0,1,...,n\}$ the value of this space the constant morphism $\bm\rightarrow[n]$ (resp. $[m]\rightarrow[n]$) with value $i$ and by $\cE_1^{i,j}$ the value on $\langle1\rangle\rightarrow[n]$ that sends $0$ to $i$ and $1$ to $j$. In these notations we need to prove the following isomorphism
\begin{tiny}
\begin{align}\label{eq:five}
\begin{split}
    &\underset{\{m_i\}_{i\in\{0,1,...,n\}}\in(\Delta^\op)^n}{\colim} (\cD^0_{m_0}\times_{\cD_0^0}\cE^{0,1}_1\times_{\cD_0^1}\cD^1_{m_1})\times_{\cD^1_{m_1}}...\times_{\cD^{n-1}_{m_{n-1}}}(\cD^{n-1}_{m_{n-1}}\times_{\cD_0^{n-1}}\cE^{n-1,n}_1\times_{\cD_0^{n}}\cD^n_{m_n})\\
    \cong&\underset{(m_0,m_1)\in(\Delta^\op)^2}{\colim}(\cD^0_{m_0}\times_{\cD_0^0}\cE^{0,1}_1\times_{\cD_0^1}\cD^1_{m_1})\times_{\underset{m_1\in\Delta^\op}{\colim}\cD^1_{m_1}}...\times_{\underset{m_{n-1}\in\Delta^\op}{\colim}\cD^{n-1}_{m_{n-1}}}\underset{(m_{n-1},m_n)\in(\Delta^\op)^2}{\colim}(\cD^{n-1}_{m_{n-1}}\times_{\cD_0^{n-1}}\cE^{n-1,n}_1\times_{\cD_0^{n}}\cD^n_{m_n}).
\end{split}    
\end{align}
\end{tiny}
We first assume that $n=2$. Denote by \[p_{0,1}:\cD^{0,1}_{m_{0,1}}\times_{\cD_0^{0,1}}\cE^{(0,1),(1,2)}_1\times_{\cD_0^{1,2}}\cD^{1,2}_{m_{1,2}}\rightarrow\cD^{1}_{m_{1}}\]
and by
\[p_2:\cD^0_{m_0}\times_{\cD_0^0}\cE^{0,1}_1\times_{\cD_0^1}\cD^1_{m_1}\times_{\cD_0^1}\cE^{1,2}_1\times_{\cD_0^2}\cD^2_{m_2}\rightarrow\cD^0_{m_0}\times_{\cD_0^0}\cE^{0,1}_1\times_{\cD_0^1}\cD^1_{m_1}\] 
the natural projections, where we view $p_{0,1}$ (resp. $p_2$) as morphisms of the corresponding right fibrations over $(\Delta^\op)^2$ and $\Delta^\op$ (resp. over $(\Delta^\op)^3$ and $(\Delta^\op)^2$). With these notations we need to prove that the following commutative diagram
\[
\begin{tikzcd}[row sep=huge, column sep=huge]
{|\cD^0_{m_0}\times_{\cD_0^0}\cE^{0,1}_1\times_{\cD_0^1}\cD^1_{m_1}\times_{\cD_0^1}\cE^{1,2}_1\times_{\cD_0^2}\cD^2_{m_2}|}\arrow[r]\arrow[d, "|p_2|"]&{|\cD^0_{m_0}\times_{\cD_0^0}\cE^{0,1}_1\times_{\cD_0^1}\cD^1_{m_1}|}\arrow[d, "|p_0|"]\\
{|\cD^1_{m_1}\times_{\cD_0^1}\cE^{1,2}_1\times_{\cD_0^2}\cD^2_{m_2}|}\arrow[r, "|p_1|"]&{|\cD^1_{m_1}|}
\end{tikzcd}
\]
is a pullback square, where $|-|$ denotes the functor of geometric realization of an $\infty$-category. To prove it we will make use of Quillen's theorem B in the form it is presented in \cite[Theorem 5.16.]{ayala2017fibrations}. Under suitable conditions it allows us to identify the fiber of $|p_0|$ over $\overline{\alpha}$ in $|\cD^1_{m_1}|$ with $|p_0/\overline{\alpha}|$. Observe that, since $p_0$ is a coCartesian fibration, we have an isomorphism 
\begin{align}\label{eq:six}
    |p_0/\overline{\alpha}|\cong|p_0^{-1}(\overline{\alpha})|
\end{align}
and that similar statement holds for $p_2$. To prove that the commutative square above is a pullback it suffices to prove that $|p_2|^{-1}(\overline{\beta})\cong|p_0|^{-1}(|p_1|(\overline{\beta}))$, however this easily follows from \cref{eq:six} and Quillen's theorem B since $p_2$ is a pullback of $p_0$.\par
It now remains to prove that the conditions of Quillen's theorem B are satisfied for $p_0$. More explicitly, we need to prove that for any morphism $a:\overline{\alpha}\rightarrow\overline{\alpha}'$ the induced functor $|p_0^{-1}(\overline{\alpha})|\rightarrow|p_0^{-1}(\overline{\alpha}')|$ is an isomorphism. An object $\overline{\alpha}$ of $\cD^1_{m_1}$ can be identified with a string of morphisms 
\[(d_0\xrightarrow{\alpha_1}d_1\xrightarrow{\alpha_2}...\xrightarrow{\alpha_m}d_m)\]
in $\cD^1$, the objects of the fiber $p_0^{-1}(\overline{\alpha})$ are given by strings 
\begin{align}\label{eq:seven}
(d'_0\xrightarrow{\beta_1}d'_1\xrightarrow{\beta_2}...\xrightarrow{\beta_l}d'_l\xrightarrow{c}d_0\xrightarrow{\alpha_1}d_1\xrightarrow{\alpha_2}...\xrightarrow{\alpha_m}d_m),
\end{align}
where $c$ is a morphism in $\cE^{0,1}$ and $d'_i$ are objects of $\cD^0$. Observe that any morphism in $\cD^1_{m_1}$ lying over an active morphism in $\Delta^\op$ acts as an identity on the fibers of $p_0$, the same is true for inert morphisms that preserve the minimal element. For an inert morphism $[k]\rightarrowtail[m]$ in $\Delta$ that sends $0$ to $s$ the corresponding action on the fiber sends the string (\ref{eq:seven}) above to 
\[(d'_0\xrightarrow{\beta_1}d'_1\xrightarrow{\beta_2}...\xrightarrow{\beta_l}d'_l\xrightarrow{f(\alpha_s)\circ f(\alpha_{s-1})\circ...\circ f(\alpha_1)\circ c}d_s\xrightarrow{\alpha_{s+1}}d_{s+1}\xrightarrow{\alpha_{s+2}}...\xrightarrow{\alpha_m}d_m).\]
However, since the images under $f$ of all morphisms in $\cD$ were assumed to be invertible in $\cE$, we see that this morphism also induces an isomorphism on fibers. Thus we have proved that the conditions of Quillen's theorem B apply, which suffices to prove the statement for $n=2$ by our previous considerations. Finally, the claim for $n>2$ follows by iterated application of the same argument.
\end{proof}
\begin{remark}\label{rem:fact_cat}
Given an object $(\cD\xrightarrow{f}\cE)\in\cat^\wrr_{/\cC}$ satisfying the conditions of \cref{lem:fact_cat}, the category $p_!(\cD\rightarrow\cE)$ over $\cC$ is obtained by taking the images under $f$ of the morphisms in the fibers $\cD_x$ and contracting them.
\end{remark}
\begin{construction}\label{constr:free_cat}
For a given Segal space $\cC$ denote by $\widetilde{\ccat}^\wrr_{/\cC}$ the category whose objects are morphisms $f:\bn\rightarrow\cC$, where $\bn$ is an interval of length $n$ with edges marked with $0$ or $1$ and $f$ sends all edges marked with $0$ to identities, and morphisms are morphisms in $\ccat_{/\cC}$ that preserve the markings. We will call a morphism active (resp. inert) if the underlying morphism in $\ccat_{/\cC}$ is active (resp. inert) and we will call an object elementary if the underlying object of $\ccat_{/\cC}$ is elementary, this gives $\widetilde{\ccat}^\wrr_{/\cC}$ the structure of an algebraic pattern. Observe that there is a natural faithful inclusion \[i:\widetilde{\ccat}^\wrr_{/\cC}\hookrightarrow\ccat^\wrr_{/\cC}.\]
\end{construction}
\begin{remark}
The Segal spaces for $\widetilde{\ccat}^\wrr_{/\cC}$ are easily seen to be given by a pair of morphisms $(\cD\rightarrow\cC,\cE\rightarrow\cC^\sim)$ in $\cat$ whose spaces of objects are isomorphic.
\end{remark}
\begin{lemma}\label{lem:free_cat}
The left Kan extension along $i$ induces a functor between the categories of respective Segal spaces.
\end{lemma}
\begin{proof}
First, observe that we can reduce to the case $\cC\cong[n]$ using the same argument as in \cref{lem:fact_cat}. For a given $x\in\ccat^\wrr_{/[n]}$ denote by $(i/x)^\act\hookrightarrow(i/x)$ the full subcategory of $(i/x)$ on active morphisms $i(y)\twoheadrightarrow x$ (for the duration of the proof we will use the words active morphism and inert morphism to refer to morphisms in $\ccat^\wrr_{/[n]}$ whose underlying morphisms in $\ccat_{/[n]}$ are active or inert, this is \textit{different} to active and inert morphisms in $\ccat^\wrr_{/[n]}$ itself). The active/inert factorization in $\ccat_{/[n]}$ shows that $(i/x)^\act$ is a cofinal subcategory. Observe that although $i$ does not have unique lifting of inert morphism, it does have unique lifting for those morphisms that preserve the markings. In particular, it has unique lifting for all morphisms in $(\ccat^\wrr_{/[n]})^{\el,\mc}_{x/}$, which is a coinitial subcategory in $(\ccat^\wrr_{/[n]})^\el_{x/}$. It follows from this and the properties of $\ccat_{/[n]}$ that we have an isomorphism 
\[(i/x)^\act\cong\underset{(x\rightarrowtail e)\in(\ccat^\wrr_{/[n]})^{\el,\mc}_{x/}}{\lim}(i/e)^\act.\]
Moreover, the Segal condition for $\widetilde{\ccat}^\wrr_{/[n]}$ implies that for a Segal space $\cF\in\seg_{\widetilde{\ccat}^\wrr_{/[n]}}(\cS)$ and an active morphism $i(y)\twoheadrightarrow x$ we have \[\cF(y)\cong\underset{(x\overset{i}{\rightarrowtail}e)\in(\ccat^\wrr_{/[n]})^{\el,\mc}_{x/}}{\lim}\cF(i_! y).\]
So it remains to prove that $(i/x)^\act$-colimits distribute over $(\ccat^\wrr_{/[n]})^{\el,\mc}_{x/}$-limits. To do this, first observe that $(i/\{j\})^\act\cong*$ where $\{j\}$ denotes any morphism of the form $([0]\xrightarrow{\{j\}}[n])$. It follows from this and the observations above  that
\[(i/x)^\act\cong\underset{(x\rightarrowtail e)\in(\ccat^\wrr_{/[n]})^{\el,\mc}_{x/}}{\lim}(i/e)^\act\cong\prod_{x\rightarrowtail e,e\in\{[1],\langle1\rangle\}}(i/e)^\act.\]
Now the claim follows just as at the end of the proof of \cref{prop:cell_def}.
\end{proof}
\begin{remark}\label{rem:free_cat}
Given an object $(\cD\rightarrow\cC,\cE\rightarrow\cC^\sim)$, the category $i_!(\cD,\cE)$ has the same space of objects as $\cD$ and $\cE$ and its morphisms are given by strings $(x_1\xrightarrow{f_1}x_2\xrightarrow{f_2}...\xrightarrow{f_n}x_{n+1})$ where each $f_i$ is a morphism either in $\cD$ or in $\cE$ modulo the equivalence relation that identifies \[(y_1\xrightarrow{g_1}...\xrightarrow{g_i}y_{i+1}\xrightarrow{g_{i+1}}y_{i+2}\xrightarrow{g_{i+2}}y_{i+3}\xrightarrow{g_{i+3}}...\xrightarrow{g_n}y_{n+1})\] with \[(y_1\xrightarrow{g_1}...\xrightarrow{g_i}y_{i+1}\xrightarrow{g_{i+2}\circ g_{i+1}}y_{i+3}\xrightarrow{g_{i+3}}...\xrightarrow{g_n}y_{n+1})\]
if both $g_{i+1}$ and $g_{i+2}$ belong to the same category.
\end{remark}
\begin{prop}\label{prop:bo}
Denote by $\cat^\eso_{/\cC}$ the full subcategory of $\cat_{/\cC}$ on those morphisms $\cD\xrightarrow{f}\cC$ that induce an epimorphism on objects and by $\cat^\bo_{/\cC}$ the full subcategory on morphisms that induce an isomorphism o objects. Then the natural inclusion $i:\cat^\eso_{/\cC}\hookrightarrow\cat^\bo_{/\cC}$ admits a left adjoint $L^\bo$. Moreover, if $\cD$ is a complete Segal space, then so is $L^\bo\cD$.
\end{prop}
\begin{proof}
Given a morphism $\cD\xrightarrow{f}\cC$ first consider the object $(\cD,C_\bullet(f_0))$ where $C_\bullet(f_0)$ denotes the Cech nerve of $f_0:\cD_0\rightarrow\cC_0$, i.e. it is a groupoid whose space of objects is given by $\cD_0$ and morphisms are $\cD_0\times_{\cC_0}\cD_0$. We then define
\[L^\bo\cD\bydef p_! i_! (\cD,C_\bullet(f_0)).\]
First, observe that $L^\bo\cD$ indeed belongs to $\cat_{/\cC}$ since by \cref{rem:free_cat} the images of morphisms in $C_\bullet(f_0)$ in $ i_! (\cD,C_\bullet(f_0))$ are invertible, and so this follows from \cref{lem:fact_cat}. Moreover, it even lies in $\cat^\bo_{/\cC}$ since
\[L^\bo\cD_0\cong\underset{\Delta^\op}{\colim}\overbrace{\cD_0\times_{\cC_0}...\times_{\cC_0}\cD_0}^\text{$n$ times}\cong\cC_0\]
because $f$ was assumed to be an effective epimorphism. To prove the required universal property observe that for any $\cE\in\cat^\bo_{/\cC}$ we have 
\[\mor_{\cat^\bo_{/\cC}}(L^\bo\cD,\cE)\cong\mor_{\widetilde{\cat}^\wrr_{/\cC}}((\cD,C_\bullet(f_0)),i^*p^*\cE)\cong\mor_{\widetilde{\cat}^\wrr_{/\cC}}((\cD,C_\bullet(f_0)),(\cE,\cE_0)).\]
The objects of the latter space are given by pairs of morphisms  $(\cD\xrightarrow{g}\cE,C_\bullet(f_0)\xrightarrow{h}\cE_0)$ over $\cC$. Observe that by definition $\cE_0\cong\cC_0$, so there is a unique morphism $h:C_\bullet(f_0)\rightarrow\cE_0$ which proves the first claim.\par
To prove the second claim observe using \cref{rem:fact_cat} and \cref{rem:free_cat} that the morphisms in $L^\bo\cD$ are given by strings $(d_1\xrightarrow{f_1}d_2\xrightarrow{f_2}...\xrightarrow{f_{n-1}}d_n)$ of not necessarily composable morphisms in $\cD$ such that their images in $\cC$ form a composable string of morphisms modulo the equivalence relation identifying \[(d_1\xrightarrow{f_1}d_2\xrightarrow{f_2}...\xrightarrow{f_i}d_{i+1}\xrightarrow{f_{i+1}}d_{i+2}\xrightarrow{f_{i+2}}d_{i+3}\xrightarrow{f_{i+3}}...\xrightarrow{f_{n-1}}d_n)\] with 
\[(d_1\xrightarrow{f_1}d_2\xrightarrow{f_2}...\xrightarrow{f_i}d_{i+1}\xrightarrow{f_{i+2}\circ f_{i+1}}d_{i+3}\xrightarrow{f_{i+3}}...\xrightarrow{f_{n-1}}d_n)\]
if $f_{i+1}$ and $f_{i+2}$ are actually composable in $\cD$. In particular, it follows easily that invertible morphisms in $L^\bo\cD$ are given by strings $(d\xrightarrow{f}e)$ where $f$ is an invertible morphism in $\cD$, and those are all identities by assumption.
\end{proof}
\begin{cor}\label{cor:adj_comp}
The natural inclusion $i:\cat^{\lcomp,\wrr}_{/\cC}\hookrightarrow\cat^{\comp,\wrr}_{/\cC}$ admits a left adjoint $L^\comp$.
\end{cor}
\begin{proof}
Let $(\cE\xrightarrow{f}\cD)$ be an element of $\cat^{\lcomp,\wrr}_{/\cC}$, so in particular $\cE$ is a complete Segal space. Let $(\cK\xrightarrow{g}\cL)$ be some element of $\cat^{\comp,\wrr}_{/\cC}$, then we see that for every morphism $(\cE,\cD)\xrightarrow{(u,v)}(\cK,\cL)$ the morphism $v$ factors as $\cD\xrightarrow{p}\widehat{\cD}\xrightarrow{\widetilde{v}}\cL$ since $\cL$ is complete. We will define $L^\comp(\cE,\cD)$ to be given by $(\widetilde{\cE}\xrightarrow{\widetilde{f}}\widehat{\cD})$, where 
\[\widetilde{\cE}\bydef L^\bo(p\circ f)\]
in the notation of \cref{prop:bo}. Then we have the following commutative diagram over $\cC$
\[
\begin{tikzcd}[row sep=huge, column sep=huge]
{}&\widetilde{\cE}\arrow[dr, "\widetilde{f}"]\arrow[ddl, dotted, near end, "\exists!" swap]\\
\cE\arrow[ur, "\eta"]\arrow[r, "f"]\arrow[d, "u"]&\cD\arrow[r, "p"]\arrow[d, "v"]&\widehat{\cD}\arrow[dl, "\widetilde{v}" swap]\\
\cK\arrow[r, "g"]&\cL
\end{tikzcd}.
\]
We need to prove that the space of dotted arrows making the diagram commute is contractible. Observe that from the definition of $L^\bo$ it follows that this space can be described as the space of morphisms 
\[C_\bullet(p_0\circ f_0)\xrightarrow{h}\cK_0\]
such that $g_0\circ h\cong v_0\circ f_0$. Observe that this space is contractible since $g_0$ is an isomorphism by the definition of $\cat^\wrr_{/\cC}$.
\end{proof}
\section{Completeness for monads and theories}\label{sect:six}
In this short final section we reap the fruits of our labours and prove the main results of this paper - \cref{thm:art} and \cref{thm:comp}. The former of these provides an isomorphism between the category of $\cC$-theories for a given category $\cC$ and the category of lax functors from $\cC$ to the flagged bicategory of correspondences with arities. This specializes to the main result of \cite{berger2012monads} in the case $\cC\cong*$.\par
The latter has to do with the notion of completeness. This notion is unique to higher category theory and was first introduced in a closely related context in \cite{chu2019homotopy}. More specifically, we introduce the notion of complete $\cC$-theory in \cref{def:thr} and \cref{thm:comp} states that for a certain class of theories (which we call good) described in \cref{def:good} we can introduce the functor of completion that is left adjoint to the inclusion of the category of complete theories into the category of good theories. In particular, it specializes to the completion functor of \cite{chu2019homotopy} for the theories associated to algebraic patterns.
\begin{defn}
Denote by $\corr^\art$ the following twofold Segal space: its objects are given by complete Segal spaces $\cC$ equipped with a full subcategory $i_\cC:\cE_\cC\hookrightarrow\mor_\cat(\cC,\cS)$. We define the category of morphisms from $\cC$ to $\cD$ to be the full subcategory of $\corr(\cC,\cD)$ on those correspondences $K:\cC\nrightarrow\cD$ that take $\cE_\cC$ to $\cE_\cD$ (under the equivalence $\corr(*,\cC)\cong\mor_\cat(\cC,\cS)$). We will call the objects of $\corr^\art$ the \textit{categories with arities} and we will say that the correspondences in $\corr^\art$ \textit{respect arities}. For a complete Segal space $\cK$ we will denote by $\mor^{\lax,\art}_\cat(\cK,\corr^\art)$ the category whose objects are the lax functors $F:\cK\rightsquigarrow\corr^\art$ and whose morphisms are natural transformations in $\mor^{\lax,!}_\cat(\cK,\corr^\art)$ for which we in addition require that any component $\alpha_{c,!}:F(k)\rightarrow G(k)$ for $c\in\cC$ respects the arities.
\end{defn}
\begin{defn}\label{def:thr}
For a complete Segal space $\cC$ denote by $\thr(\cC)$ the category whose objects $\cT$ are given by morphisms $(\cT_0\xrightarrow{t}\cT_1)$ over $\cC$ that belong to $\cat^{\lcomp,\overset{\eso}{\rightarrow}}_{/\cC}$ such that in addition every $\cT_{0,c}$ is a category with arities $\cE_c$, for every non-unital morphism $f:c\rightarrow c'$ in $\cC$ the correspondence $\cT_{0,f}:\cT_{0,c}\nrightarrow\cT_{0,c'}$ respects arities and for every $c$ the correspondence $t_c^* t_{c,!}:\cT_{0,c}\nrightarrow\cT_{0,c}$ also respects arities. We will refer to the objects of $\thr(\cC)$ as \textit{$\cC$-theories}. We will call a $\cC$-theory $\cT$ \textit{complete} if the underlying object of $\cat^{\lcomp,\overset{\eso}{\rightarrow}}_{/\cC}$ belongs to $\cat^{\comp,\wrr}_{/\cC}$ and denote the corresponding subcategory by $\thr^\comp(\cC)$. We will call a functor $\cF:\cT_1\rightarrow\cS$ a \textit{model} for the theory if for every $c\in\cC$ its restriction to $\cT_{0,c}$ belongs to $\cE_c$; we will denote the full subcategory of $\mor_\cat(\cT_1,\cS)$ on models by $\modl_\cT(\cS)$.
\end{defn}
\begin{remark}
For the theories $\cT$ of the type described in \cref{ex:preshv} one can more generally define for any presentable category $\cV$ a subcategory $\cE_\cV\bydef\mor_\cat(\cD,\cV)\overset{i_*}{\hookrightarrow}\mor_\cat(\cC_0,\cV)$ and further define the category $\modl_\cT(\cV)$ of $\cV$-models of $\cT$ to be the category of functors $\cC_1\rightarrow \cV$ whose restriction to $\cC_0$ belongs to $\cE_\cV$, however we will not need this construction.
\end{remark}
\begin{theorem}\label{thm:art}
For every complete Segal space $\cC$ there is an equivalence 
\[\mor^{\lax,\art}_\cat(\cC,\corr^\art)\cong\thr(\cC).\]
Moreover, under this equivalence the category of models for a theory $\cT$ is equivalent to $\con_{\corr^\art}^\lax(*,F_\cT)$ for the corresponding lax functor $F_\cT$, where $*$ denotes the terminal category viewed as a category with arities by declaring $\cE_*\bydef \cS$.
\end{theorem}
\begin{proof}
The first statement follows immediately from \cref{thm:lax_corr} after an elementary observation that for every functor $F:\cC\rightsquigarrow\corr$ with the corresponding object $(\cD\xrightarrow{f}\cE)\in\cat^\wrr_{/\cC}$ and for every object $c\in\cC$ the monad corresponding to $*\xrightarrow{\{c\}}\cC\overset{F}{\rightsquigarrow}\corr$ is given by $f_c^*f_{c,!}$.\par
To prove the the second claim, first observe that to every lax cone from $*$ to $F_\cT$ we can attach a functor $\cF:\cT_1\rightarrow\cS$ using \cref{thm:lax_colim}. Observe that it follows directly from the definitions that this cone lies in $\corr^\art$ if and only if the restrictions $\cF|_{\cT_{0,c}}$ belong to $\cE_c$ for all $c\in\cC$, which is exactly the definition of a model for $\cT$.
\end{proof}
\begin{cor}
For a complete Segal space $\cD$ with arities $\cE_\cD$ the category of monads in $\corr^\art$ with the underlying object $\cD$ is equivalent to the category of functors $F:\cD\rightarrow\cE$ that are epimorphisms on objects and such that $F^*F_!\cE\subset\cE$. 
\end{cor}
\begin{proof}
This follows immediately from \cref{cor:monad} and \cref{thm:art}.
\end{proof}
\begin{ex}\label{ex:preshv}
A category with arities can be associated to the data of a complete Segal space $\cC_0$, an isomorphism-on-objects functor $\cC_0\xrightarrow{f}\cC_1$ and a full subcategory $i:\cD\hookrightarrow\cC_0$. In this case $\cE_\cC\bydef\mor_\cat(\cD,\cS)\overset{i_*}{\hookrightarrow}\mor_\cat(\cC_0,\cS)$. The construction of \cite[Section 13]{chu2019homotopy} allows us to attach a theory to every cartesian monad $T$ on a presheaf category, namely the pair $(\cU(T)\xrightarrow{j}\cW(T))$ with the arities given by $\cI\hookrightarrow\cU(T)$ in the notation of \cite{chu2019homotopy}.
\end{ex}
\begin{defn}\label{def:good}
Given a $\cC$ theory $\cT$ denote by $u$ the natural morphism $(\cT_0\rightarrow L^\bo\cT_0)$. We will cal the theory \textit{good} if it satisfies the following list of conditions:
\begin{itemize}
    \item For every object $c\in\cC$ we have $u_{c,!}u^*_c\cong\id$.
    \item For every non-unital morphism $f:c_1\rightarrow c_2$ we have $\cT_{0,f}\cong u^*_{c_2}\circ L^\bo\cT_{0,f}\circ u_{c_1,!}$.
    \item For every object $c\in\cC$ denote by $L^\comp\cE_c$ the essential image of $\cE_c$ under $u_!$. Then $u_c^*$ takes $L^\comp\cE_c$ to $\cE_c$ and moreover the following commutative square
    \[
    \begin{tikzcd}[row sep=huge, column sep=huge]
    L^\comp\cE_c\arrow[r, hook,"L^\comp i_c"]\arrow[d, "u_c^*"]&\mor_\cat(L^\comp\cT_{0,c},\cS)\arrow[d, "u_c^*"]\\
    \cE_c\arrow[r, hook, "i_c"]&\mor_\cat(\cT_{0,c},\cS)
    \end{tikzcd}
    \]
    is a pullback square.
    \item For every object $c\in\cC$ the monad $\widetilde{t}^*_c\widetilde{t}_{c,!}$ respects $L^\comp\cE_c$, where $\widetilde{t}:L^\bo\cT_0\rightarrow\cT_1$ is the image of $t$ under $L^\bo$.
\end{itemize}
We will denote by $\thr^\good(\cC)$ the full subcategory on good $\cC$-theories. Observe that every complete theory is good.
\end{defn}
\begin{remark}
It follows from \cite[Corollary 13.4.]{chu2019homotopy} and \cite[Corollary 13.8.]{chu2019homotopy} that the theory of \cref{ex:preshv} is good.
\end{remark}
\begin{theorem}\label{thm:comp}
The natural inclusion $i_\comp:\thr^\comp(\cC)\hookrightarrow\thr^\good(\cC)$ admits a left adjoint $L^\comp$. Moreover, there is an isomorphism $\modl_\cT(\cS)\cong\modl_{L^\comp\cT}(\cS)$ for every good $\cC$-theory $\cT$.
\end{theorem}
\begin{proof}
We define the underlying $\cat^\wrr_{/\cC}$-object of $L^\comp\cT$ to be given by $L^\bo(\cT_0\xrightarrow{t}\cT_1)$. Observe that to provide the left adjoint it now suffices to endow $L^\bo(\cT_0\xrightarrow{t}\cT_1)$ with the structure of a theory and the natural morphism $u:(\cT_0\xrightarrow{t}\cT_1)\rightarrow L^\bo(\cT_0\xrightarrow{t}\cT_1)$ with the structure of a morphism in $\corr^\act$. For every $c\in\cC$ we will denote by $L^\comp\cE_c$ the essential image of $\cE_c$ under $u_{c,!}$. We need to show that this actually gives us the object of $\thr(\cC)$. The fact that $\widetilde{t}^*_c\widetilde{t}_{c,!}$ respects $L^\comp\cE_c$ follows directly from \cref{def:good}. Now we need to prove that for every $f:c\rightarrow c'$ the correspondence $L^\bo\cT_{0,f}:L^\bo:\cT_{0,c}\rightarrow L^\bo\cT_{0,c'}$ takes $L^\comp\cE_c$ to $L^\comp\cE_{c'}$. To do this it suffices to prove that the following diagram in $\corr$ commutes
\[
\begin{tikzcd}[row sep=huge, column sep=huge]
\cT_{0,c}\arrow[r, "\cT_{0,f}"]\arrow[d, "u_{c,!}"]&\cT_{0,c'}\arrow[d, "u_{c',!}"]\\
L^\bo\cT_{0,c}\arrow[r, "L^\bo\cT_{0,f}"]&L^\bo\cT_{0,c'}
\end{tikzcd}.
\]
Observe that, since $u_{c',!}u^*_{c'}\cong\id$, it suffices to prove that $\cT_{0,f}\cong u^*_{c'}\circ L^\bo\cT_{0,f}\circ u_{c,!}$, which is a part of \cref{def:good}. This proves that $L^\comp\cT$ indeed belongs to $\thr(\cC)$ thus finishing the proof of the first claim.\par
To prove the second claim first consider the following commutative diagram
\[
\begin{tikzcd}[row sep=huge, column sep=huge]
\modl_\cT(\cS)\arrow[r]\arrow[d]&\mor_\cat(\cT_1,\cS)\arrow[d]\\
\coprod_{c\in\cC}L^\comp\cE_c\arrow[r, "\coprod_{c\in\cC} L^\comp i_c"]\arrow[d, "\coprod_{c\in\cC} u^*_c"]&\coprod_{c\in\cC}\mor_\cat(L^\comp\cT_{0,c},\cS)\arrow[d, "\coprod_{c\in\cC}u^*_c"]\\
\coprod_{c\in\cC}\cE_c\arrow[r, "\coprod_{c\in\cC} i_c"]&\coprod_{c\in\cC}\mor_\cat(\cT_{0,c},\cS)
\end{tikzcd}.
\]
The outer square in this diagram is a pullback by the definition of $\modl_\cT(\cS)$, the lower square is a pullback by \cref{def:good}, so it follows that the top square is also a pullback, which is exactly what we needed to prove.
\end{proof}
\medskip
\bibliographystyle{alpha}
\bibliography{ref.bib}

\begin{thebibliography}{GHL20b}

\bibitem[AF17]{ayala2017fibrations}
David Ayala and John Francis.
\newblock Fibrations of $\infty$-categories.
\newblock {\em arXiv preprint arXiv:1702.02681}, 2017.

\bibitem[AF18]{ayala2018flagged}
David Ayala and John Francis.
\newblock Flagged higher categories.
\newblock {\em Topology and quantum theory in interaction, Contemp. Math},
  718:137--173, 2018.

\bibitem[B{\'e}n73]{benabou1973distributeurs}
Jean B{\'e}nabou.
\newblock Les distributeurs.
\newblock {\em Universit{\'e} Catholique de Louvain, Institut de
  Math{\'e}matique Pure et Appliqu{\'e}e, rapport}, 33, 1973.

\bibitem[BF06]{bunge2006singular}
Marta Bunge and Jonathon Funk.
\newblock {\em Singular coverings of toposes}.
\newblock Springer, 2006.

\bibitem[BMW12]{berger2012monads}
Clemens Berger, Paul-Andr{\'e} Mellies, and Mark Weber.
\newblock Monads with arities and their associated theories.
\newblock {\em Journal of Pure and Applied Algebra}, 216(8-9):2029--2048, 2012.

\bibitem[CH19]{chu2019homotopy}
Hongyi Chu and Rune Haugseng.
\newblock Homotopy-coherent algebra via segal conditions.
\newblock {\em arXiv preprint arXiv:1907.03977}, 2019.

\bibitem[GHL20a]{gagna2020fibrations}
Andrea Gagna, Yonatan Harpaz, and Edoardo Lanari.
\newblock Fibrations and lax limits of $(\infty, 2) $-categories.
\newblock {\em arXiv preprint arXiv:2012.04537}, 2020.

\bibitem[GHL20b]{gagna2020gray}
Andrea Gagna, Yonatan Harpaz, and Edoardo Lanari.
\newblock Gray tensor products and lax functors of ($\infty$, 2)-categories.
\newblock {\em arXiv preprint arXiv:2006.14495}, 2020.

\bibitem[GHN15]{gepner2015lax}
David Gepner, Rune Haugseng, and Thomas Nikolaus.
\newblock Lax colimits and free fibrations in $\infty $-categories.
\newblock {\em arXiv preprint arXiv:1501.02161}, 2015.

\bibitem[GR17]{gaitsgory2017study}
Dennis Gaitsgory and Nick Rozenblyum.
\newblock {\em A study in derived algebraic geometry}, volume~1.
\newblock American Mathematical Soc., 2017.

\bibitem[Hau15]{haugseng2015bimodules}
Rune Haugseng.
\newblock Bimodules and natural transformations for enriched
  $\infty$-categories.
\newblock {\em arXiv preprint arXiv:1506.07341}, 2015.

\bibitem[Hau18a]{haugseng2018iterated}
Rune Haugseng.
\newblock Iterated spans and classical topological field theories.
\newblock {\em Mathematische Zeitschrift}, 289(3):1427--1488, 2018.

\bibitem[Hau18b]{haugseng2018equivalence}
Rune Haugseng.
\newblock On the equivalence between $\theta_n$-spaces and iterated segal
  spaces.
\newblock {\em Proceedings of the American Mathematical Society},
  146(4):1401--1415, 2018.

\bibitem[Hau19]{haugseng2019segal}
Rune Haugseng.
\newblock Segal spaces, spans, and semicategories.
\newblock {\em arXiv preprint arXiv:1901.08264}, 2019.

\bibitem[Hau20]{haugseng2020fibrational}
Rune Haugseng.
\newblock A fibrational mate correspondence for $\infty $-categories.
\newblock {\em arXiv preprint arXiv:2011.08808}, 2020.

\bibitem[JT07]{joyal2007quasi}
Andr{\'e} Joyal and Myles Tierney.
\newblock Quasi-categories vs segal spaces.
\newblock {\em Contemporary Mathematics}, 431(277-326):10, 2007.

\bibitem[Koc72]{kock1972monads}
A~Kock.
\newblock Monads for which structures are adjoint to units. aarhus universitet
  math.
\newblock {\em Preprint Series}, 35, 1972.

\bibitem[Law63]{lawvere1963functorial}
F~William Lawvere.
\newblock Functorial semantics of algebraic theories.
\newblock {\em Proceedings of the National Academy of Sciences of the United
  States of America}, 50(5):869, 1963.

\bibitem[Lur]{luriehigher}
Jacob Lurie.
\newblock Higher algebra, september 2017.
\newblock {\em available at his webpage https://www. math. ias. edu/\~{}
  lurie}.

\bibitem[Lur09]{lurie2009higher}
Jacob Lurie.
\newblock {\em Higher topos theory}.
\newblock Princeton University Press, 2009.

\bibitem[Str72]{street1972two}
Ross Street.
\newblock Two constructions on lax functors.
\newblock {\em Cahiers de topologie et g{\'e}om{\'e}trie diff{\'e}rentielle
  cat{\'e}goriques}, 13(3):217--264, 1972.

\bibitem[Str76]{street1976limits}
Ross Street.
\newblock Limits indexed by category-valued 2-functors.
\newblock {\em Journal of Pure and Applied Algebra}, 8(2):149--181, 1976.

\bibitem[Str80]{street1980fibrations}
Ross Street.
\newblock Fibrations in bicategories.
\newblock {\em Cahiers de topologie et g{\'e}om{\'e}trie diff{\'e}rentielle
  cat{\'e}goriques}, 21(2):111--160, 1980.

\bibitem[Web04]{weber2004generic}
Mark Weber.
\newblock Generic morphisms, parametric representations and weakly cartesian
  monads.
\newblock {\em Theory Appl. Categ}, 13(14):191--234, 2004.

\bibitem[Z{\"o}b76]{zoberlein1976doctrines}
Volker Z{\"o}berlein.
\newblock Doctrines on 2-categories.
\newblock {\em Mathematische Zeitschrift}, 148(3):267--279, 1976.

\end{thebibliography}
\bigskip
Roman Kositsyn, \textsc{ Department of Mathematics, National Research University Higher School of Economics, Moscow}, \href{mailto:roman.kositsin@gmail.com}{roman.kositsin@gmail.com}

\end{document}